\documentclass[12pt]{amsart}

\usepackage{tikz-cd}

\usepackage{graphicx}
\usepackage{graphicx}
\usepackage{epsfig}
\usepackage{pstricks}
\usepackage{eucal}
\usepackage{tikz}
\usetikzlibrary{arrows}
\usetikzlibrary{calc,decorations.pathmorphing,patterns}
\usetikzlibrary{patterns,snakes}
\usetikzlibrary{decorations.markings}
\usetikzlibrary{decorations.pathreplacing}
\usetikzlibrary{positioning,automata}
\usepackage{framed}
\usepackage[vcentermath]{youngtab}
\usepackage{young}
\usepackage{ytableau}
\usepackage[nospace,noadjust]{cite}

\theoremstyle{definition}

\newtheorem{thm}{Theorem}[section]

\newtheorem{defn}[thm]{Definition}
\newtheorem{lemma}[thm]{Lemma}

\newtheorem{cor}[thm]{Corollary}
\newtheorem{remark}[thm]{Remark}
\newtheorem{example}[thm]{Example}

\usepackage{amsmath}
\usepackage{amsxtra}
\usepackage{amscd}
\usepackage{amsthm}
\usepackage{amsfonts}
\usepackage{amssymb}
\usepackage{eucal}

\newcommand{\g}{{\mathfrak{g}}}

\newcommand{\n}{{\mathfrak{n}}}

\newcommand{\supp}{{\rm{Supp}}}
\newcommand{\length}{{\rm{length}}}
\newcommand{\head}{{\rm{head}}}
\newcommand{\tail}{{\rm{tail}}}
\newcommand{\Bd}{{\rm{Bd}}}

\tikzset{->-/.style={decoration={
  markings,
  mark=at position #1 with {\arrow{>}}},postaction={decorate}}}

\usetikzlibrary{fit}
\makeatletter
\tikzset{
  fitting node/.style={
    inner sep=0pt,
    fill=none,
    draw=none,
    reset transform,
    fit={(\pgf@pathminx,\pgf@pathminy) (\pgf@pathmaxx,\pgf@pathmaxy)}
  },
  reset transform/.code={\pgftransformreset}
}
\makeatother

\numberwithin{equation}{section}
\def\set@mathmode@YT{%
\gdef\skipin@YT{$}%
\gdef\skipout@YT{$}
\def\smallfont@YT{\scriptstyle}%
}

\usepackage{pifont}

\newcommand{\xdasharrow}[2][->]{
\tikz[baseline=-\the\dimexpr\fontdimen22\textfont2\relax]{
\node[anchor=south,font=\scriptsize, inner ysep=1.5pt,outer xsep=2.2pt](x){#2};
\draw[shorten <=3.4pt,shorten >=3.4pt,dashed,#1](x.south west)--(x.south east);
}
}

\begin{document}

\title[$(q,t)$-characters of KR-modules of type $A_r$ as quantum cluster variables]{$(q,t)$-characters of Kirillov-Reshetikhin modules of type $A_r$ as quantum cluster variables}
\author{Bolor Turmunkh}
\address{Department of Mathematics, University of Illinois MC-382, Urbana, IL 61821, U.S.A. e-mail: turmunk2@illinois.edu}
\date{\today}
\maketitle

\begin{abstract}
Nakajima introduced \cite{nakajima2004quiver,nakajima2003t,nakajima2002t,nakajima2001quiver,nakajima2000t} a $t$-deformation of $q$-characters, $(q,t)$-characters for short, and their twisted multiplication through the geometry of quiver varieties. The Nakajima $(q,t)$-characters of Kirillov-Reshetikhin modules satisfy a $t$-deformed $T$-system \cite{nakajima2003t}. The $T$-system is a discrete dynamical system that can be interpreted as a mutation relation in a cluster algebra in two different ways, depending on the choice of direction of evolution. In this paper, we show that the Nakajima $t$-deformed $T$-system of type $A_r$ forms a quantum mutation relation in a quantization of exactly one of the cluster algebra structures attached to the $T$-system. 

\end{abstract}


\section{Introduction}

Let $\mathfrak{g}$ be a simple Lie algebra over $\mathbb{C}$ of rank $r$, and let $U_q(\widehat{\mathfrak{g}})$ be the corresponding untwisted quantum affine algebra. Let $I$ be the set $\left\{1,2,\ldots, r\right\}$, and let $q\in \mathbb{C}^*$ be not a root of unity. The category of finite-dimensional complex $U_q(\widehat{\mathfrak{g}})$-modules has been classified by Chari and Pressley \cite{chari1995quantum}. Simple objects in this category are parametrized by an $r$-tuple of polynomials of one variable with constant term $1$, called the Drinfeld polynomials \cite{chari1995quantum}. 

Given $i\in I$, $k\in \mathbb{Z}_{\geq 0}$, $j\in \mathbb{Z}$, let 
\begin{equation} \label{KRDrinfeld}
\mathbf{P}_{k,j}^{(i)} = \left(  (P_{k,j}^{(i)})_\alpha (u) \right)_{\alpha\in I} \,\, ,\text{ where }
 (P_{k,j}^{(i)})_\alpha(u) =\left\{ \begin{array}{cc} \prod_{s=1}^k (1-q^j q^{2s-2}u) &\text{if } \alpha=i\\
 1&\text{otherwise}\end{array}\right.  
\end{equation}

A finite-dimensional irreducible module with an $r$-tuple of Drinfeld polynomials given by $\mathbf{P}_{k,j}^{(i)}$ is called a Kirillov-Reshetikhin module (KR-module) and denoted $W_{k,j}^{(i)}$. KR-modules were introduced in \cite{kirillov1990representations}, and then further studied by Kuniba, Nakanishi, and Suzuki in \cite{kuniba1994functional}, Hatayama, Kuniba, Okado, Takagi, and Yamada in \cite{hatayama1998remarks} and Chari in \cite{chari2001fermionic}. One of the main tools used to study finite-dimensional $U_q(\widehat{\mathfrak{g}})$-modules is their \emph{$q$-characters}, which are the affine analogs of $U_q(\g)$-characters. The theory of $q$-characters was introduced by Knight \cite{knight1995spectra} and Frenkel-Reshetikhin \cite{frenkel1999q} for the Yangians and the quantum affine algebras respectively. One of the key properties concerning the KR-modules is that their characters and $q$-characters satisfy the functional relations called the $Q$-system and the $T$-system respectively. Nakajima proved the latter result in \cite{nakajima2003t} for types ADE using the $t$-analog of $q$-characters, $(q,t)$-characters for short, defined geometrically through quiver varieties.

The $T$-system of type $A_r$ is a recursion relation on commuting variables $\left\{T_{k}^{(i)}(a)\right\}$, for $a\in \mathbb{C}^*, k\in \mathbb{Z}_+$, and $i\in I$, defined as follows:
\begin{eqnarray}\label{eq:classicTsystemSpec}
T_{k}^{(i)}\left( a\right) T_{k}^{(i)}\left(aq^2 \right) = T_{k+1}^{(i)}(a)T_{k-1}^{(i)}(aq^2) + T_{k}^{(i+1)}(a)T_{k}^{(i-1)}(a)\,,
\end{eqnarray}
with the convention $T_{k}^{(0)}(a) = T_{k}^{(r+1)}(a)=1$. Without loss of generality, we can always assume $a\in \mathbb{C}^*$ is fixed. Then we need only keep track of the powers of $q$. Using a change of variables $T_k^{(i)}(aq^j) \rightarrow T_{k, k+j}^{(i)}$ and relabeling $l = k+j+1$, we arrive at another form of the $T$-system, also known as the \emph{octahedron recurrence}:
\begin{eqnarray} \label{eq:classicTsystem}
T_{k,l-1}^{(i)} T_{k,l+1}^{(i)} = T_{k+1,l}^{(i)} T_{k-1,l}^{(i)} + T_{k,l}^{(i-1)} T_{k,l}^{(i+1)} ,
\end{eqnarray}
for  $k\in \mathbb{Z}_+$, $i=1,2,\ldots, r$, and with the convention $T_{k,l}^{(0)} = T_{k,l}^{(r+1)} = 1$. The $T$-system of type $A_r$ in \eqref{eq:classicTsystemSpec} was originally discovered by Bazhanov-Reshetikhin in \cite{bazhanov1990restricted} as a functional relations among the transfer matrices of the generalized RSOS models, and later, generalized by Kuniba, Nakanishi, and Suzuki in \cite{kuniba1994functional} to all Dynkin types. The Nakajima $(q,t)$-characters of KR-modules satisfy a deformed $T$-system with a twisted multiplication on the variables $T_{k,l}^{(i)}$. The $t$-deformed $T$-system is a quadratic recursion relation on non-commutative variables that reduce to the classical $T$-system when $t=1$.

\emph{Cluster algebras} are commutative algebras generated by the union of commutative variables, called \emph{cluster variables}. The generators are related by rational transformations called \emph{mutations}, which are determined by an \emph{exchange matrix} \cite{fomin2002cluster}. Since all cluster variables are related to one another via mutations, it suffices to state a single cluster, called the \emph{fundamental cluster}, along with the exchange matrix, in order to define a cluster algebra. When the exchange matrix is invertible, there exists a canonical Poisson structure on the cluster variables \cite{gekhtman2003cluster}. A quantization of this canonical Poisson structure was introduced by Berenstein-Zelevinsky in \cite{berenstein2005quantum}, and is called a \emph{quantum cluster algebra}. Quantum cluster algebras are non-commutative algebras, and as such, their generators are not required to have any commutation relation at all. However, the variables within the same cluster satisfy a $t$-commutation relation.

Kedem \cite{kedem2008q} and Di Francesco-Kedem \cite{di2009positivity} realized $Q$ and $T$ systems as mutation relations in certain cluster algebras. The $Q$-system is a recursion relation on commuting variables $\left\{Q_{k}^{(i)}\right\}$ obtained from the $T$-system in the $l\rightarrow \infty$ limit. The cluster algebra formulation of the $Q$-system in \cite{kedem2008q} was used to obtain a unique quantization of the $Q$-system in \cite{di2012solution,di2011non}. The resulting quantum $Q$-system was shown to have deep connections with many areas such as the fusion product, defined in \cite{feigin1999generalized}, of KR-modules \cite{di2013quantum}, the quantum current subalgebra $U_q(\n_+[u,u^{-1}])$ in $U_q(\widehat{\mathfrak{sl}}_2)$ \cite{di2016quantum} , where $\n_+$ is the positive nilpotent subalgebra in $\mathfrak{sl}_2$, and a new set of $q$-difference operators \cite{di2015difference}, which are generalizations of the Macdonald raising operators in the limit $t\rightarrow \infty$, acting on the characters of KR-modules of type $A_r$.

The $T$-system equation can be interpreted as mutation relations in an infinite rank cluster algebra if we declare the $T_{k,l}^{(i)}$ variables to be invertible  \cite{hernandez2013cluster,di2009positivity}. The $T$-system in \eqref{eq:classicTsystem} is written such that the direction of mutation is in the direction of the $l$-parameter. That is, one can obtain $T_{k,l+1}^{(i)}$ using variables with lower value of $l$ only: 
\begin{eqnarray*}
T_{k,l+1}^{(i)} =  (T_{k,l-1}^{(i)})^{-1} \left(T_{k+1,l}^{(i)}T_{k-1,l}^{(i)} + T_{k,l}^{(i-1)}T_{k,l}^{(i+1)} \right).
\end{eqnarray*}
It is possible to rewrite the $T$-system in \eqref{eq:classicTsystem} such that the direction of mutation is in the $k$-parameter (see \cite{di2012solution}). Because the $T$-system cluster algebra is an infinite rank cluster algebra, these two choices of directions of mutations define $2$ distinct cluster algebra structures. Moreover, since the exchange matrices in both cases are of infinite rank, there are no canonical Poisson structures to these cluster algebras associated with the $T$-system.  Di Francesco and Kedem considered a quantization of the $A_1$ $T$-system in \cite{di2012solution}. Their choice corresponds to the quantization of the $T$-system cluster algebra with direction of mutation in the $k$-parameter. 

In light of the above results, it is natural to ask if the Nakajima deformed $T$-system forms a quantum mutation in a quantum cluster algebra, and if yes, then which one. In this paper, we provide answers to both questions. In particular, we show that the Nakajima deformed $T$-system \emph{does} form a quantum mutation, but \emph{only} when the direction of mutation is in the $l$-parameter. 

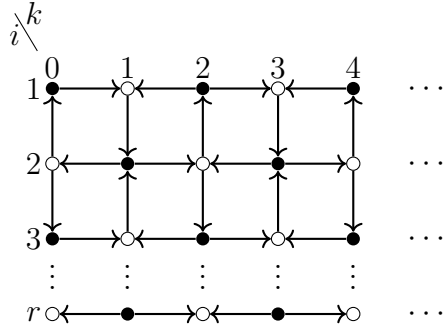
\begin{figure}[h]
\begin{tikzpicture}

	\tikzstyle{black} = [circle, minimum width=5pt, fill, inner sep=0pt]
	\tikzstyle{white} = [circle, minimum width=5pt,draw,inner sep=0pt]

\node at (-0.5,0)[right] {$k$};
\node at (0,-1) [above] {$0$};
\node at (1,-1) [above] {$1$};
\node at (2,-1) [above] {$2$};
\node at (3,-1) [above] {$3$};
\node at (4,-1) [above] {$4$};

\draw (-.5,0) -- (-.2,-.5);

\node at (-0.5,0) [left,below] {$i$};
\node at (0,-1) [left] {$1$};
\node at (0,-2) [left] {$2$};
\node at (0,-3) [left] {$3$};
\node at (0,-4) [left] {$r$};

\node[black] (10) at (0,-1) {};
\node[white] (11) at (1,-1) {};
\node[black] (12) at (2,-1) {};
\node[white] (13) at (3,-1) {};
\node[black] (14) at (4,-1) {};
\node at (5,-1) {$\cdots$};
\node at (5,-2) {$\cdots$};
\node at (5,-3) {$\cdots$};
\node at (5,-4) {$\cdots$};
\node[white] (20) at (0,-2) {};
\node[black] (21) at (1,-2) {};
\node[white] (22) at (2,-2) {};
\node[black] (23) at (3,-2) {};
\node[white] (24) at (4,-2) {};

\node[black] (30) at (0,-3) {};
\node[white] (31) at (1,-3) {};
\node[black] (32) at (2,-3) {};
\node[white] (33) at (3,-3) {};
\node[black] (34) at (4,-3) {};
\node at (0,-3.4) {$\vdots$};
\node at (1,-3.4) {$\vdots$};
\node at (2,-3.4) {$\vdots$};
\node at (3,-3.4) {$\vdots$};
\node at (4,-3.4) {$\vdots$};

\node[white] (r0) at (0,-4) {};
\node[black] (r1) at (1,-4) {};
\node[white] (r2) at (2,-4) {};
\node[black] (r3) at (3,-4) {};
\node[white] (r4) at (4,-4) {};

\path[->,thick]

(10) edge (11)
(11) edge (21)
(21) edge (20)
(20) edge (10)
(20) edge (30)
(30) edge (31)
(31) edge (21)

(12) edge (11)
(21) edge (22)
(22) edge (12)
(22) edge (32)
(32) edge (31)

(12) edge (13)
(13) edge (23)
(23) edge (22)
(33) edge (23)
(32) edge (33)

(14) edge (13)
(23) edge (24)
(24) edge (14)
(24) edge (34)
(34) edge (33)

(r1) edge (r0)
(r1) edge (r2)
(r3) edge (r2)
(r3) edge (r4)

;

\end{tikzpicture}
\caption{The quiver $\Gamma_B$}
\label{fig:Tsystem}
\end{figure}

We now summarize the main findings of this paper. We first define the cluster algebra formulation of the $T$-system in a similar way to that in \cite{di2009positivity}, with the difference being that the direction of evolution is in the $l$-direction. The parameters $i,k,l$ in \eqref{eq:classicTsystem} correspond to $\alpha, k, j$ parameters in \cite{di2009positivity} respectively. Let $B$ be the signed adjacency matrix of the quiver $\Gamma_B$ in Figure~\ref{fig:Tsystem}, which is the exchange matrix of the cluster algebra we are constructing. 

\begin{defn} \label{notation:mod}
For any $n\in \mathbb{Z}_+$, we define $(n)_2 := n\,{\rm{mod}}2$. 
\end{defn}
To each vertex $(k,i)$ in $\Gamma_B$, we associate the fundamental cluster variable $T_{k,(i+k+1)_2}^{(i)}$, which is given by the $q$-character of the corresponding KR-module. The fundamental cluster is given by the set:
\begin{equation} \label{eq:FundClustT}
\mathcal{C} := \left\{T_{k,(i+k+1)_2}^{(i)} \,\, \big| \,\, i\in I, \,\, k\in \mathbb{Z}_+\right\}\,\, ,
\end{equation}
with the boundary conditions $T_{k,l}^{(i)} = 1$ if either $k=0$, $i=0$, or $i=r+1$. With the exchange matrix $B$ and the fundamental cluster $\mathcal{C}$, mutation relations are given precisely by the $T$-system equations. 

We now state the main result of this paper, which answers the question stated above:
\begin{thm}\label{thm:main}
Let $(B, \mathcal{C})$ be as before, only now we substitute $T_{k,l}^{(i)}$ with their non-commutative versions given by $(q,t)$-characters of KR-modules. 

(I). The fundamental cluster variables $t$-commute with respect to the Nakajima's twisted multiplication (this condition ensures that it is possible to carry out the quantum mutation). Moreover, the commutation matrix $\Lambda = (\Lambda_{i,k,l}^{i',k',l'}),$ defined as:
\begin{equation} \label{LambdaMatrix}
T_{k,l}^{(i)} * T_{k',l'}^{(i')} = t^{\Lambda_{i,k,l}^{i',k',l'} } T_{k',l'}^{(i')} * T_{k,l}^{(i)} \text{ for all } T_{k',l'}^{(i')} , T_{k,l}^{(i)} \in \mathcal{C},
\end{equation}
is an integer-valued matrix.

(II). $(B,\Lambda)$ forms a compatible pair, i.e. 
\begin{equation} \label{Cond2}
\Lambda B = D,
\end{equation} 
where $D$ is a diagonal matrix with positive entries. 

(III). The quantum mutation is given by 
\begin{equation*} \label{Cond3}
T_{k,l-1}^{(i)}*T_{k,l+1}^{(i)} = t^{\frac{1}{2} \Lambda_{i,k,l-1}^{i,k-1,l} + \frac{1}{2}\Lambda_{i,k,l-1}^{i,k+1,l} - \frac{1}{2} \Lambda_{i,k-1,l}^{i,k+1,l} } T_{k-1,l}^{(i)}*T_{k+1,l}^{(i)}+t^{\frac{1}{2} \Lambda_{i,k,l-1}^{i-1,k,l} + \frac{1}{2}\Lambda_{i,k,l-1}^{i+1,k,l} - \frac{1}{2} \Lambda_{i-1,k,l}^{i+1,k,l} } T_{k,l}^{(i-1)}*T_{k,l}^{(i+1)}.
\end{equation*}
In other words, $(B,\mathcal{C}, \Lambda)$ defines a quantum cluster algebra. 
\end{thm}

There are $2$ main parts to this work. The bulk of the work is a combinatorial construction that proves $(q,t)$-characters in the fundamental cluster $t$-commute with one another under Nakajima's twisted multiplication. We use a slightly modified version of the tableaux-sum notation for $q$-characters introduced in \cite{nakajima2002t} and define the notion of a block-tableau, which plays an integral role in the proof. Once $t$-commutativity is established, the second half of this paper is concerned with the commutation coefficients of the fundamental cluster variables. In particular, we show that the commutation coefficients are compatible with the cluster algebra exchange matrix and the mutation relations in the language of Berenstein-Zelevinsky \cite{berenstein2005quantum}.  

The structure of the paper is as follows. In Chapter 2, we give the necessary background material. In Chapter 3, we introduce a modification of the tableaux-sum notation for the $q$-characters of Kirillov-Reshetikhin modules that appeared in \cite{nakajima2002t} and prove Theorem \ref{thm:main}, condition I. In Chapter 4, we show  conditions II and III in Theorem \ref{thm:main} are satisfied. In Chapter 5, we show that Nakajima's $t$-deformed $T$-system is not a quantum cluster algebra when the direction of mutation is in $k$ parameter using a simple counter example. And we conclude with Chapter 6.

\textbf{Acknowlodgements:} I would like to thank Rinat Kedem, Philippe Di Francesco, and Maarten Bergvelt for their valuable discussion. This work was partially funded by the NSF grant DMS-1404988.  

\section{Definitions} \label{defs}
\subsection{$(q,t)$-characters and the deformed $T$-system}
For the rest of this paper, $\g$ will always refer to the simple Lie algebra of type A and rank $r$, i.e. $\mathfrak{sl}_{r+1}$. Let $\left\{\alpha_i\right\}_{i \in I}$ and $\left\{ \omega_i\right\}_{i\in I}$ be the simple roots and the fundamental weights of $\g$. We denote $Q_+=\bigoplus_{i=1}^r \mathbb{Z}_{+} \alpha_i$ to be the positive root lattice and $P=\bigoplus_{i=1}^r \mathbb{Z} \omega_i$ to be the weight lattice. Let $\mathbf{R}$ and $\widehat{\mathbf{R}}$ be the Grothendieck rings of the categories of finite dimensional representations of $U_q(\g)$ and $U_q(\widehat{\g})$ respectively. 

Given $\lambda\in P$ and the irreducible highest weight module $V_\lambda = \oplus_{\alpha \in Q_+} V_{\lambda - \alpha}$ in $\mathbf{R}$, the character of $V_\lambda$ is a formal sum $\chi(V_\lambda) = \sum_{\alpha \in Q_+} \dim(V_{\lambda - \alpha}) e^{\lambda - \alpha}$. If we write $y_i = e^{\omega_i}$, the character is an injective homomorphism of commutative rings:
\begin{eqnarray*}
\chi :\mathbf{R} \rightarrow \mathbb{Z}[y_i, y_i^{-1}]_{i\in I}.
\end{eqnarray*}
Let us also define monomials:
\begin{equation} \label{eq:finiteA}
a_i = y_i^2 y_{i-1}^{-1} y_{i+1}^{-1},
\end{equation}
with the convention $y_i = 1$ if $i\notin I$. Notice that $a_i$ is identified with $e^{\alpha_i}$. 

The affine analog of the character map $\chi$ is the $q$-character map $\chi_q$, first introduced in \cite{knight1995spectra, frenkel1999q}, and it uniquely characterizes the isomorphism classes of irreducible modules in $\widehat{\mathbf{R}}$. The character of a finite dimensional $U_q(\g)$-module is the generating series of the weight space multiplicities. Similarly, the $q$-character of a finite dimensional $U_q(\widehat{\g})$-module is the generating series of the affine analog of the weight space multiplicities, called the $l$-weight spaces. The $q$-character map \cite{frenkel1999q} is an injective ring homomorphism $\chi_q$ that makes the following diagram commute:
\[
\begin{tikzcd}
\widehat{\mathbf{R}} \arrow{r}{\chi_q} \arrow[swap]{d}{\rm{res}} & \mathbb{Z}[Y_{i,j}^{\pm 1}]_{i\in I, j\in \mathbb{Z}} \arrow{d}{p} \\
\mathbf{R}  \arrow{r}{\chi} & \mathbb{Z}[y_i^{\pm 1}]_{i\in I}
\end{tikzcd}
\]
where $p(Y_{i,j})=y_i$ and $\rm{res}$ is the restriction map.  

Recall that the Drinfeld polynomial corresponding to $V\in \widehat{\mathbf{R}}$ is an $r$-tuple of polynomials of one variable. We identify every term of the form $(1-q^ju)$ appearing in the $i$th polynomial with $Y_{i,j}$. With this identification, the Drinfeld polynomial is mapped to a monomial of positive powers of $Y_{i,j}$'s only, i.e. a \emph{dominant monomial}, which corresponds to the $U_q(\g)$-highest weight vector in $V$. The Drinfeld polynomial corresponding to the KR-module $W_{k,j}^{(i)}$ defined in \eqref{KRDrinfeld} is identified with the dominant monomial $\mathbf{Y}_{k,j}^{(i)}$ as follows:
\begin{equation} \label{KRdom}
\mathbf{P}_{k,j}^{(i)} \mapsto \mathbf{Y}_{k,j}^{(i)} := Y_{i,j}Y_{i,j+2} \cdots Y_{i,j+2(k-2)}.
\end{equation}

Nakajima \cite{nakajima2000t,nakajima2002t,nakajima2004quiver} introduced a $t$-analog of the the Grothendieck ring  $\widehat{\mathbf{R}}_t = \mathbb{Z}[t,t^{-1}]\otimes_{\mathbb{Z}} \widehat{\mathbf{R}}$ and the $t$-analog of $q$-characters through the geometry of Quiver varieties. The $(q,t)$-map is a $\mathbb{Z}[t,t^{-1}]$-linear injective map:
\begin{eqnarray*}
\chi_{q,t}: \widehat{\mathbf{R}}_t \rightarrow \mathbb{Z}[t,t^{-1}]\otimes \mathbb{Z}[Y_{i,j},Y_{i,j}^{-1}]_{i\in I, j\in \mathbb{Z}},
\end{eqnarray*}
with the property that $\chi_{q,t=1}=\chi_q$. Frenkel and Mukhin \cite{frenkel2001combinatorics} gave a combinatorial algorithm that computes the $q$-characters of a class of modules that includes the fundamental modules \cite{frenkel2001combinatorics} and the KR-modules \cite{nakajima2002t}. Nakajima gave a similar algorithm that computes the $(q,t)$-characters of the fundamental modules. Although $\chi_{q,t}$ is not a ring homomorphism, Nakajima introduced a twisted multiplication on both the source and the target of $\chi_{q,t}$ that makes it into a homomorphism of non-commutative rings. Thus, starting with the $(q,t)$-characters of the fundamental modules, one can obtain the $(q,t)$-character of any twisted product of fundamental modules. 

The $(q,t)$-characters of KR-modules of type $A$ are identical to their $q$-characters \cite{nakajima2000t}. Therefore, we will omit the description of the algorithm for obtaining the $(q,t)$-characters of the fundamental modules and move straight to the description of the twisted multiplication. 

Let
\begin{eqnarray}\label{def:A}
A_{i,j} = Y_{i,j-1}Y_{i,j+1} Y_{i-1,j}^{-1}Y_{i+1,j}^{-1} \,\, ,
\end{eqnarray} 
with the convention that if $i \notin I$, then $Y_{i,j}=1$. This is the affine analog of the monomials in \eqref{eq:finiteA}. Given $V\in \widehat{\mathbf{R}}$, the $q$-character of $V$ \cite{frenkel2001combinatorics}, and in our case, the $(q,t)$-character of $V$ if $V$ is a KR-module, is of the form:
\begin{eqnarray*}
\chi_q(V) = m_+ + \sum_{M\in \mathcal{A}} m_+ M \,\, ,
\end{eqnarray*}
where $m_+$ is a dominant monomial and $\mathcal{A}$ is some subset of monomials in $\mathbb{Z}_+[A_{i,j}^{-1}]_{i\in I j\in \mathbb{Z}}$. 

\begin{defn} 
Let $\mathcal{M}$ be the set of monomials in $\mathbb{Z}[Y_{i,j}, Y_{i,j}^{-1}]_{i\in I, j\in \mathbb{Z}}$. Given $m,m'\in \mathcal{M}$, we say $m'$ is a descendant of $m$ if $m'$ is obtained by applying a set of $A_{i,j}^{-1}$ to $m$, i.e. $m'm^{-1}\in \mathbb{Z}[A_{i,j}^{-1}]_{i\in I, j\in \mathbb{Z}}$. 
\end{defn}
With this definition, given a $q$-character of an irreducible module, all its monomials are descendants of the dominant monomial. 
\begin{defn} \label{uv}
For $m_+,m\in \mathcal{M}$ such that $m_+$ dominant and $m$ descendant of $m_+$, define $u_{i,j}(m) \in \mathbb{Z}$ and $v_{i,j}(m,m_+)\in \mathbb{N}$ as follows: 
\begin{equation*}
m = \prod_{i,j} Y_{i,j}^{u_{i,j}(m)} =m_+ \prod_{i,j} A_{i,j}^{-v_{i,j}(m,m_+)}\,\, .
\end{equation*}
\end{defn}

The following definitions are due to Nakajima \cite{nakajima2003t}.

\begin{defn} \label{uinverse}
For $m\in \mathcal{M}$, let $\tilde{u}_{i,j}(m) \in \mathbb{R}$ ($i\in I, j\in \mathbb{Z}$) be a solution of the system
\begin{equation*} 
u_{i,j}(m) = \tilde{u}_{i,j-1}(m) + \tilde{u}_{i,j+1}(m) - \tilde{u}_{i-1,j}(m) - \tilde{u}_{i+1,j}(m) \,\, ,
\end{equation*}
such that $\tilde{u}_{i,j}(m)=0$ for $j$ sufficiently small. As usual, $\tilde{u}_{i,j}(m) = 0$ if $i\notin I$. 
\end{defn}
\begin{remark}
For any monomial $m\in \mathcal{M}$, it can be seen from Definition \ref{uv}, there are only finitely many non-zero $u_{i,j}(m)$. The condition $\tilde{u}_{i,j}(m)=0$ for $j$ sufficiently small ensures there is in fact a unique integral solution to the system, which can be verified through direct computation. However, there can be infinitely many non-zero $\tilde{u}_{i,j}(m)$. The system in Definition \ref{uinverse} looks cryptic. As we will never have to solve this system explicitly, the reader need only realize that there is a unique solution to the system at this point.  
\end{remark}

\begin{defn} \label{def:Nakajima}
Let $m_+^1,m_+^2 \in \mathcal{M}$ be dominant monomials and $m^1,m^2\in \mathcal{M}$ such that $m^i$ is a descendant of $m_+^i$. 
\begin{eqnarray*}
\epsilon(m_+^1, m_+^2) &:=& - \sum_{i,j} u_{i,j+1}(m_+^1) \tilde{u}_{i,j}(m_+^2) + \sum_{i,j} u_{i,j+1}(m_+^2) \tilde{u}_{i,j}(m_+^1) \,\, ,\\
d(m^1,m_+^1; m^2, m_+^2) &:=& \sum_{i,j} v_{i,j+1}(m^1,m_+^1)u_{i,j}(m^2) + u_{i,j+1}(m_+^1) v_{i,j}(m^2, m_+^2) \,\, ,\\
\gamma(m^1, m_+^1; m^2, m_+^2) &:=& d(m^1,m_+^1; m^2, m_+^2) - d(m^2,m_+^2; m^1, m_+^1)\,\, .
\end{eqnarray*}
\end{defn}
Notice that the above sums are all finite sums as there are only finitely many non-zero values of $u_{i,j}(m)$ and $v_{i,j}(m,m_+)$ for any monomials $m$ and $m_+$. Also, it follows immediately from the definition that both $\gamma$ and $\epsilon$ are anti-symmetric. 

\begin{remark} \label{rmk:simplegamma}
Given a monomial $m$ in a $(q,t)$-character, in our application of the $\gamma$ function, it will always be clear what the dominant monomial of $m$ is. Therefore, we simplify our notation and write
\begin{equation*}
\gamma(m^1,m^2) := \gamma(m^1,m_+^1; m^2, m^2_+)\,\,.
\end{equation*}
\end{remark}

\begin{defn} \label{twistedMult}
Let $V^1, V^2 \in \widehat{\mathbf{R}}$ and let $\chi_{q,t}^1$ and $\chi_{q,t}^2$ be their $(q,t)$-characters. Let $m_+^1, m_+^2$ be the dominant monomials and $m^1, m^2$ any monomials in $\chi_{q,t}^1$ and $\chi_{q,t}^2$ respectively. The twisted multiplication on monomials is defined as follows:
\begin{equation*} 
m^1 * m^2 := t^{\gamma(m^1,m^2) + \epsilon(m_+^1,m_+^2)} m^1m^2 \,\, ,
\end{equation*}
where $m^1m^2$ is the usual multiplication of monomials. Multiplication on $(q,t)$-characters is defined by linearly expanding the twisted multiplication on monomials. 
\end{defn}
We also define $*_\gamma$ multiplication as follows: 
\begin{equation} \label{gammaMult}
m^1*_\gamma m^2 := t^{-\epsilon(m_+^1, m_+^2)} m^1*m^2 = t^{\gamma(m^1,m^2)} m^1m^2 \,\, .
\end{equation}
We now state the result that this work is based on. 
\begin{thm}[Nakajima, \cite{nakajima2003t}] \label{thm:Nakajima}
Let $\chi_{k,j}^{(i)} := \chi_{q,t}(W_{k,j}^{(i)})$, where $W_{k,j}^{(i)}$ is the KR-module with dominant monomial $\mathbf{Y}_{k,j}^{(i)}$ (see \eqref{KRdom}). The following relations hold between the $(q,t)$-characters:
\begin{equation} \label{eq:tTsystem}
\chi_{k,j}^{(i)} *_\gamma \chi_{k,j+2}^{(i)} =\chi_{k+1,j}^{(i)} *_\gamma \chi_{k-1,j+2}^{(i)} + t^{-1} \chi_{k,j+1}^{(i-1)} *_\gamma \chi_{k,j+1}^{(i+1)} \,\, ,
\end{equation}
with the convention $\chi_{k,j}^{(i)} = 1$ if $k=0$ or $i\notin I$. 
\end{thm}
\begin{remark} \label{rmk:changeOfVar}
Equation \eqref{eq:tTsystem} can be transformed into a deformation of \eqref{eq:classicTsystem} by a change of variables $\chi_{k,j}^{(i)} \rightarrow T_{k, k+j}^{(i)}$ and relabeling $l = k+j+1$ with no changes to the coefficients. 
\end{remark}
The goal is to show that the deformed $T$-system in \eqref{eq:tTsystem} and the $(q,t)$-characters of KR-modules form a quantum cluster algebra. The exchange matrix is the adjacency matrix $B$ of the quiver $\Gamma_B$ in Figure~\ref{fig:Tsystem}. The fundamental cluster variables, given in \eqref{eq:FundClustT}, can be written in terms of $\chi_{k,j}^{(i)}$ variables as follows:
\begin{equation} \label{eq:FundClust}
\mathcal{C} = \left\{ \chi_{k,-k+(i+k+1)_2}^{(i)} \,\, \big| \,\,  i\in I, \,\, k \in \mathbb{Z}_+ \right\} \,.
\end{equation}
Notice that the fundamental cluster \eqref{eq:FundClust} is obtained from \eqref{eq:FundClustT} by the change of variables described in Remark \ref{rmk:changeOfVar}. 

\begin{remark}
Let $\chi_1$ and $\chi_2$ be $(q,t)$-characters with dominant monomials $m_+^1$ and $m_+^2$ respectively. Suppose $\chi_1$ and $\chi_2$ $t$-commute with respect to $*$. That is, there exists some $\alpha\in \mathbb{Z}$ such that
\begin{eqnarray}\label{CondIEquiv}
\chi_1*\chi_2 = t^{\alpha} \chi_2*\chi_1 \,\, ,
\end{eqnarray}
which is equivalent to
\begin{eqnarray} 
\chi_1*_\gamma \chi_2 = t^{\alpha - 2\epsilon(m_+^1,m_+^2)} \chi_2*_\gamma \chi_1\,\, .
\end{eqnarray}
Notice that $\gamma(m_+^1,m_+^2,) =0$ since $v_{i,j}(m_+,m_+)=0$ for all $i,j$ and any dominant monomial $m_+$.  That is, the dominant monomials in $ \chi_1 *_\gamma \chi_2$ and $\chi_2 *_\gamma \chi_1$ have coefficient $1$. Therefore, \eqref{CondIEquiv} holds if and only if $\alpha= 2\epsilon(m_+^1, m_+^2)$, and  $\chi_1 *_\gamma \chi_2 = \chi_2 *_\gamma \chi_1$. 
\end{remark}
  
\begin{cor}\label{gammaComm}
The fundamental cluster variables $t$-commute with respect to $*$, i.e. Condition I holds, if and only if they commute with respect to $*_\gamma$. Moreover, if the fundamental cluster variables $t$-commute, then the commutation matrix $\Lambda$, defined in \eqref{LambdaMatrix}, is given by 
\begin{eqnarray*}
\Lambda_{i,j,k}^{i',j',k'} = 2\epsilon( \mathbf{Y}_{k,j}^{(i)}, \mathbf{Y}_{k',j'}^{(i')} ) \,\, ,
\end{eqnarray*}
where $\mathbf{Y}_{k,j}^{(i)}$ is the dominant monomial of $\chi_{q,t}(W_{k,j}^{(i)})$ defined in \eqref{KRdom}. 
\end{cor}

\section{Proof of Condition I}

\newcommand\X{0.5}
\newcommand\D{2}
\newcommand*{\QEDA}{\hfill\ensuremath{\blacksquare}}%
\newcommand*{\QEDB}{\hfill\ensuremath{\square}}%

The main tool we use in this section is the tableaux-sum expression for the $q$-characters of KR-modules, introduced by Nakajima in \cite{nakajima2002t}, and the mapping $m:T\mapsto m_T$ from the space of allowed tableaux, called KR-tableaux in the text and to be defined in Section \ref{tableaux-sum}, to the space of monomials $\mathcal{M}$. Let $\mathcal{B}_{k,j}^{(i)}$ be the set of all KR-tableaux parametrizing the monomials of $\chi_{q}(W_{k,j}^{(i)})$.  

The idea of the proof of Condition I is as follows. We define a division of the set $\mathcal{B}_{k,j}^{(i)} \times\mathcal{B}_{k',j'}^{(i')}$ into $3$ disjoint subsets $\mathcal{P}_0, \mathcal{P}_1, \mathcal{P}_{-1}$ and an automorphism $\sigma$ on this set such that 
\begin{itemize}
\item $\sigma$ fixes the elements of $\mathcal{P}_0$ and is an involution between $\mathcal{P}_1$ and $\mathcal{P}_{-1}$
\item $m_{(C,T)} = m_{\sigma(C,T)}$, where $m_{(C,T)} := m_Cm_T$. 
\item $\gamma(m_C,m_T) =: \gamma(C,T) = -\gamma(\sigma(C,T))$.  
\end{itemize}
The main result of this section is:
\begin{thm} \label{thm:main1}
Let $(C,T)\in \mathcal{P}_0$, i.e. $(C,T)$ is a fixed point of $\sigma$. Then $\gamma(C,T)=0$. 
\end{thm}

The proof of Theorem \ref{thm:main1} is given in Section \ref{proofOfMain1}. Let us show how Theorem \ref{thm:main1} implies Condition I. 

\begin{cor}
The fundamental cluster variables $t$-commute with respect to the twisted multiplication $*$ defined in \eqref{twistedMult}. 
\end{cor}
\begin{proof}
By Corollary \ref{gammaComm}, Condition I is equivalent to the statement that elements in $\mathcal{C}$ (see \eqref{eq:FundClust}) commute with respect to $*_\gamma$. That is, given any $\chi_{k,j}^{(i)},\chi_{k',j'}^{(i')}\in \mathcal{C}$, we want to show
\begin{equation} \label{Lambda}
\chi_{k,j}^{(i)}* _\gamma \chi_{k',j'}^{(i')} = \chi_{k',j'}^{(i')} *_\gamma \chi_{k,j}^{(i)}
\end{equation}
Let's write
\begin{eqnarray*}
\chi_{k,j}^{(i)} = \sum_{C\in \mathcal{B}_{k,j}^{(i)}} m_C \text{ and }\chi_{k',j'}^{(i')} = \sum_{T\in \mathcal{B}_{k',j'}^{(i')}} m_T \,\, ,
\end{eqnarray*}
where $\mathcal{B}_{k,j}^{(i)}$ and $\mathcal{B}_{k',j'}^{(i')}$ are the sets of allowed KR-tableaux. Denote $\mathcal{B} = \mathcal{B}_{k,j}^{(i)} \times \mathcal{B}_{k',j'}^{(i')}$. Then,
\begin{eqnarray*}
\chi_{k,j}^{(i)} *_\gamma\chi_{k',j'}^{(i')} &=& \sum_{(C,T) \in \mathcal{B}} t^{\gamma(C,T)} m_Cm_T \\
&=& \underbrace{     \sum_{(C,T) \in \mathcal{P}_0\cap \mathcal{B}}  t^{\gamma(C,T)} m_{(C,T)} }_\text{$\gamma(C,T)=0$ by Theorem \ref{thm:main1}} + \sum_{(C,T) \in \mathcal{P}_1 \cap \mathcal{B}}  \left(t^{\gamma(C,T)} m_{(C,T)} + t^{\gamma(\sigma(C,T))} m_{\sigma(C,T)}\right)  \\
&=& \sum_{(C,T)\in \mathcal{P}_0 \cap \mathcal{B} } m_{T}m_{C} +  \sum_{ (C,T)\in \mathcal{P}_1 \cap \mathcal{B}} \left( t^{-\gamma(T,C)}  m_{T}m_{C} + t^{ \gamma(T,C)} m_{T}m_{C} \right)\\
&=&\chi_{k',j'}^{(i')}*_\gamma \chi_{k,j}^{(i)}
\end{eqnarray*}
where we used $\gamma(\sigma(C,T)) = -\gamma(C,T) = \gamma(T,C)$ and $m_{(C,T)} = m_Cm_T = m_{(T,C)}$. 
\end{proof}

\subsection{Tableaux-sum notation} \label{tableaux-sum}
The tableaux-sum notation we are using is a slight adjustment of the notation introduced in \cite{nakajima2002t}. Let us begin with the tableaux-sum description of the ordinary character $\chi(V)$ for $V\in \mathbf{R}$ as a motivation to the later definitions of the tableaux-sum notation for $q$-characters.

Let $V_\lambda$ be the highest weight irreducible $U_q(\mathfrak{sl}_{r+1})$-module with highest weight $\lambda = \sum_{i=1}^r \lambda_i \omega_i$. Then there exists a basis of $V_\lambda$ parametrized by semi-standard Young tableaux of shape $\Lambda = (\Lambda_1, \ldots, \Lambda_r)$, where $\Lambda_j = \sum_{i=j}^r \lambda_i$. Let $S(\Lambda)$ be the set of all semi-standard Young tableaux on the letters $\left\{1,2, \ldots, r+1\right\}$ of shape $\Lambda$. We define a map
\begin{eqnarray} \label{map:mon}
\begin{array}{cccc}m: &S(\Lambda) &\rightarrow & \multicolumn{1}{l}{\mathbb{Z}[y_i^{\pm 1}]_{i\in I}}\\
&&&\\
&T&\mapsto& m_T = \prod_{i\in I} y_i^{ \#T(i) - \# T(i+1)} \end{array}
\end{eqnarray}
where $\#T(i)$ is the number of times $i$ appears in $T$. With this mapping, we obtain the tableaux-sum expression for the character of $V_\lambda$. More precisely, 
\begin{eqnarray*}
\chi(V_\lambda) = \sum_{T \in S(\Lambda)} m_T\,\, .
\end{eqnarray*}

\begin{remark} \label{rectangularT}
The $U_q(\g)$-highest weight of the KR-module $W_{k,j}^{(i)}$ is given by $k\omega_i$. Therefore, the tableaux that parametrize the character of KR-modules, considered as $U_q(\g)$-modules, are rectangular of length $i$ and width $k$.  
\end{remark}

\begin{example}\label{ex:finite}
Consider $\g = \mathfrak{sl}_4$ and $V=V_{\omega_3}$. Then, $\Lambda = (1,1,1)$. 
\ytableausetup{mathmode,boxsize=1em}

\vspace{0.1in}
\begin{center}
\begin{tikzpicture}

\node at (-2,2) {$\chi(V_{\omega_3})$};
\node at (-1,2) {$=$};

\node (1') at (0,2) {$y_3$};

\node (1) at (0,0) {
\begin{ytableau}
1\\
2\\
3
\end{ytableau}
};

\node (2') at (2,2) { $y_2y_3^{-1}$};

\node (2) at (2,0) {
\begin{ytableau}
1\\
2\\
4
\end{ytableau}
};

\node (3') at (4,2) { $y_1y_2^{-1}$};

\node (3) at (4,0) {
\begin{ytableau}
1\\
3\\
4
\end{ytableau}
};

\node (4') at (6,2) { $y_1^{-1}$};

\node (4) at (6,0) {
\begin{ytableau}
2\\
3\\
4
\end{ytableau}
};

\node at (1,2) {$+$};
\node at (3,2) {$+$};
\node at (5,2) {$+$};

\path[->,thick]

(1) edge node[above] {$a_3^{-1}$}  (2)
(2) edge node[above] {$a_2^{-1}$}  (3)
(3) edge node[above] {$a_1^{-1}$} (4)
(1') edge[dotted] (1)
(2') edge[dotted] (2)
(3') edge[dotted] (3)
(4') edge[dotted] (4)
;

\end{tikzpicture}
\end{center}
\vspace{0.1in}
Recall that the monomials $a_i$ from Definition \ref{finiteA}, are identified with $e^{\alpha_i}$. Therefore, multiplying a monomial by $a_i^{-1}$ is equivalent to applying a lowering operator with weight $-\alpha_i$. The action of $a_i^{-1}$ on the tableaux is given by changing a box with $i$ to $i+1$. \qed
\end{example}

We now describe the tableaux-sum notation for $q$-characters. Recall that the restriction of $U_q(\widehat{\g})$-modules to $U_q(\g)$ corresponds to the map $Y_{i,j} \rightarrow y_i$ on the monomials, and the finite weight of $Y_{i,j}$ and $y_i$ are the same, given by $\omega_i$. Therefore, the tableau representation of $Y_{i,j}$ is almost identical to that of $y_i$, except there are infinitely many $Y_{i,j}$'s corresponding to $j\in \mathbb{Z}$. This infinite property is represented by an extra vertical coordinate added to the usual tableau data. 

\begin{defn}
\begin{enumerate}
\item Let $T$ be a diagram consisting of a single column of length $i$ equipped with an additional datum $j\in \mathbb{Z}$. To each box in $T$, we associate an \emph{index} as follows: the top box gets an index $\frac{1-i -j}{2}$ and the indices of the lower boxes increase by one starting from the top box. We call the set of indices of $T$ the \emph{support} of $T$, denoted $\rm{Supp}(T)$. Then,
\begin{eqnarray*}
\rm{Supp}(T) = \left\{ \frac{1-i-j}{2}, \ldots, \frac{1 + i - j}{2} \right\}
\end{eqnarray*}
The diagrams that we will consider in this paper will have $i$ and $j$ such that $(1-i-j)$ is always divisible by $2$. That is, $\supp(T)\subset \mathbb{Z}$. We say $T$ is a \emph{column diagram} of shape $(i,j)$.  
\item The \emph{head} of a column diagram $T$, denoted $\head(T)$, is the index of the first box.
\item The \emph{tail} of a column diagram $T$, denoted $\tail(T)$, is the index of the last box.
\item The \emph{length} of a column diagram $T$ is given by $i=\tail(T)-\head(T)+1$. For convenience, we denote it $\length(T)$. 
\item A \emph{column tableau} $T$ is a column diagram $T$ of some shape $(i,j)$ decorated with letters $\left\{1,2,\ldots, r+1\right\}$, i.e. we equip $T$ with an arbitrary map
 \begin{equation*}
 \supp(T) \rightarrow \left\{1,2,\ldots, r+1\right\} \,,
 \end{equation*}
 where $r$ is the rank of the Lie algebra $\g$. Equivalently, a column tableau is a column diagram with integers between $1$ and $r+1$ filled in each box.
 
The image of the map at $p\in \supp(T)$, i.e. the integer in the box with index $p$, is denoted $T[p]$ and is called the \textsl{value} of $T$ at $p$. If $p<\head(T)$, we set $T[p]=0$ and if $p>\tail(T)$, we set $T[p]=\infty$. With this redefinition, we can consider $T$ to be defined for all $p\in \mathbb{Z}$ and $\supp(T)$ is where the value of $T$ is nonzero and finite. 
\item A \emph{strip} in $T$ between $p_0$ and $p_1$, denoted $T[p_0,p_1]$, is the tableau given by the piece of $T$ between indices $p_0$ and $p_1$ with end points included. 
\item A \emph{general tableau} is obtained by stacking column tableaux horizontally with the indices of the boxes, defined in (1), aligned. 
\end{enumerate}
\end{defn}

Let $T=(T_1,\ldots, T_k)$ be a general tableau. Let $\supp(T) = \cup \supp(T_l)$. We identify $T$ with a monomial $m_T\in \mathcal{M}$ given by 
\begin{eqnarray} \label{map:qmon}
m_T= \prod_{p\in \rm{Supp}(T)} \prod_{i=1}^r Y_{i,i-2p-1}^{\#(T[p]=i) - \#(T[p+1]=i+1)} \,, 
\end{eqnarray}
where $\#(T[p]=i)$ is the number of times $i$ appears in $T$ at index $p$. This is the affine analog of the map \eqref{map:mon}. 
\begin{remark}
If we drop the indices of the tableaux and collapse all the columns until the heads of every column are on the same level, then we obtain Young tableaux, which gives the classical tableaux-sum notation for the highest weight modules. In particular, if we collapse the general tableau, then the map \eqref{map:qmon} reduces to the map \eqref{map:mon}. Diagramaticaly, we can add a third column to the commutative diagram 
\[
\begin{tikzcd}
\widehat{\mathbf{R}} \arrow{r}{\chi_q} \arrow[swap]{d}{\rm{res}} & \mathbb{Z}[Y_{i,j}^{\pm 1}]_{i\in I, j\in \mathbb{Z}} \arrow{d}{p}  &\left\{\text{General Tableaux}\right\} \arrow{l}[swap]{\eqref{map:qmon}} \arrow{d}{collapse}\\
\mathbf{R}  \arrow{r}{\chi} & \mathbb{Z}[y_i^{\pm 1}]_{i\in I}& \left\{\text{Tableaux}\right\} \arrow{l}[swap]{\eqref{map:mon}}
\end{tikzcd}
\]

\end{remark}

\begin{defn}
Let $T$ be a column tableau of shape $(i,j)$. We define $T_{dom}$ (dominant) to be the column tableau of the same underlying diagram given by
\begin{eqnarray*}
T_{dom}[\head(T)] = 1 &\text{ and }&T_{dom}[p+1] = T_{dom}[p]+1\,\,,\,\,\,\, \text{ for } p\in \supp(T)
\end{eqnarray*}
\end{defn}
\begin{remark}
Let $T$ be of shape $(i,j)$. Then the map \eqref{map:qmon} identifies $m_{T_{dom}}=Y_{i,j}$. In other words, $T_{dom}$ corresponds to the dominant monomial of all column tableaux $T$ of shape $(i,j)$. More generally, the tableaux corresponding to the descendants of any dominant monomial (not necessarily column) will have the same underlying diagram as their dominant monomial, i.e. the lengths and the indices of each column will be the same.  
\end{remark}

Let us illustrate all these definitions with an example.
\begin{example}
Let $r\geq 5$. Let $T=(T_1,T_2)$ be the general tableau consisting of columns of shapes $\left\{(3,2), (2,1)\right\}$. The first column diagram $T_1$ is of shape $(3,2)$, which means the length of $T_1$ is $3$ and $h(T_1)=\frac{1-i_1-j_1}{2}=-2$. Similarly, the second column diagram $T_2$ is of length $2$ and $h(T_2)=-1$. An example of such general tableau is given below, together with the corresponding collapsed tableau and its dominant tableau.  
\begin{center}
\begin{tikzpicture}
\node at (-3, -2) {General Tableaux};
\node at (0,-0.5) {diagram of $T$};
\node (1) at (0,-2) {$
\begin{ytableau}
\none[\scriptscriptstyle{-2}]&\\
\none[\scriptscriptstyle{-1}] & &\\
\none[\scriptscriptstyle{0}]& &\\
\end{ytableau} $};

\node at (3,-0.5) {tableau $T$};
\node (2) at (3,-2) {$
\begin{ytableau}
\none[\scriptscriptstyle{-2}]&1\\
\none[\scriptscriptstyle{-1}] &3 &4\\
\none[\scriptscriptstyle{0}]&4 &5\\
\end{ytableau} $};

\node at (6,-0.5) {$T_{dom}$};
\node (3) at (6,-2) {$
\begin{ytableau}
\none[\scriptscriptstyle{-2}]&1\\
\none[\scriptscriptstyle{-1}] &2 &1\\
\none[\scriptscriptstyle{0}]&3 &2\\
\end{ytableau} $};

\node at (9, -0.5) {$m_{T_{dom}}$};
\node at (9,-2) {$Y_{3,2}Y_{2,1}$};

\node at (-3, -4) {Collapsed Tableaux};

\node (4) at (0,-4) {$
\begin{ytableau}
\none[]&&\\
\none[] & &\\
\none[]& \\
\end{ytableau} $};

\node (2) at (3,-4) {$
\begin{ytableau}
\none&1&4\\
\none&3 &5\\
\none&4\\
\end{ytableau} $};

\node (2) at (6,-4) {$
\begin{ytableau}
\none&1&1\\
\none&2 &2\\
\none&3\\
\end{ytableau} $};

\node at (9,-4) {$y_3y_2$};


\end{tikzpicture}
\end{center}

\qed
\end{example}

\begin{example}
Consider $\g=\mathfrak{sl}_4$. Let $V$ be the fundamental $U_q(\widehat{\mathfrak{sl}_4})$-module with dominant monomial $Y_{3,0}$. The $U_q(\mathfrak{sl}_4)$-highest weight of this module is $\omega_3$. This is the affine analog of Example \ref{ex:finite}. We give the $q$-character of $V$ below:
\vspace{0.1in}
\begin{center}
\begin{tikzpicture}

\node at (-2,2) {$\chi_q(V)$};
\node at (-1,2) {$=$};

\node (1') at (0,2) {$Y_{3,0}$};

\node (1) at (0,0) {
\begin{ytableau}
\none[{\color{blue}\scriptstyle{-1}}]&1\\
\none[{\color{blue}\scriptstyle{0}}]&2\\
\none[{\color{blue}\scriptstyle{1}}]&{\color{red}3}
\end{ytableau}
};

\node (2') at (3,2) { $Y_{2,1}Y_{3,2}^{-1}$};

\node (2) at (3,0) {
\begin{ytableau}
\none[{\color{blue}\scriptstyle{-1}}]&1\\
\none[{\color{blue}\scriptstyle{0}}]&{\color{red}2}\\
\none[{\color{blue}\scriptstyle{1}}]&4
\end{ytableau}
};

\node (3') at (6,2) { $Y_{1,2}Y_{2,3}^{-1}$};

\node (3) at (6,0) {
\begin{ytableau}
\none[{\color{blue}\scriptstyle{-1}}]&{\color{red}1}\\
\none[{\color{blue}\scriptstyle{0}}]&3\\
\none[{\color{blue}\scriptstyle{1}}]&4
\end{ytableau}
};

\node (4') at (9,2) { $Y_{1,4}^{-1}$};

\node (4) at (9,0) {
\begin{ytableau}
\none[{\color{blue}\scriptstyle{-1}}]&2\\
\none[{\color{blue}\scriptstyle{0}}]&3\\
\none[{\color{blue}\scriptstyle{1}}]&4
\end{ytableau}
};

\node at (1.5,2) {$+$};
\node at (4.5,2) {$+$};
\node at (7.5,2) {$+$};

\path[->,thick]

(1) edge node[above] {$A_{{\color{red}3},{\color{red}3}-2{\color{blue}\cdot 1} }^{-1}$}  (2)
(2) edge node[above] {$A_{{\color{red}2}, {\color{red}2} - 2\cdot {\color{blue} 0}}^{-1}$}  (3)
(3) edge node[above] {$A_{{\color{red}1}, {\color{red} 1} - 2 \cdot ( {\color{blue} -1} )}^{-1}$} (4)
(1') edge[dotted] (1)
(2') edge[dotted] (2)
(3') edge[dotted] (3)
(4') edge[dotted] (4)
;

\end{tikzpicture}
\end{center}
\vspace{0.1in}

\qed

\end{example}
Let $T=(T_1,\ldots,T_k)$ be a general tableaux. The monomial $m_T$ in \eqref{map:qmon} can also be written as:
\begin{eqnarray} \label{eq:aAction}
m_T = m_{T_{dom}} \prod_{l=1}^k \prod_{p\in \supp(T_l)} \prod_{i=(T_l)_{dom}[p]}^{T_l[p]-1} A_{i,i-2p}^{-1}\,.
\end{eqnarray}

We have presented so far the modified crystal basis for $q$-character monomials in \cite{nakajima2002t}. From here on, we define new concepts. 

Consider KR-modules and the associated general tableaux. Consider the KR-module with dominant monomial $\mathbf{Y}_{k,j}^{(i)}=Y_{i,j}Y_{i,j+2}\cdots Y_{i,j+2k-2}$. The diagram associated with this dominant monomial and all its descendants consists of $k$ columns stacked as follows:
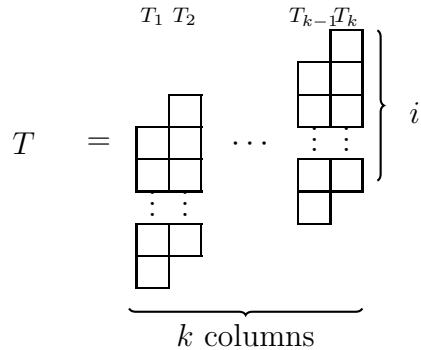
\vspace{0.2in}
\begin{center}
\begin{tikzpicture}
\node at (0,0) {$T$};
\node at (1,0) {$=$};
\node at (3,0) {$
\begin{ytableau}
\none[\scriptstyle{T_1}]&\none[\scriptstyle{T_2}]&\none&\none&\none&\none[\scriptstyle{T_{k-1}}]&\none[\scriptstyle{T_k}]\\
\none	&\none&\none	&\none&\none&\none&\\
\none	&\none&\none 	&\none &\none&&\\
\none	&&\none&	\none	&\none&&\\
		&&\none&	\none[\cdots]&\none&\none[\vdots]&\none[\vdots]	\\
		&&\none&	\none &\none	&&\\
\none[\vdots]&\none[\vdots]&\none &\none& \none &\\
		& &\none&\none\\
		&\none\\
\end{ytableau}$};

\draw [
	thick,
	decoration = {
		brace,
		mirror,
		raise=0.5cm
		},
		decorate
	] (1.4,-1.7) -- (4.5,-1.7)
node [pos=0.5, anchor=north,yshift=-0.55cm] {$k$ columns};

\draw [
	thick,
	decoration = {
		brace,
		raise=0cm
		},
		decorate
	] (4.7,1.5) -- (4.7,-0.5)
node [pos=0.5, anchor=north,xshift=0.5cm,yshift=0.2cm] {$i$ };

\end{tikzpicture}
\end{center}
where each column $T_l$, $1\leq l\leq k$, is of length $i$ with appropriate indices determined by $j_l=j+2l-2$. We call a diagram of this shape the staircase diagram of shape $(i,j,k)$. Notice that the corresponding collapsed diagram is rectangular, as expected by Remark \ref{rectangularT}. 
\begin{defn} \label{defn:KRtableau}
Define $\mathcal{B}_{k,j}^{(i)}$ to be the set of general tableaux given by staircase diagrams $T=(T_1,\ldots, T_k)$ of shape $(i,j,k)$ along with decorations by $\left\{1,2,\ldots, r+1\right\}$ such that
\begin{enumerate}
\item the values of the columns strictly increase from top to bottom, i.e. $T_i[p] < T_i[p+1]$ for all $i,p$. (see Figure~\ref{fig:KRarrows}, left)
\item the values of the diagonals weakly increase from left to right, i.e. $T_i[p] \leq T_{i+1}[p-1]$ for all $i,p$. (see Figure~\ref{fig:KRarrows}, right)
\end{enumerate}

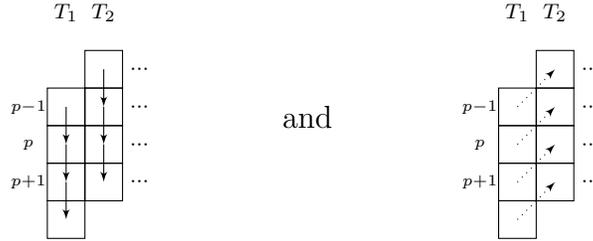
\begin{figure}[h]
\begin{center}
\begin{tikzpicture}

\node at (0,0) {\begin{tikzpicture}
\node at (0.5*\X, \X*6) {$\scriptstyle{T_{1}}$};
\draw (0,0) rectangle (\X*1,\X*1) node[fitting node] (00) {};
\draw (0, \X*1) rectangle (\X*1, \X*2) node[fitting node] (01) {};
\draw (0, \X*2) rectangle (\X*1, \X*3) node[fitting node] (02) {};
\draw (0, \X*3) rectangle (\X*1, \X* 4) node[fitting node] (03) {};
\node at ($(03.center)-(\X,0)$) {$\scriptscriptstyle p-1$};
\node at ($(02.center)-(\X,0)$) {$\scriptscriptstyle p$};
\node at ($(01.center)-(\X,0)$) {$\scriptscriptstyle p+1$};

\node at (1.5*\X, \X*6) {$\scriptstyle{T_2}$};
\draw (\X*1,\X*1) rectangle (\X*2,\X*2) node[fitting node] (11) {};
\draw (\X*1,\X*2) rectangle (\X*2,\X*3) node[fitting node] (12) {};
\draw (\X*1,\X*3) rectangle (\X*2,\X*4) node[fitting node] (13) {};
\draw (\X*1,\X*4) rectangle (\X*2,\X*5) node[fitting node] (14) {};
\node at ($(11.center)+(\X,0)$) {$ \scriptstyle \cdots$};
\node at ($(12.center)+(\X,0)$) {$ \scriptstyle \cdots$};
\node at ($(13.center)+(\X,0)$) {$ \scriptstyle \cdots$};
\node at ($(14.center)+(\X,0)$) {$ \scriptstyle \cdots$};

\draw[arrows={- latex'}] (03.center) -- (02.center);
\draw[arrows={- latex'}] (02.center) -- (01.center);
\draw[arrows={- latex'}] (01.center) -- (00.center);

\draw[arrows={- latex'}] (14.center) -- (13.center);
\draw[arrows={- latex'}] (13.center) -- (12.center);
\draw[arrows={- latex'}] (12.center) -- (11.center);
\end{tikzpicture}};

\node at (3,0) {and};

\node at (6,0) {\begin{tikzpicture}
\node at (0.5*\X, \X*6) {$\scriptstyle{T_{1}}$};
\draw (0,0) rectangle (\X*1,\X*1) node[fitting node] (00) {};
\draw (0, \X*1) rectangle (\X*1, \X*2) node[fitting node] (01) {};
\draw (0, \X*2) rectangle (\X*1, \X*3) node[fitting node] (02) {};
\draw (0, \X*3) rectangle (\X*1, \X* 4) node[fitting node] (03) {};
\node at ($(03.center)-(\X,0)$) {$\scriptscriptstyle p-1$};
\node at ($(02.center)-(\X,0)$) {$\scriptscriptstyle p$};
\node at ($(01.center)-(\X,0)$) {$\scriptscriptstyle p+1$};

\node at (1.5*\X, \X*6) {$\scriptstyle{T_2}$};
\draw (\X*1,\X*1) rectangle (\X*2,\X*2) node[fitting node] (11) {};
\draw (\X*1,\X*2) rectangle (\X*2,\X*3) node[fitting node] (12) {};
\draw (\X*1,\X*3) rectangle (\X*2,\X*4) node[fitting node] (13) {};
\draw (\X*1,\X*4) rectangle (\X*2,\X*5) node[fitting node] (14) {};
\node at ($(11.center)+(\X,0)$) {$ \scriptstyle \cdots$};
\node at ($(12.center)+(\X,0)$) {$ \scriptstyle \cdots$};
\node at ($(13.center)+(\X,0)$) {$ \scriptstyle \cdots$};
\node at ($(14.center)+(\X,0)$) {$ \scriptstyle \cdots$};

\draw[arrows={- latex'},dotted] (03.center) -- (14.center);
\draw[arrows={- latex'},dotted] (02.center) -- (13.center);
\draw[arrows={- latex'},dotted] (01.center) -- (12.center);
\draw[arrows={- latex'},dotted] (00.center) -- (11.center);

\end{tikzpicture}};

\end{tikzpicture}
\end{center}
\caption{Solid and dotted arrows indicate strict and weak inequalities in the direction of the arrows respectively, i.e. $a\rightarrow b$ is equivalent to $a<b$, and $a\dashrightarrow b$ is equivalent to $a\leq b$.}
\label{fig:KRarrows}
\end{figure}

We call a tableau that belongs to $\mathcal{B}_{k,j}^{(i)}$ for some $(i,j,k)$ a KR-tableau. 
\end{defn}

\begin{thm}
The $q$-character of the KR-module $W_{k,j}^{(i)}$ is parametrized by the KR-tableaux in $\mathcal{B}_{k,j}^{(i)}$. That is, we have:
\begin{eqnarray*}
\chi_q(W_{k,j}^{(i)}) = \sum_{T\in \mathcal{B}_{k,j}^{(i)}} m_T\,\, ,
\end{eqnarray*}
where $m_T$ is the monomial associated to the tableau $T$ through the map \eqref{map:qmon}. 
\end{thm}
\begin{proof}
Notice that if we collapse a KR-tableau, we obtain a rectangular semi-standard Young tableau. Since ${\rm{res}}(W_{k,j}^{(i)}) = V_{k\omega_i}$, and $\chi(V_{k\omega_i})$ is given by rectangular semi-standard Young tableaux, the result follows. 
\end{proof}

\subsection{Fundamental cluster diagrams} \label{fundClusterDiagrams}
Recall the fundamental cluster $\mathcal{C}$ defined in \eqref{eq:FundClust}. The column diagrams in $\mathcal{C}$ are shown in Figure~\ref{fig:fundTab}(a). Given $\chi_{k,j(k,i)}^{(i)}\in \mathcal{C}$, the corresponding staircase diagram of shape $(i,j,k)$ is constructed as follows: start with the column diagram $\chi_{1, j(1,i)}^{(i)}$ (Figure~\ref{fig:fundTab}(a)), call it the central column and start adding columns of equal length alternatingly to both sides of the central tableau. When $i$ is odd, we start adding on the left and when $i$ is even, we start adding on the right (see Figure~\ref{fig:fundTab}(b) for an example).

\begin{figure}[h]
\begin{tikzpicture}
\node at (0,-3) {\begin{minipage}{0.4\textwidth}(a) Diagrams corresonding to the Fundamental modules ($k=1$) in $\mathcal{C}$\end{minipage}};
\node at (0,0) {$
\begin{ytableau}
\none[i=]&\none& \none[1]&\none&\none[2]&\none&\none[3]&\none&\none[4]&\none&\none[5]&\none&\none[6]&\none&\none[7]&\none&\none[\cdots]\\
\none&\none&		\none	&\none&\none&\none&\none&\none&\none&\none&\none&\none&\none&\none&\none\\
\none[\scriptscriptstyle-3]&\none&	\none	&\none&\none&\none&\none&\none&\none&\none&\none&\none&\none&\none&\\
\none[\scriptscriptstyle-2]&\none&	\none	&\none&\none&\none&\none&\none&\none&\none&		&\none&		&\none&\\
\none[\scriptscriptstyle-1]&\none&	\none	&\none&\none&\none&	    &\none&	 &\none&		&\none&		&\none&\\
\none[\scriptscriptstyle0]&\none&			&\none&		&\none&	    &\none&	 &\none&		&\none&		&\none&\\
\none[\scriptscriptstyle1]&\none&	\none	&\none&		&\none&	    &\none&	 &\none&		&\none&		&\none&\\
\none[\scriptscriptstyle2]&\none&	\none	&\none&\none&\none&\none&\none&	 &\none&		&\none&		&\none&\\
\none[\scriptscriptstyle3]&\none&	\none	&\none&\none&\none&\none&\none&\none&\none&\none&\none&		&\none&\\
\end{ytableau}
$};
\node at (9,-3) {\begin{minipage}{0.4\textwidth}(b) $i=3$ and $k=1,2,3,4$. The central and the last column that is added are colored. 
\end{minipage}};
\node at (9,0) {$
\begin{ytableau}
\none[k=1]&\none&\none&\none&\none[k=2]&\none&\none&\none&\none[k=3]&\none&\none&\none&\none[k=4]&\none\\
\none&\none&\none&\none&\none&\none&\none&\none&\none&\none&\none&\none&\none&\none\\
\none&\none&\none&\none&\none&\none&\none&\none&*(blue!50)&\none&\none&\none&\none&\\
*(yellow)&\none&\none&\none&*(yellow)&\none&\none&*(yellow)&*(blue!50)&\none&\none&\none&*(yellow)&\\
*(yellow)&\none&\none&*(blue!50)&*(yellow)&\none&&*(yellow)&*(blue!50)&\none&\none&&*(yellow)&\\
*(yellow)&\none&\none&*(blue!50)&*(yellow)&\none&&*(yellow)&\none&\none&*(blue!50)&&*(yellow)&\none\\
\none&\none&\none&*(blue!50)&\none&\none&&\none&\none&\none&*(blue!50)&&\none&\none\\
\none&\none&\none&\none&\none&\none&\none&\none&\none&\none&*(blue!50)&\none&\none&\none\\
\none&\none&\none&\none&\none&\none&\none&\none&\none&\none&\none&\none&\none&\none
\end{ytableau}
$};
\end{tikzpicture}
\caption{Diagrams corresponding to the fundamental cluster variables}
\label{fig:fundTab}
\end{figure}
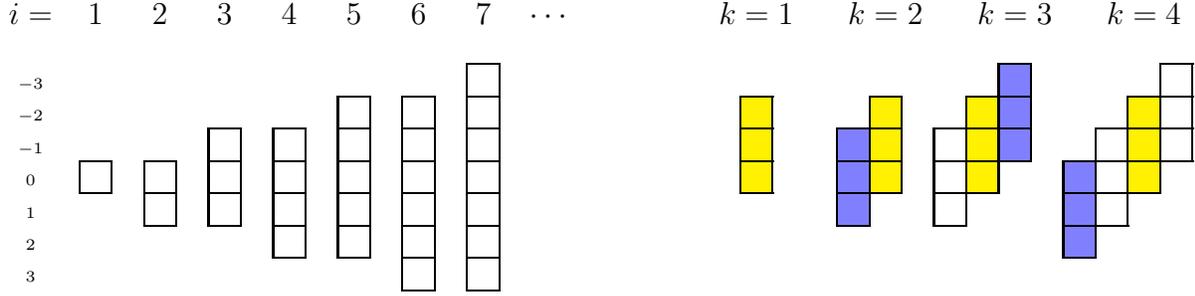

\subsection{Twisted multiplication of tableaux} \label{multTableaux}
We will now describe the twisted multiplication $*_\gamma$ from \eqref{gammaMult} on the tableaux. Let $C,T$ be KR-tableaux. In order to simplify our notation of Definition \ref{def:Nakajima}, we denote
\begin{eqnarray}  \label{simpleuv}
\begin{array}{ccc}
u_{i,j}(C):=u_{i,j}(m_C)\,; & v_{i,j}(C):=v_{i,j}(m_C, m_{C_{dom}})\,; &\gamma(C,T):=\gamma(m_C, m_T)  \,.
 \end{array}
\end{eqnarray}

\begin{remark} \label{rmk:gammaAdditive}
In order to compute $\gamma(C,T)$ for any general KR-tableaux $C$ and $T$, it suffices to compute the value of $\gamma$ between column tableaux only. Indeed, suppose $T=(T_l)$. Then,
\begin{eqnarray*}
m_T &=& \prod_{l} \left(\prod_{i,j} Y_{i,j}^{u_{i,j}(T_l)}\right) = \prod_{i,j} Y_{i,j}^{\sum_{l} u_{i,j}(T_l)} \,\, ,\\
m_T &=& \prod_{l} \left( m_{{T_l}_{dom}} \prod_{i,j} A_{i,j}^{-v_{i,j}(T_l)}  \right) = m_{T_{dom}} \prod_{i,j} A_{i,j}^{-\sum_{l} v_{i,j}(T_l)}\,\, .
\end{eqnarray*}
That is, both $u_{i,j}$ and $v_{i,j}$ are additive. Therefore, for any general tableau $C=(C_k)$, since $\gamma(C,T)$ is a linear expression in $u_{i,j}$ and $v_{i,j}$, we have
\begin{eqnarray*}
\gamma(C,T) = \sum_{k,l} \gamma(C_k,T_l) \,.
\end{eqnarray*}
\end{remark}

\begin{defn}
Let $(C,T)$ be a pair of column KR-tableaux. A \textsl{block} in $(C,T)$ at index $p$ is a pair of boxes given by $(C[p], T[p])$. We say there is an $L^\pm$-block, an $N^\pm$-block, or a $U$-block at $p$ in $(C,T)$ if the following inequality conditions hold:
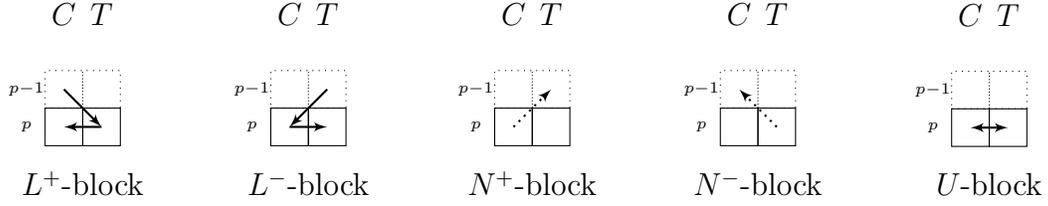
\begin{center}
\begin{tikzpicture}
\node at (0,0) {
\begin{tikzpicture}
\node at (0.5*\X,2.5*\X+.5) {$C$};
\node at (1.5*\X,2.5*\X+.5) {$T$};
\node at (-.5*\X,1.5*\X) {$\scriptscriptstyle p-1$};
\node at (-.5*\X,.5*\X) {$\scriptscriptstyle p$};

\draw (0*\X,0*\X) rectangle (1*\X,1*\X) node[fitting node] (00) {};
\draw[dotted] (0*\X,1*\X) rectangle (1*\X,2*\X) node[fitting node] (01) {};
\draw (1*\X,0*\X) rectangle (2*\X,1*\X) node[fitting node] (10) {};
\draw[dotted] (1*\X,1*\X) rectangle (2*\X,2*\X) node[fitting node] (11) {};

\draw[arrows={- latex'},thick] (01.center) -- (10.center);
\draw[arrows={- latex'},thick] (10.center) -- (00.center);

\node at (1*\X,-1*\X) { $L^+$-block};
\end{tikzpicture}
};

\node at (3,0) {
\begin{tikzpicture}
\node at (0.5*\X,2.5*\X+.5) {$C$};
\node at (1.5*\X,2.5*\X+.5) {$T$};
\node at (-.5*\X,1.5*\X) {$\scriptscriptstyle p-1$};
\node at (-.5*\X,.5*\X) {$\scriptscriptstyle p$};

\draw (0*\X,0*\X) rectangle (1*\X,1*\X) node[fitting node] (00) {};
\draw[dotted] (0*\X,1*\X) rectangle (1*\X,2*\X) node[fitting node] (01) {};
\draw (1*\X,0*\X) rectangle (2*\X,1*\X) node[fitting node] (10) {};
\draw[dotted] (1*\X,1*\X) rectangle (2*\X,2*\X) node[fitting node] (11) {};

\draw[arrows={latex'-},thick] (00.center) -- (11.center);
\draw[arrows={latex' -},thick] (10.center) -- (00.center);

\node at (1*\X,-1*\X) { $L^-$-block};
\end{tikzpicture}
};

\node at (6,0) {
\begin{tikzpicture}
\node at (0.5*\X,2.5*\X+.5) {$C$};
\node at (1.5*\X,2.5*\X+.5) {$T$};
\node at (-.5*\X,1.5*\X) {$\scriptscriptstyle p-1$};
\node at (-.5*\X,.5*\X) {$\scriptscriptstyle p$};

\draw (0*\X,0*\X) rectangle (1*\X,1*\X) node[fitting node] (00) {};
\draw[dotted] (0*\X,1*\X) rectangle (1*\X,2*\X) node[fitting node] (01) {};
\draw (1*\X,0*\X) rectangle (2*\X,1*\X) node[fitting node] (10) {};
\draw[dotted] (1*\X,1*\X) rectangle (2*\X,2*\X) node[fitting node] (11) {};

\draw[arrows={- latex'},thick,dotted] (00.center) -- (11.center);

\node at (1*\X,-1*\X) { $N^+$-block};
\end{tikzpicture}
};

\node at (9,0) {
\begin{tikzpicture}
\node at (0.5*\X,2.5*\X+.5) {$C$};
\node at (1.5*\X,2.5*\X+.5) {$T$};
\node at (-.5*\X,1.5*\X) {$\scriptscriptstyle p-1$};
\node at (-.5*\X,.5*\X) {$\scriptscriptstyle p$};

\draw (0*\X,0*\X) rectangle (1*\X,1*\X) node[fitting node] (00) {};
\draw[dotted] (0*\X,1*\X) rectangle (1*\X,2*\X) node[fitting node] (01) {};
\draw (1*\X,0*\X) rectangle (2*\X,1*\X) node[fitting node] (10) {};
\draw[dotted] (1*\X,1*\X) rectangle (2*\X,2*\X) node[fitting node] (11) {};

\draw[arrows={- latex'},thick,dotted] (10.center) -- (01.center);

\node at (1*\X,-1*\X) { $N^-$-block};
\end{tikzpicture}
};
\node at (12,0) {
\begin{tikzpicture}
\node at (0.5*\X,2.5*\X+.5) {$C$};
\node at (1.5*\X,2.5*\X+.5) {$T$};
\node at (-.5*\X,1.5*\X) {$\scriptscriptstyle p-1$};
\node at (-.5*\X,.5*\X) {$\scriptscriptstyle p$};

\draw (0*\X,0*\X) rectangle (1*\X,1*\X) node[fitting node] (00) {};
\draw[dotted] (0*\X,1*\X) rectangle (1*\X,2*\X) node[fitting node] (01) {};
\draw (1*\X,0*\X) rectangle (2*\X,1*\X) node[fitting node] (10) {};
\draw[dotted] (1*\X,1*\X) rectangle (2*\X,2*\X) node[fitting node] (11) {};

\draw[arrows={latex' - latex'},thick]  (00.center) -- (10.center);

\node at (1*\X,-1*\X) { $U$-block};
\end{tikzpicture}
};
\end{tikzpicture}
\end{center}
The arrow notation is the same as in Figure~\ref{fig:KRarrows} and two-sided arrow is equivalent to equality. 
\end{defn}
It is important to stress that a block consists of $2$ boxes at index $p$, even though there may be an inequality requirement coming from boxes with index $p-1$, such as the case of $L$ and $N$ blocks.  
\begin{defn}
If there is either an $L$-block or a $U$-block at $p$ in $(C,T)$, for convenience we say there is an $LU$-block at $p$. Any other combinations are allowed (e.g. $L^+UN^-$-block).  
\end{defn}

\begin{remark} \label{rmk:obs1}
As we are dealing with KR-tableaux, inequality conditions in Figure~\ref{fig:KRarrows} are always present. Due to this restriction, any block in a pair of KR-tableaux $(C,T)$ is precisely one of the 5 types: $L^\pm, N^\pm$ or $U$. Moreover, the strictly increasing columns condition puts additional restrictions on the order with which the blocks can appear. For example, (a) an $L^\pm$-block is never followed by an $N^\pm$-block; (b) an $N^\pm$-block is never followed by an $N^\mp$-block, and (c) a $U$-block is never followed by an $N$-block. The reason can be seen by simply composing the arrows. For example,
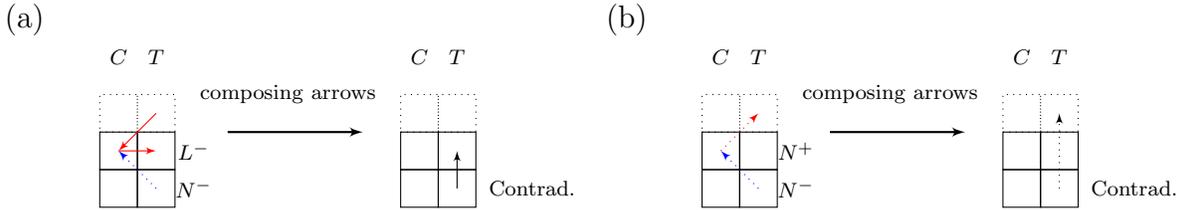
\begin{center}
\begin{tikzpicture}

\node (1) at (0,0) {
\begin{tikzpicture}
\node at (-1,3) {(a)};

\node at (0.5*\X, \X*5) {$\scriptstyle{C}$};

\draw (0, \X*1) rectangle (\X*1, \X*2) node[fitting node] (01) {};
\draw (0, \X*2) rectangle (\X*1, \X*3) node[fitting node] (02) {};
\draw[dotted] (0, \X*3) rectangle (\X*1, \X* 4) node[fitting node] (03) {};

\node at (\X*1.5, \X*5) {$\scriptstyle{T}$};
\draw (\X*1,\X*1) rectangle (\X*2,\X*2) node[fitting node] (11) {};
\draw (\X*1,\X*2) rectangle (\X*2,\X*3) node[fitting node] (12) {};
\node at ($(12)+(0.5,0)$) {$\scriptstyle{L^-}$};
\draw[dotted] (\X*1,\X*3) rectangle (\X*2,\X*4) node[fitting node] (13) {};
\node at ($(11)+(0.5,0)$) {$\scriptstyle{N^-}$};

\draw[arrows={- latex'},red] (13.center) -- (02.center);
\draw[arrows={- latex'},red] (02.center) -- (12.center);
\draw[arrows={- latex'},blue,dotted] (11.center) -- (02.center);

\node at (\X*1-1 + 1.5*\D, 2) {$\scriptstyle \text{composing arrows}$};
\draw[arrows={- latex'},thick] (\X*1-.8 + \D,1.5) -- (\X*1+2*\D-1,1.5);

\node at (0.5*\X+3*\D-2, \X*5) {$\scriptstyle{C}$};

\draw (0+3*\D-2, \X*1) rectangle (\X*1+3*\D-2, \X*2) node[fitting node] (c01) {};
\draw (0+3*\D-2, \X*2) rectangle (\X*1+3*\D-2, \X*3) node[fitting node] (c02) {};
\draw[dotted] (0+3*\D-2, \X*3) rectangle (\X*1+3*\D-2, \X* 4) node[fitting node] (c03) {};

\node at (\X*1.5+3*\D-2, \X*5) {$\scriptstyle{T}$};
\draw (\X*1+3*\D-2,\X*1) rectangle (\X*2+3*\D-2,\X*2) node[fitting node] (t11) {};
\draw (\X*1+3*\D-2,\X*2) rectangle (\X*2+3*\D-2,\X*3) node[fitting node] (t12) {};
\draw[dotted] (\X*1+3*\D-2,\X*3) rectangle (\X*2+3*\D-2,\X*4) node[fitting node] (t13) {};
\node at ($(t11)+(1,0)$) {$\scriptstyle{\text{Contrad.}}$}; 

\draw[arrows={- latex'}] (t11.center) -- (t12.center);

\end{tikzpicture}
};

\node (2) at (8,0) {

\begin{tikzpicture}
\node at (-1,3) {(b)};

\node at (0.5*\X, \X*5) {$\scriptstyle{C}$};

\draw (0, \X*1) rectangle (\X*1, \X*2) node[fitting node] (01) {};
\draw (0, \X*2) rectangle (\X*1, \X*3) node[fitting node] (02) {};
\draw[dotted] (0, \X*3) rectangle (\X*1, \X* 4) node[fitting node] (03) {};

\node at (\X*1.5, \X*5) {$\scriptstyle{T}$};
\draw (\X*1,\X*1) rectangle (\X*2,\X*2) node[fitting node] (11) {};
\draw (\X*1,\X*2) rectangle (\X*2,\X*3) node[fitting node] (12) {};
\node at ($(12)+(0.5,0)$) {$\scriptstyle{N^+}$};
\draw[dotted] (\X*1,\X*3) rectangle (\X*2,\X*4) node[fitting node] (13) {};
\node at ($(11)+(0.5,0)$) {$\scriptstyle{N^-}$};

\draw[arrows={- latex'},dotted,red] (02.center) -- (13.center);
\draw[arrows={- latex'},dotted,blue] (11.center) -- (02.center);

\node at (\X*1-1 + 1.5*\D, 2) {$\scriptstyle \text{composing arrows}$};
\draw[arrows={- latex'},thick] (\X*1-.8 + \D,1.5) -- (\X*1+2*\D-1,1.5);

\node at (0.5*\X+3*\D-2, \X*5) {$\scriptstyle{C}$};

\draw (0+3*\D-2, \X*1) rectangle (\X*1+3*\D-2, \X*2) node[fitting node] (c01) {};
\draw (0+3*\D-2, \X*2) rectangle (\X*1+3*\D-2, \X*3) node[fitting node] (c02) {};
\draw[dotted] (0+3*\D-2, \X*3) rectangle (\X*1+3*\D-2, \X* 4) node[fitting node] (c03) {};

\node at (\X*1.5+3*\D-2, \X*5) {$\scriptstyle{T}$};
\draw (\X*1+3*\D-2,\X*1) rectangle (\X*2+3*\D-2,\X*2) node[fitting node] (t11) {};
\draw (\X*1+3*\D-2,\X*2) rectangle (\X*2+3*\D-2,\X*3) node[fitting node] (t12) {};
\draw[dotted] (\X*1+3*\D-2,\X*3) rectangle (\X*2+3*\D-2,\X*4) node[fitting node] (t13) {};
\node at ($(t11)+(1,0)$) {$\scriptstyle{\text{Contrad.}}$}; 

\draw[arrows={- latex'},dotted] (t11.center) -- (t13.center);

\end{tikzpicture}

};
\end{tikzpicture}
\end{center}
\end{remark}

\begin{defn}
Let $(C,T)$ be a pair of column KR-tableaux. We define for each $p\in \mathbb{Z}$ functions $L_p$ and $N_p$ as follows:
\begin{eqnarray*}
L_p(C,T) = \left\{ \begin{array}{cc} 
1 & \text{ if } C[p-1] < T[p] < C[p]\\
-1 & \text{ if } T[p-1] < C[p] < T[p]\\
0 & \text{ otherwise} 
\end{array}\right.
&\text{;}& N_p(C,T)=
\left\{\begin{array}{cc}
1 & \text{ if }C[p] \leq T[p-1]\\
-1 & \text{ if } T[p] \leq C[p-1]\\
0 &\text{ otherwise }
\end{array}\right.
\end{eqnarray*}
In other words, $L_p(C,T)$ is $\pm1$ if there is an $L^\pm$-block at $p$ and $0$ otherwise, and similarly for $N_p(C,T)$. When there is no confusion as to which KR-tableaux we are referring to, we will suppress the dependance on $C$ and $T$ and simply write $L_p$ and $N_p$. 
\end{defn}

\begin{defn}
Let $(C,T)$ be a pair of column KR-tableaux with $\supp(C)\cap \supp(T) = \left\{ h, \ldots, t\right\}$. The block-tableau of $(C,T)$, denoted $B_{CT}$, is a single column diagram with $\supp(B_{CT}) = \left\{ h, \ldots, t, t+1\right\}$ decorated with letters $\left\{L^+,L^-,N^+,N^-,U\right\}$, such that the value of $B_{CT}$ at index $p$ is given by the corresponding block in $(C,T)$ at $p$. 
\end{defn}

\begin{example}
Block-tableau is a convenient way to describe all the blocks of $(C,T)$ at once. Saying there is an $L^+$-block at $p$ in $(C,T)$ is equivalent to $B_{CT}[p]=L^+$.
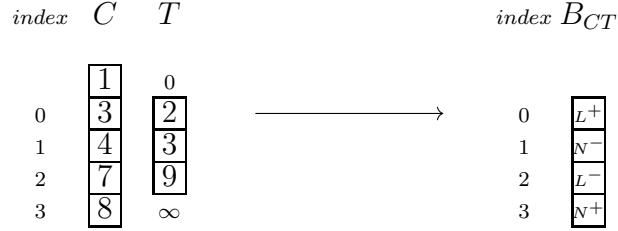
\begin{center}
\begin{tikzpicture}
\node at (0,0) {$
\ytableausetup
{mathmode, boxsize=1em}
\begin{ytableau}
\none[\scriptstyle index]&\none&\none[C] 	&\none	&\none[T]\\
\none&\none&\none	&\none	&\none\\
\none&\none&1		&\none	&\none[\scriptstyle 0]\\
\none[\scriptstyle 0]&\none&3		&\none	&2\\
\none[\scriptstyle 1]&\none&4		&\none	&3\\
\none[\scriptstyle 2]&\none&7		&\none	&9\\
\none[\scriptstyle 3]&\none&8		&\none &\none[\scriptstyle \infty]	\\
\end{ytableau}$};

\draw[->] (2,0)--(4.5,0);

\node at (6,0) {$
\begin{ytableau}
\none[\scriptstyle index]&\none&\none[B_{CT}] 		\\
\none&\none&\none	\\
\none&\none&\none		\\
\none[\scriptstyle 0]&\none&\scriptscriptstyle L^+		\\
\none[\scriptstyle 1]&\none&\scriptscriptstyle N^-		\\
\none[\scriptstyle 2]&\none&\scriptscriptstyle L^-		\\
\none[\scriptstyle 3]&\none&\scriptscriptstyle N^+			
\end{ytableau}$};
\end{tikzpicture}
\end{center}

\end{example}

\begin{defn} \label{def:pairs}
Let $(C,T)$ be a pair of column tableaux. We say $(C,T)$ is a \emph{fundamental} pair if $\head(C) \geq \head(T)$ and $\tail(C)\leq \tail(T)$. We say $(C,T)$ is \emph{anti-fundamental} if $(T,C)$ is fundamental. We say $(C,T)$ is a \emph{regular} pair if $\head(C) > \head(T)$ and $\tail(C) > \tail(T)$. We say $(C,T)$ is \emph{anti-regular} if $(T,C)$ is regular (see Figure~\ref{types}). 
\end{defn}

The reason for the name fundamental pair is because of the diagrams in $\mathcal{C}$ corresponding to the fundamental modules, i.e KR-modules with $k=1$. Every pair of fundamental modules in $\mathcal{C}$ forms a fundamental or an anti-fundamental pair as in Definition \ref{def:pairs} (see Figure~\ref{fig:fundTab}(a)). 

\begin{figure}[h] 
\begin{center}
\begin{tikzpicture}
\node at (0,0.5) {Fundamental};
\node at (0,-2) {$
\begin{ytableau}
\none[\scriptstyle index]			&\none		&\none[C]		&\none&\none[T]	\\
\none					&\none		&\none[\scriptstyle 0]			&\none&\none[\scriptstyle 0]		\\
\none[\scriptstyle h-1]		&\none		&\none[\scriptstyle 0]	&\none&			\\
\none[\scriptstyle h]			&\none		&*(yellow)			&\none&*(yellow)			\\
\none					&\none		&			&\none&			\\
\none[\scriptstyle t]			&\none		&			&\none&			\\
\none[\scriptstyle t+1]		&\none		&\none[\scriptstyle \infty]		&\none&*(yellow)			\\
\none[\scriptstyle t+2]		&\none		&\none[\scriptstyle \infty]		&\none&	\none[\scriptstyle \infty]		\\
\end{ytableau}
$};

\node at (3.5,-0.3) {$B_{CT}$};
\node at (3.5,-2) {$
\begin{ytableau}
\none\\
\none\\
		*(yellow)						\\
					\\
					\\
		*(yellow)			\\
\end{ytableau}
$};


\node at (7,0.5) {Regular};

\node at (7,-2) {$
\begin{ytableau}
\none[C]			&\none&\none[T]			&\none&\none[\scriptstyle index]							\\
\none[\scriptstyle 0]	&\none&\none[\scriptstyle 0]	&		\none&\none			\\
\none[\scriptstyle 0]	&\none&					&		\none&\none[\scriptstyle h-1]		\\
*(yellow)			&\none&*(yellow)			&			\none&\none[\scriptstyle h]		\\
				&\none&					&	\none&\none						\\
				&\none&					&			\none&\none[\scriptstyle t]		\\
*(yellow)			&\none&\none[\scriptstyle \infty]	&		\none&\none[\scriptstyle t+1]			\\
\none[\scriptstyle \infty]			&\none&\none[\scriptstyle \infty]	&	\none&\none[\scriptstyle t+2]		\\
\end{ytableau}
$};

\end{tikzpicture}
\end{center}
\caption{Types of pairs of column diagrams}
\label{types}
\end{figure}
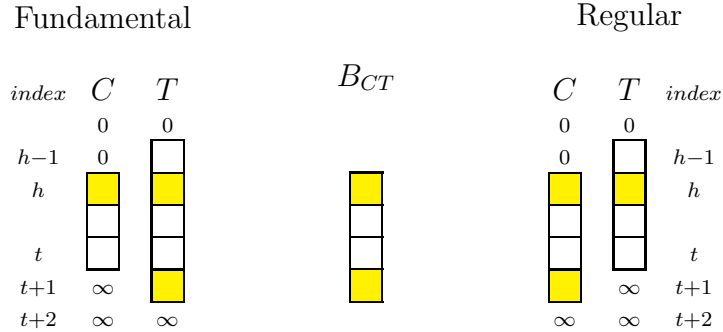


\begin{lemma} \label{rmk:obs2}

Let $(C,T)$ be a pair of column KR-tableaux and suppose $(C,T)$ is either fundamental or regular. Denote $h=\head(C,T)$ and $t=\tail(C,T)$. Then,

(a) $B_{CT}[h]$ is never $N^-$ since $C[h-1] = 0 \ngeq T[p]$.

(b) If $(C,T)$ is fundamental, $B_{CT}[t+1]$ is never $L^+$ or $N^+$ since $T[t]<T[t+1] \leq \infty = C[t]$.

(c) If $(C,T)$ is regular, $B_{CT}[t+1]$ is never $L^-$ or $N^-$ since $C[t]<C[t+1] < \infty = T[t]$.
\end{lemma}
\begin{proof}
The statements are clear from the following pictures:
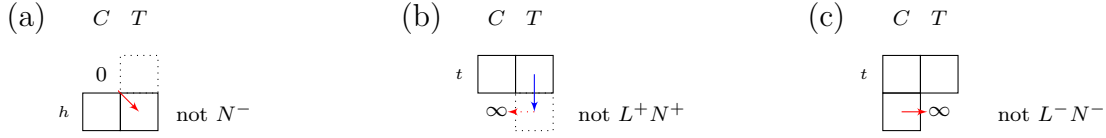
\begin{center}
\begin{tikzpicture}

\node at (0,0) {\begin{tikzpicture}

\node at (0+\D+0.5*\X-5, \X*5-0.5) {(a)};
\node at (0+\D+0.5*\X-4, \X*5-0.5) {$\scriptstyle{C}$};
\node at (0+\D+\X*1.5-4, \X*5-0.5) {$\scriptstyle{T}$};

\draw[white] (0+\D-4, \X*2) rectangle (\X*1+\D-4, \X*3) node[fitting node] (c02) {};
\draw (0+\D-4, \X*1) rectangle (\X*1+\D-4, \X*2) node[fitting node] (c01) {};
\node at ($(c01.center)-(\X,0)$) {$\scriptscriptstyle h$};

\draw[dotted] (\X*1+\D-4,\X*2) rectangle (\X*2+\D-4,\X*3) node[fitting node] (t12) {};
\draw (\X*1+\D-4,\X*1) rectangle (\X*2+\D-4,\X*2) node[fitting node] (t11) {};

\draw[arrows={- latex'},red] (c02.center) -- (t11.center);
\node[fill=white] at (c02.center) {$\scriptstyle{0}$};
\node at ($(t11)+(1,0)$) {$\scriptstyle{\text{not }N^-}$};
\end{tikzpicture}
};

\node at (5.5,0) {\begin{tikzpicture}
\node at (0+\D+0.5*\X-1, \X*5-0.5) {(b)};
\node at (0+\D+0.5*\X, \X*5-0.5) {$\scriptstyle{C}$};
\node at (0+\D+\X*1.5, \X*5-0.5) {$\scriptstyle{T}$};

\draw[white] (0+\D, \X*1) rectangle (\X*1+\D, \X*2) node[fitting node] (c01) {};
\draw (0+\D, \X*2) rectangle (\X*1+\D, \X*3) node[fitting node] (c02) {};
\node at ($(c02.center)-(\X,0)$) {$\scriptscriptstyle t$};

\draw (\X*1+\D,\X*2) rectangle (\X*2+\D,\X*3) node[fitting node] (t12) {};
\draw[dotted] (\X*1+\D,\X*1) rectangle (\X*2+\D,\X*2) node[fitting node] (t11) {};

\draw[arrows={- latex'},red,dotted] (t11.center) -- ($(t11.west)-(.1,0)$);
\draw[arrows={-latex'},blue]  (t12.center) -- (t11.center);
\node at (c01.center) {$\scriptstyle{\infty}$};
\node at ($(t11)+(1.3,0)$) {$\scriptstyle{\text{not }L^+N^+}$};
\end{tikzpicture}
};

\node at (11,0) {\begin{tikzpicture}
\node at (0+\D+0.5*\X+3, \X*5-0.5) {(c)};
\node at (0+\D+0.5*\X+4, \X*5-0.5) {$\scriptstyle{C}$};
\node at (0+\D+\X*1.5+4, \X*5-0.5) {$\scriptstyle{T}$};

\draw[white] (\X*1+\D+4,\X*1) rectangle (\X*2+\D+4,\X*2) node[fitting node] (11) {};
\draw (\X*1+\D+4,\X*2) rectangle (\X*2+\D+4,\X*3) node[fitting node] (12) {};

\draw (0+\D+4, \X*1) rectangle (\X*1+\D+4, \X*2) node[fitting node] (01) {};
\draw (0+\D+4, \X*2) rectangle (\X*1+\D+4, \X*3) node[fitting node] (02) {};
\node at ($(02.center)-(\X,0)$) {$\scriptscriptstyle t$};

\draw[arrows={- latex'},red] (01.center) -- ($(11.west)+(0.1,0)$);
draw[arrows={-latex'},blue] (12.center) -- (11.center);
\node  at (11.center) {$\scriptstyle{\infty}$};
\node at ($(11)+(1.5,0)$) {$\scriptstyle{\text{not }L^-N^-}$};
\end{tikzpicture}
};

\end{tikzpicture}

\end{center}
\end{proof}

\begin{remark}
It suffices to compute $\gamma(C,T)$ for $(C,T)$ either fundamental or regular. If $(C,T)$ is of neither type, then $(T,C)$ is either fundamental or regular and $\gamma(C,T)=-\gamma(T,C)$. 
\end{remark}

The value of $\gamma(C,T)$ depends on the occurrence of $L$ and $N$ blocks in $(C,T)$ encoded in the functions $L_p$ and $N_p$. More precisely,

\begin{thm}\label{gamma}
Let $(C,T)$ be a pair of column KR-tableaux and let $\supp(C)\cap \supp(T) := \left\{ h, \ldots, t\right\}$ (see Figure~\ref{types}). Then
\begin{eqnarray*}
\gamma(C,T) = \sum_{p=h}^{t} L_p (C,T) + \left\{ \begin{array}{cc} 
N_{t+1}(C,T) &\text{ if } (C,T) \text{ is (anti-)fundamental}\\
L_{t+1}(C,T) &\text{ otherwise}
\end{array}\right.  \,\, .
\end{eqnarray*}
\end{thm}

\begin{remark}
In other words, $\gamma(C,T)$ counts the number of $L^+$-blocks minus the number of $L^-$-blocks in $(C,T)$ with  one more contribution from the tail block. 
\end{remark}

\begin{proof} 
Let's decompose $L_p$ into its positive $L_p^+$ and negative $L_p^-$ parts such that 
\begin{eqnarray*}
L_p = L_p^+ - L_p^- 
\end{eqnarray*}
Recall that $\gamma(C,T) = d(C,T) - d(T,C)$. The calculation of $d(C,T)$ is borrowed from \cite{nakajima2002t}. For the reader's convenience, we reproduce it here. 

We want to compute (see Definition \ref{def:Nakajima}, Remark \ref{rmk:simplegamma}, and \eqref{simpleuv} )
\begin{equation} \label{d}
d(C,T)=\sum_{i,p} v_{i,i-2p}(C)u_{i,i-2p-1}(T) + \sum_{i,p}u_{i,i-2p}(C_{dom})v_{i,i-2p-1}(T) \,\, .
\end{equation}
Let $\left< condition \right>$ be $1$ if the $condition$ is true and $0$ otherwise. By \eqref{eq:aAction} and Definition \ref{uv}, we know
\begin{equation} \label{1}
v_{i,i-2p}(C) = \left<C_{dom}[p] \leq i \leq C[p]-1\right> \,\, .
\end{equation}
By \eqref{map:qmon} and Definition \ref{uv}, we have
\begin{equation} \label{2}
u_{i,i-2p-1}(T) = \left<T[p]=i\right> - \left<T[p+1]=i+1\right>\,\, .
\end{equation}
Putting \eqref{1} and \eqref{2} together,
\begin{eqnarray*}
\sum_i v_{i,i-2p}(C) u_{i,i-2p-1}(T) &=& \sum_i \left<C_{dom}[p] \leq i \leq C[p]-1\right>  \left(  \left<T[p]=i\right> - \left<T[p+1]=i+1\right> \right)\\
&=&\left<C_{dom}[p] \leq T[p] < C[p]\right> - \left<C_{dom}[p + 1] \leq T[p+1] \leq C[p]\right> \,\, ,
\end{eqnarray*}
where we used $C_{dom}[p+1] = C_{dom}[p]+1$. We want to sum up the above value for all $p$. By shifting the summation index $p$, we consider the sums of the form
\begin{eqnarray*}
- \left< C_{dom}[p] < T[p] \leq C[p-1] \right> + \left< C_{dom}[p] \leq T[p] < C[p]\right> = \left< C[p-1]<T[p]<C[p]\right> = L_p^+ \,\, .
\end{eqnarray*}
Now, summing over all $p$, we obtain:
\begin{equation} \label{A}
\sum_{i,p} v_{i,i-2p}(C)u_{i,i-2p-1}(T) = \sum_{p=h}^t L_p^+ - \left< C_{dom}[t] < T[t+1] \leq C[t]\right> \,\, ,
\end{equation}
which gives the first half of $d(C,T)$. The second half of \eqref{d} consists of a single term. That is because $m_{C_{dom}} = Y_{i, j}$, where $i = \length(C)= \tail(C)-\head(C) + 1$ and $j= -\tail(C) - \head(C)$. By setting $j = i-2p+1$, we evaluate $p^*=\tail(C) + 1$ and $i^*=\length(C)$ are the only values for which $u_{i^*,i^*-2p^*+1}(C_{dom}) = 1$. Therefore, 
\begin{eqnarray} 
\sum_{i,p}u_{i,i-2p+1}(C_{dom})v_{i,i-2p}(T)  &=& v_{i^*,i^*-2p^*}(T) = \left<T_{dom}[p^*]\leq i^*\leq T[p^*]-1\right> \nonumber\\
&=&\left<  T_{dom}[\tail(C) + 1]\leq \length(C) \leq T[\tail(C)+1]-1 \right> \label{B} \,\, .
\end{eqnarray}
By adding \eqref{A} and \eqref{B}, we obtain $d(C,T)$. The value of $d(T,C)$ is obtained by switching $C$ and $T$, which gives a sum of $L_p^-$ in the equivalent of \eqref{A}. Putting everything together, we obtain
\begin{eqnarray*}
\gamma(C,T) = d(C,T) - d(T,C) = \sum_{p=n}^m (L_p^+ - L_p^-)  + \Bd(C,T) \,\, ,
\end{eqnarray*}
where $Bd(C,T)$ is the boundary term given by
\begin{eqnarray*}
\Bd(C,T) &=&  \left< T_{dom}[t] < C[t+1] \leq T[t]\right>  - \left< C_{dom}[t] < T[t+1] \leq C[t]\right> \\
&&+ \left<  T_{dom}[\tail(C) + 1]\leq \length(C)< T[\tail(C)+1] \right> \\
&&- \left<  C_{dom}[\tail(T) + 1]\leq \length(T)< C[\tail(T)+1] \right>\,\, .
\end{eqnarray*}
We simplify the boundary term next. We consider two cases: when $(C,T)$ is fundamental and when $(C,T)$ is regular. 

Suppose $(C,T)$ is fundamental. Then $\tail(C) = t$, $\tail(T)\geq t$, and $C[p]=C_{dom}[p]=\infty$ for $p\geq t+1$ (see Figure~\ref{types}). Notice that,
\begin{equation}\label{eq:rel1}
C_{dom}[t] = \length(C) \leq \length(T[\head(T),t+1]) = T_{dom}[t+1] \leq T[t+1] \,\, .
\end{equation}
We compute
\begin{eqnarray*}
 0 &=& \left< T_{dom}[t] < \underbrace{C[t+1]}_\text{$\infty$} \leq \underbrace{T[t]}_\text{$\neq \infty$}\right> = \left<  \underbrace{T_{dom}[t + 1]\leq \length(C)}_\text{never true by Eq. \ref{eq:rel1}} < T[t+1] \right> \\&=& \left<  \underbrace{C_{dom}[\tail(T) + 1]}_\text{$\infty$ since $\tail(T)\geq t$} \leq \length(T)< C[\tail(T)+1] \right>\,\, .
\end{eqnarray*}
Thus,
\begin{equation}
\Bd(C,T) = \left< \underbrace{C_{dom}[t] < T[t+1]}_\text{always true by Eq. \ref{eq:rel1}} \leq C[t] \right> = \left< T[t+1] \leq C[t]\right> = N_{t+1}
\end{equation}
Suppose $(C,T)$ is regular. Then $\tail(T) = t$, $\tail(C) \geq t$, and $T[p]= T_{dom}[p]=\infty$ for all $p\geq t+1$. Therefore,
\begin{eqnarray*}
\left< C_{dom}[t] < \underbrace{T[t+1]}_\text{$\infty$} \leq \underbrace{C[t]}_\text{$\neq \infty$} \right> =\left<  \underbrace{T_{dom}[\tail(C) + 1]}_\text{$\infty$ since $\tail(C)\geq t$}\leq \length(C) < T[\tail(C)+1] \right> = 0 \,\, .
\end{eqnarray*}
Notice that
\begin{eqnarray*}
C_{dom}[t+1] = t+1 - \head(C) + 1 < \underbrace{t - \head(C) + 1 \leq \tail(T) - \head(T) + 1}_\text{$\head(C)>\head(T)$ (see Figure~\ref{types}) and $\tail(T)=t$} = \length(T) \,\, .
\end{eqnarray*}
Therefore, removing the condition that is always satisfied and substituting $\length(T)=T_{dom}[t]$, we obtain:
\begin{eqnarray*}
\Bd(C,T) &=& \left< T_{dom}[t] < C[t+1] \leq T[t]\right> - \left<  T_{dom}[t]< C[t+1] \right> \\
&=& - \left< T[t]<C[t+1] \right> = - \left< T[t]<C[t+1] <T[t+1]\right> = - L_{t+1}^- = L_{t+1} \,\, ,
\end{eqnarray*}
where we added the condition $C[t+1] < T[t+1]=\infty$, which is always satisfied and, therefore, does not affect the outcome. 
\end{proof}

\subsection{Exchanging boxes, compatibility conditions, and the involution $\sigma$} \label{inv}
In this section, we define the core concept of exchanging boxes of tableaux, which defines the map $\sigma$ and the three subsets $\mathcal{P}_0,\mathcal{P}_1,\mathcal{P}_{-1}$ of the set $\mathcal{B}_{k,j}^{(i)}\times \mathcal{B}_{k',j'}^{(i')}$ for any $(i,j,k)$ and $(i',j',k')$.

\begin{defn}
Let $(C,T)$ be a pair of column tableaux. A \emph{strip} in $(C,T)$ is a slice of $(C,T)$ given by $(C[p_0,p_1], T[p_0,p_1])$ for some $p_0<p_1\in \supp(C)\cap \supp(T)$. An \emph{$L$-strip} is a strip that starts with an $L$-block and includes all the $N$-blocks that follow. 
\end{defn}

\begin{defn}
Let $(C,T)$ be a pair of column KR-tableaux. We say an $L$-strip in $(C,T)$ is \emph{column-compatible} if it is possible to exchange the boxes in the $L$-strip and obtain a pair of valid column KR-tableaux. That is, the new tableaux also have strictly increasing columns.  
\end{defn}

\begin{example} \label{ex:columncompat}
Here is an $L^+$-strip starting at $p=0$ and ending at $p=2$. It is column-compatible because it is possible to exchange $C[0,2]$ and $T[0,2]$ and obtain a new pair of column tableaux $(C',T')$. Notice that both $C'$ and $T'$ have strictly increasing columns.  
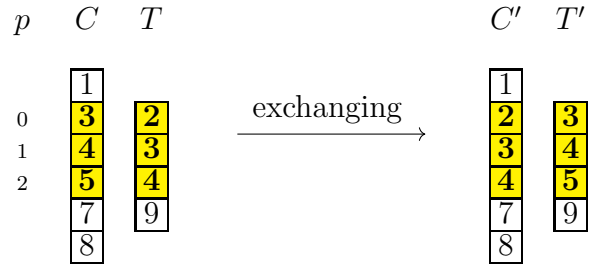
\begin{center}
\begin{tikzpicture}
\node at (0,0) {$
\ytableausetup
{mathmode, boxsize=1em}
\begin{ytableau}
\none[p]&\none&\none[C] 	&\none	&\none[T]\\
\none&\none&\none	&\none	&\none\\
\none&\none&1		&\none	&\none\\
\none[\scriptstyle 0]&\none&*(yellow)\mathbf{3}		&\none	&*(yellow)\mathbf{2}\\
\none[\scriptstyle 1]&\none&*(yellow)\mathbf{4}		&\none	&*(yellow)\mathbf{3}\\
\none[\scriptstyle 2]&\none&*(yellow)\mathbf{5}		&\none	&*(yellow)\mathbf{4}\\
\none &\none&7		&\none	&9\\
\none&\none&8		&\none	\\
\end{ytableau}$};

\draw[->] (2,0)--(4.5,0);
\node at (3.2,0)[above] {exchanging};

\node at (6,0) {$
\begin{ytableau}
\none[C'] 	&\none	&\none[T']\\
\none	&\none	&\none\\
1		&\none	&\none\\
*(yellow)\mathbf{2}		&\none	&*(yellow)\mathbf{3}\\
*(yellow)\mathbf{3}		&\none	&*(yellow)\mathbf{4}\\
*(yellow)\mathbf{4}		&\none	&*(yellow)\mathbf{5}\\
7		&\none	&9\\
8		&\none	
\end{ytableau}$};
\end{tikzpicture}
\end{center}
\end{example}

\begin{remark} \label{cross-pattern}
For a strip to be column-compatible, it must begin and end with an equality condition of a cross-pattern, which occurs at $LU$-blocks. Indeed, 

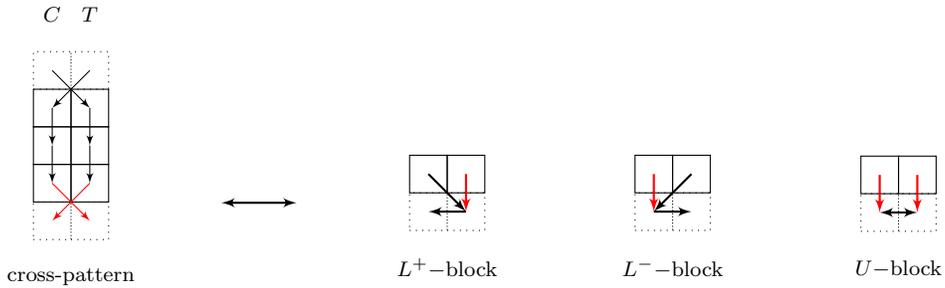
\begin{center}
\begin{tikzpicture}

\node at (-5,0) {
\begin{tikzpicture}
\node at (0.5*\X, \X*5) {$\scriptstyle{C}$};

\draw[dotted] (0,-1*\X) rectangle (\X*1,\X*0) node[fitting node] (c-1) {};
\draw (0,0) rectangle (\X*1,\X*1) node[fitting node] (c00) {};
\draw (0, \X*1) rectangle (\X*1, \X*2) node[fitting node] (c01) {};
\draw (0, \X*2) rectangle (\X*1, \X* 3) node[fitting node] (c03) {};
\draw[dotted] (0, \X*3) rectangle (\X*1, \X* 4) node[fitting node] (c04) {};

\node at (1.5*\X, \X*5) {$\scriptstyle{T}$};
\draw[dotted] (\X*1,-1*\X) rectangle (\X*2,\X*0) node[fitting node] (t-1) {};
\draw (\X*1,\X*0) rectangle (\X*2,\X*1) node[fitting node] (t00) {};
\draw (\X*1,\X*1) rectangle (\X*2,\X*2) node[fitting node] (t01) {};
\draw (\X*1,\X*2) rectangle (\X*2,\X*3) node[fitting node] (t03) {};
\draw[dotted] (\X*1,\X*3) rectangle (\X*2,\X*4) node[fitting node] (t04) {};

\node at (\X*1, -2*\X) {$\scriptstyle \text{cross-pattern}$
};

\draw[arrows={- latex'}] (c04.center) -- (t03.center);
\draw[arrows={- latex'}] (t04.center) -- (c03.center);
\draw[arrows={- latex'}] (t03.center) -- (t01.center);
\draw[arrows={- latex'}] (t01.center) -- (t00.center);
\draw[arrows={- latex'},red] (t00.center) -- (c-1.center);

\draw[arrows={- latex'}] (c03.center) -- (c01.center);
\draw[arrows={- latex'}] (c01.center) -- (c00.center);
\draw[arrows={- latex'},red] (c00.center) -- (t-1.center);

\end{tikzpicture}
};
   
\draw[arrows={latex' - latex'},thick] (-3,-.75) -- (-2,-.75);

\node at (0,-1) {
\begin{tikzpicture}
\draw[dotted] (0*\X,0*\X) rectangle (1*\X,1*\X) node[fitting node] (00) {};
\draw (0*\X,1*\X) rectangle (1*\X,2*\X) node[fitting node] (01) {};
\draw[dotted] (1*\X,0*\X) rectangle (2*\X,1*\X) node[fitting node] (10) {};
\draw (1*\X,1*\X) rectangle (2*\X,2*\X) node[fitting node] (11) {};

\draw[arrows={- latex'},thick] (01.center) -- (10.center);
\draw[arrows={- latex'},thick] (10.center) -- (00.center);
\draw[arrows={- latex'},thick,red] (11.center) -- (10.center);

\node at (1*\X,-1*\X) { $\scriptstyle L^+-\text{block}$};
\end{tikzpicture}
};

\node at (3,-1) {
\begin{tikzpicture}
\draw[dotted] (0*\X,0*\X) rectangle (1*\X,1*\X) node[fitting node] (00) {};
\draw (0*\X,1*\X) rectangle (1*\X,2*\X) node[fitting node] (01) {};
\draw[dotted] (1*\X,0*\X) rectangle (2*\X,1*\X) node[fitting node] (10) {};
\draw (1*\X,1*\X) rectangle (2*\X,2*\X) node[fitting node] (11) {};

\node at (1*\X,-1*\X) { $\scriptstyle L^- -\text{block}$};

\draw[arrows={latex'-},thick] (00.center) -- (11.center);
\draw[arrows={latex' -},thick] (10.center) -- (00.center);
\draw[arrows={- latex'},thick,red] (01.center) -- (00.center);
\end{tikzpicture}
};

\node at (6,-1) {
\begin{tikzpicture}
\draw[dotted] (0*\X,0*\X) rectangle (1*\X,1*\X) node[fitting node] (00) {};
\draw (0*\X,1*\X) rectangle (1*\X,2*\X) node[fitting node] (01) {};
\draw[dotted] (1*\X,0*\X) rectangle (2*\X,1*\X) node[fitting node] (10) {};
\draw (1*\X,1*\X) rectangle (2*\X,2*\X) node[fitting node] (11) {};

\node at (1*\X,-1*\X) { $\scriptstyle U -\text{block}$};

\draw[arrows={latex' - latex'},thick]  (00.center) -- (10.center);
\draw[arrows={- latex'},thick,red] (11.center) -- (10.center);
\draw[arrows={- latex'},thick,red] (01.center) -- (00.center);
\end{tikzpicture}
};
\end{tikzpicture}
\end{center}
An $L$-strip can either be followed by an $LU$-block or by nothing at all. In the first case, the $L$-strip is always column-compatible. In particular, this means only the very last $L$-strip is potentially not column-compatible. All others are column-compatible since they are, by definition, at the very least followed by the next $L$-block.
\end{remark}

\begin{example} \label{ex:noncolumncompat}
Here is an $L^-$-strip starting at $p=0$ and ending at $p=3$. It is not column-compatible since exchanging $C[0,3]$ and $T[0,3]$ will violate the strictly increasing columns condition post-exchange. 
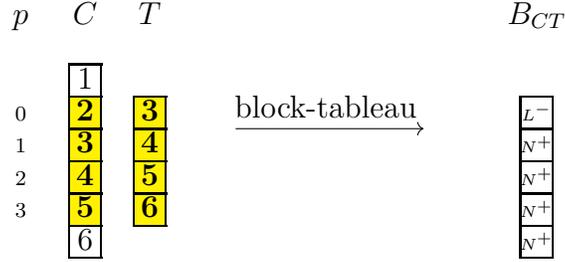
\begin{center}
\begin{tikzpicture}
\node at (0,0) {$
\ytableausetup
{mathmode, boxsize=1em}
\begin{ytableau}
\none[p]&\none&\none[C] 	&\none	&\none[T]\\
\none&\none&\none	&\none	&\none\\
\none&\none&1		&\none	&\none\\
\none[\scriptstyle 0]&\none&*(yellow)\mathbf{2}		&\none	&*(yellow)\mathbf{3}\\
\none[\scriptstyle 1]&\none&*(yellow)\mathbf{3}		&\none	&*(yellow)\mathbf{4}\\
\none[\scriptstyle 2]&\none&*(yellow)\mathbf{4}		&\none	&*(yellow)\mathbf{5}\\
\none[\scriptstyle 3] &\none&*(yellow)\mathbf{5}		&\none	&*(yellow)\mathbf{6}\\
\none&\none&6	&\none	\\
\end{ytableau}$};

\draw[->] (2,0)--(4.5,0);
\node at (3.2,0)[above] {block-tableau};

\node at (6,0) {$
\begin{ytableau}
\none[B_{CT}] 	\\
\none	\\
\none		\\
\scriptscriptstyle L^-		\\
\scriptscriptstyle N^+		\\
\scriptscriptstyle N^+		\\
\scriptscriptstyle N^+		\\
\scriptscriptstyle N^+	\\
\end{ytableau}$};
\end{tikzpicture}
\end{center}
\end{example}

\begin{defn} \label{defn:compat}
Let $C=(C_{0}, C_1)$ and $T=(T_{0},T_1)$ be KR-tableaux and let $B_{i,j}$ denote the block-tableau of $(C_i,T_j)$. 

Suppose $(C_1,T_1)$ forms an $L$-strip.  We say $(C_1,T_1)$ is \emph{left-compatible} if the weakly increasing diagonals conditions in Definition \ref{defn:KRtableau}(2) are not violated when $C_1$ and $T_1$ are exchanged. Pictorially, $(C_1,T_1)$ is left-compatible if conditions $(lC)$ (left-compatibility from $C$) and $(lT)$ (left-compatibility from $T$) are satisfied:

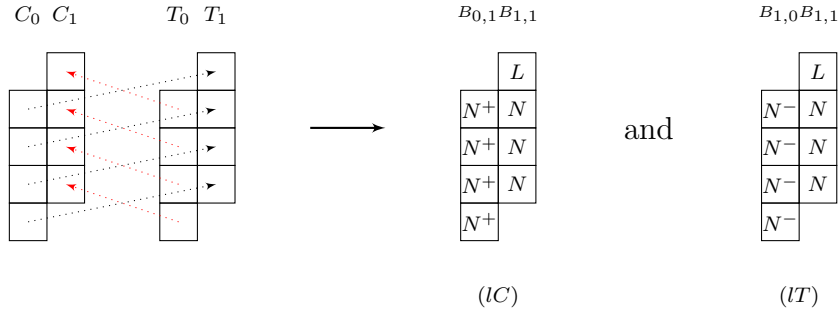
\begin{center}
\begin{tikzpicture}

\node at (0.5*\X, \X*6) {$\scriptstyle{C_{0}}$};
\draw (0,0) rectangle (\X*1,\X*1) node[fitting node] (00) {};
\draw (0, \X*1) rectangle (\X*1, \X*2) node[fitting node] (01) {};
\draw (0, \X*2) rectangle (\X*1, \X*3) node[fitting node] (02) {};
\draw (0, \X*3) rectangle (\X*1, \X* 4) node[fitting node] (03) {};

\node at (1.5*\X, \X*6) {$\scriptstyle{C_1}$};
\draw (\X*1,\X*1) rectangle (\X*2,\X*2) node[fitting node] (11) {};
\draw (\X*1,\X*2) rectangle (\X*2,\X*3) node[fitting node] (12) {};
\draw (\X*1,\X*3) rectangle (\X*2,\X*4) node[fitting node] (13) {};
\draw (\X*1,\X*4) rectangle (\X*2,\X*5) node[fitting node] (14) {};

\node at (0.5*\X+\D, \X*6) {$\scriptstyle{T_{0}}$};
\draw (0+\D,0) rectangle (\X*1+\D,\X*1) node[fitting node] (t00) {};
\draw (0+\D, \X*1) rectangle (\X*1+\D, \X*2) node[fitting node] (t01) {};
\draw (0+\D, \X*2) rectangle (\X*1+\D, \X*3) node[fitting node] (t02) {};
\draw (0+\D, \X*3) rectangle (\X*1+\D, \X* 4) node[fitting node] (t03) {};

\node at (1.5*\X+\D, \X*6) {$\scriptstyle{T_1}$};
\draw (\X*1+\D,\X*1) rectangle (\X*2+\D,\X*2) node[fitting node] (t11) {};
\draw (\X*1+\D,\X*2) rectangle (\X*2+\D,\X*3) node[fitting node] (t12) {};
\draw (\X*1+\D,\X*3) rectangle (\X*2+\D,\X*4) node[fitting node] (t13) {};
\draw (\X*1+\D,\X*4) rectangle (\X*2+\D,\X*5) node[fitting node] (t14) {};

\draw[arrows={- latex'},dotted] (03.center) -- (t14.center);
\draw[arrows={- latex'},dotted] (02.center) -- (t13.center);
\draw[arrows={- latex'},dotted] (01.center) -- (t12.center);
\draw[arrows={- latex'},dotted] (00.center) -- (t11.center);

\draw[arrows={- latex'},dotted,red] (t03.center) -- (14.center);
\draw[arrows={- latex'},dotted,red] (t02.center) -- (13.center);
\draw[arrows={- latex'},dotted,red] (t01.center) -- (12.center);
\draw[arrows={- latex'},dotted,red] (t00.center) -- (11.center);

\draw[arrows={- latex'},thick] (4,1.5) -- (5,1.5);

\node at (0 + 3*\D+.23, \X*6) {$\scriptscriptstyle{B_{0,1}}$};
\draw (0+3*\D,0) rectangle (\X*1+3*\D,\X*1) node[fitting node] (b00) {};
\draw (0+3*\D, \X*1) rectangle (\X*1+3*\D, \X*2) node[fitting node] (b01) {};
\draw (0+3*\D, \X*2) rectangle (\X*1+3*\D, \X*3) node[fitting node] (b02) {};
\draw (0+3*\D, \X*3) rectangle (\X*1+3*\D, \X* 4) node[fitting node] (b03) {};

\node at (b00.center) {$\scriptstyle{N^+}$};
\node at (b01.center) {$\scriptstyle{N^+}$};
\node at (b02.center) {$\scriptstyle{N^+}$};
\node at (b03.center) {$\scriptstyle{N^+}$};

\node at ($(b00)+(0.25,-1)$) {$\scriptstyle{(lC)}$};

\node at (2*\X+3*\D-.23, \X*6) {$\scriptscriptstyle{B_{1,1}}$};
\draw (\X*1+3*\D,\X*1) rectangle (\X*2+3*\D,\X*2) node[fitting node] (b11) {};
\draw (\X*1+3*\D,\X*2) rectangle (\X*2+3*\D,\X*3) node[fitting node] (b12) {};
\draw (\X*1+3*\D,\X*3) rectangle (\X*2+3*\D,\X*4) node[fitting node] (b13) {};
\draw (\X*1+3*\D,\X*4) rectangle (\X*2+3*\D,\X*5) node[fitting node] (b14) {};

\node at (b11.center) {$\scriptstyle{N}$};
\node at (b12.center) {$\scriptstyle{N}$};
\node at (b13.center) {$\scriptstyle{N}$};
\node at (b14.center) {$\scriptstyle{L}$};

\node at (8.5, 1.5) {and};

\node at (0 + 5*\D+.23, \X*6) {$\scriptscriptstyle{B_{1,0}}$};
\draw (0+5*\D,0) rectangle (\X*1+5*\D,\X*1) node[fitting node] (g00) {};
\draw (0+5*\D, \X*1) rectangle (\X*1+5*\D, \X*2) node[fitting node] (g01) {};
\draw (0+5*\D, \X*2) rectangle (\X*1+5*\D, \X*3) node[fitting node] (g02) {};
\draw (0+5*\D, \X*3) rectangle (\X*1+5*\D, \X* 4) node[fitting node] (g03) {};

\node at (g00.center) {$\scriptstyle{N^-}$};
\node at (g01.center) {$\scriptstyle{N^-}$};
\node at (g02.center) {$\scriptstyle{N^-}$};
\node at (g03.center) {$\scriptstyle{N^-}$};

\node at (2*\X+5*\D-.23, \X*6) {$\scriptscriptstyle{B_{1,1}}$};
\draw (\X*1+5*\D,\X*1) rectangle (\X*2+5*\D,\X*2) node[fitting node] (g11) {};
\draw (\X*1+5*\D,\X*2) rectangle (\X*2+5*\D,\X*3) node[fitting node] (g12) {};
\draw (\X*1+5*\D,\X*3) rectangle (\X*2+5*\D,\X*4) node[fitting node] (g13) {};
\draw (\X*1+5*\D,\X*4) rectangle (\X*2+5*\D,\X*5) node[fitting node] (g14) {};

\node at (g11.center) {$\scriptstyle{N}$};
\node at (g12.center) {$\scriptstyle{N}$};
\node at (g13.center) {$\scriptstyle{N}$};
\node at (g14.center) {$\scriptstyle{L}$};

\node at ($(g00)+(0.25,-1)$) {$\scriptstyle{(lT)}$};

\end{tikzpicture}
\end{center}

Suppose $(C_0,T_0)$ forms an $L$-strip. We say $L$-strip $(C_0,T_0)$ is \emph{right-compatible} if the weakly increasing diagonals condition in Definition \ref{defn:KRtableau}(2) is not violated when the boxes in $C_0$ and $T_0$ are exchanged. Pictorially, $(C_0,T_0)$ is right-compatible if conditions $(rC)$ (right-compatibility from $C$) and $(rT)$ (right-compatibility from $T$) are satisfied:

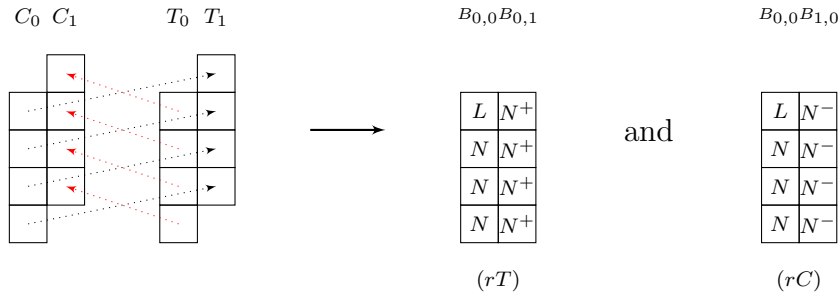
\begin{center}
\begin{tikzpicture}

\node at (0.5*\X, \X*6) {$\scriptstyle{C_{0}}$};
\draw (0,0) rectangle (\X*1,\X*1) node[fitting node] (00) {};
\draw (0, \X*1) rectangle (\X*1, \X*2) node[fitting node] (01) {};
\draw (0, \X*2) rectangle (\X*1, \X*3) node[fitting node] (02) {};
\draw (0, \X*3) rectangle (\X*1, \X* 4) node[fitting node] (03) {};

\node at (1.5*\X, \X*6) {$\scriptstyle{C_1}$};
\draw (\X*1,\X*1) rectangle (\X*2,\X*2) node[fitting node] (11) {};
\draw (\X*1,\X*2) rectangle (\X*2,\X*3) node[fitting node] (12) {};
\draw (\X*1,\X*3) rectangle (\X*2,\X*4) node[fitting node] (13) {};
\draw (\X*1,\X*4) rectangle (\X*2,\X*5) node[fitting node] (14) {};

\node at (0.5*\X+\D, \X*6) {$\scriptstyle{T_{0}}$};
\draw (0+\D,0) rectangle (\X*1+\D,\X*1) node[fitting node] (t00) {};
\draw (0+\D, \X*1) rectangle (\X*1+\D, \X*2) node[fitting node] (t01) {};
\draw (0+\D, \X*2) rectangle (\X*1+\D, \X*3) node[fitting node] (t02) {};
\draw (0+\D, \X*3) rectangle (\X*1+\D, \X* 4) node[fitting node] (t03) {};

\node at (1.5*\X+\D, \X*6) {$\scriptstyle{T_1}$};
\draw (\X*1+\D,\X*1) rectangle (\X*2+\D,\X*2) node[fitting node] (t11) {};
\draw (\X*1+\D,\X*2) rectangle (\X*2+\D,\X*3) node[fitting node] (t12) {};
\draw (\X*1+\D,\X*3) rectangle (\X*2+\D,\X*4) node[fitting node] (t13) {};
\draw (\X*1+\D,\X*4) rectangle (\X*2+\D,\X*5) node[fitting node] (t14) {};

\draw[arrows={- latex'},dotted] (03.center) -- (t14.center);
\draw[arrows={- latex'},dotted] (02.center) -- (t13.center);
\draw[arrows={- latex'},dotted] (01.center) -- (t12.center);
\draw[arrows={- latex'},dotted] (00.center) -- (t11.center);

\draw[arrows={- latex'},dotted,red] (t03.center) -- (14.center);
\draw[arrows={- latex'},dotted,red] (t02.center) -- (13.center);
\draw[arrows={ - latex'},dotted,red] (t01.center) -- (12.center);
\draw[arrows={- latex'},dotted,red] (t00.center) -- (11.center);

\draw[arrows={- latex'},thick] (4,1.5) -- (5,1.5);

\node at (0 + 3*\D+.23, \X*6) {$\scriptscriptstyle{B_{0,0}}$};
\draw (0+3*\D,0) rectangle (\X*1+3*\D,\X*1) node[fitting node] (b00) {};
\draw (0+3*\D, \X*1) rectangle (\X*1+3*\D, \X*2) node[fitting node] (b01) {};
\draw (0+3*\D, \X*2) rectangle (\X*1+3*\D, \X*3) node[fitting node] (b02) {};
\draw (0+3*\D, \X*3) rectangle (\X*1+3*\D, \X* 4) node[fitting node] (b03) {};

\node at (b00.center) {$\scriptstyle{N}$};
\node at (b01.center) {$\scriptstyle{N}$};
\node at (b02.center) {$\scriptstyle{N}$};
\node at (b03.center) {$\scriptstyle{L}$};

\node at (2*\X+3*\D-.23, \X*6) {$\scriptscriptstyle{B_{0,1}}$};
\draw (\X*1+3*\D,\X*1) rectangle (\X*2+3*\D,\X*2) node[fitting node] (b11) {};
\draw (\X*1+3*\D,\X*2) rectangle (\X*2+3*\D,\X*3) node[fitting node] (b12) {};
\draw (\X*1+3*\D,\X*3) rectangle (\X*2+3*\D,\X*4) node[fitting node] (b13) {};
\draw (\X*1+3*\D,\X*0) rectangle (\X*2+3*\D,\X*1) node[fitting node] (b14) {};

\node at (b11.center) {$\scriptstyle{N^+}$};
\node at (b12.center) {$\scriptstyle{N^+}$};
\node at (b13.center) {$\scriptstyle{N^+}$};
\node at (b14.center) {$\scriptstyle{N^+}$};

\node at (0+3.25*\D, -1*\X) {$\scriptstyle{(rT)}$};

\node at (8.5, 1.5) {and};

\node at (0 + 5*\D+.23, \X*6) {$\scriptscriptstyle{B_{0,0}}$};
\draw (0+5*\D,0) rectangle (\X*1+5*\D,\X*1) node[fitting node] (g00) {};
\draw (0+5*\D, \X*1) rectangle (\X*1+5*\D, \X*2) node[fitting node] (g01) {};
\draw (0+5*\D, \X*2) rectangle (\X*1+5*\D, \X*3) node[fitting node] (g02) {};
\draw (0+5*\D, \X*3) rectangle (\X*1+5*\D, \X* 4) node[fitting node] (g03) {};

\node at (g00.center) {$\scriptstyle{N}$};
\node at (g01.center) {$\scriptstyle{N}$};
\node at (g02.center) {$\scriptstyle{N}$};
\node at (g03.center) {$\scriptstyle{L}$};

\node at (2*\X+5*\D-.23, \X*6) {$\scriptscriptstyle{B_{1,0}}$};
\draw (\X*1+5*\D,\X*1) rectangle (\X*2+5*\D,\X*2) node[fitting node] (g11) {};
\draw (\X*1+5*\D,\X*2) rectangle (\X*2+5*\D,\X*3) node[fitting node] (g12) {};
\draw (\X*1+5*\D,\X*3) rectangle (\X*2+5*\D,\X*4) node[fitting node] (g13) {};
\draw (\X*1+5*\D,\X*0) rectangle (\X*2+5*\D,\X*1) node[fitting node] (g14) {};

\node at (g11.center) {$\scriptstyle{N^-}$};
\node at (g12.center) {$\scriptstyle{N^-}$};
\node at (g13.center) {$\scriptstyle{N^-}$};
\node at (g14.center) {$\scriptstyle{N^-}$};

\node at (0+5.25*\D, -1*\X) {$\scriptstyle{(rC)}$};

\end{tikzpicture}
\end{center}
\end{defn}

\begin{remark} \label{automaticComp}
An $L^-$-strip always satisfies $(rT)$ and $(lC)$ and an $L^+$-strip always satisfies $(rC)$ and $(lT)$. This can be seen by composing the arrows of the $L$-strip with the arrows within the KR-tableaux. For example, 

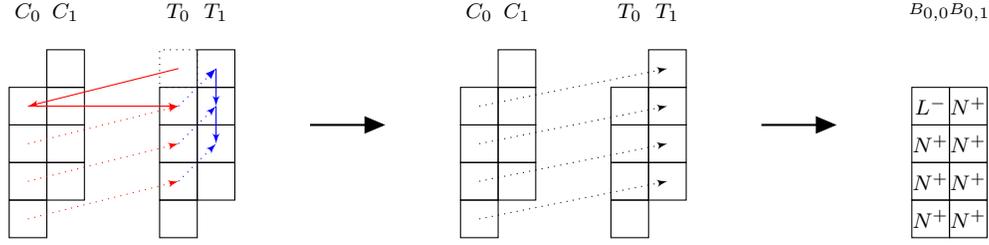
\begin{center}
\begin{tikzpicture}

\node at (0.5*\X, \X*6) {$\scriptstyle{C_{0}}$};
\draw (0,0) rectangle (\X*1,\X*1) node[fitting node] (00) {};
\draw (0, \X*1) rectangle (\X*1, \X*2) node[fitting node] (01) {};
\draw (0, \X*2) rectangle (\X*1, \X*3) node[fitting node] (02) {};
\draw (0, \X*3) rectangle (\X*1, \X* 4) node[fitting node] (03) {};

\node at (1.5*\X, \X*6) {$\scriptstyle{C_1}$};
\draw (\X*1,\X*1) rectangle (\X*2,\X*2) node[fitting node] (11) {};
\draw (\X*1,\X*2) rectangle (\X*2,\X*3) node[fitting node] (12) {};
\draw (\X*1,\X*3) rectangle (\X*2,\X*4) node[fitting node] (13) {};
\draw (\X*1,\X*4) rectangle (\X*2,\X*5) node[fitting node] (14) {};

\node at (0.5*\X+\D, \X*6) {$\scriptstyle{T_{0}}$};
\draw[dotted] (0+\D, \X*4) rectangle (\X*1+\D, \X*5) node[fitting node] (t04) {};

\draw (0+\D,0) rectangle (\X*1+\D,\X*1) node[fitting node] (t00) {};
\draw (0+\D, \X*1) rectangle (\X*1+\D, \X*2) node[fitting node] (t01) {};
\draw (0+\D, \X*2) rectangle (\X*1+\D, \X*3) node[fitting node] (t02) {};
\draw (0+\D, \X*3) rectangle (\X*1+\D, \X* 4) node[fitting node] (t03) {};

\node at (1.5*\X+\D, \X*6) {$\scriptstyle{T_1}$};
\draw (\X*1+\D,\X*1) rectangle (\X*2+\D,\X*2) node[fitting node] (t11) {};
\draw (\X*1+\D,\X*2) rectangle (\X*2+\D,\X*3) node[fitting node] (t12) {};
\draw (\X*1+\D,\X*3) rectangle (\X*2+\D,\X*4) node[fitting node] (t13) {};
\draw (\X*1+\D,\X*4) rectangle (\X*2+\D,\X*5) node[fitting node] (t14) {};

\draw[arrows={- latex'},red] (t04.center) -- (03.center);
\draw[arrows={- latex'},red] (03.center) -- (t03.center);
\draw[arrows={- latex'},dotted,color=red] (02.center) -- (t03.center);
\draw[arrows={- latex'},dotted,color=red] (01.center) -- (t02.center);
\draw[arrows={- latex'},dotted,color=red] (00.center) -- (t01.center);

\draw[arrows={- latex},dotted,color=blue] (t03.center) -- (t14.center);
\draw[arrows={- latex'},dotted,color=blue] (t02.center) -- (t13.center);
\draw[arrows={- latex'},dotted,color=blue] (t01.center) -- (t12.center);

\draw[arrows={- latex'},blue] (t14.center) -- (t13.center);
\draw[arrows={- latex'},blue] (t13.center) -- (t12.center);

\draw[arrows={-triangle 45},thick] (4,1.5) -- (5,1.5);


\node at (0.5*\X+3*\D, \X*6) {$\scriptstyle{C_{0}}$};
\draw (0+3*\D,0) rectangle (\X*1+3*\D,\X*1) node[fitting node] (c00) {};
\draw (0+3*\D, \X*1) rectangle (\X*1+3*\D, \X*2) node[fitting node] (c01) {};
\draw (0+3*\D, \X*2) rectangle (\X*1+3*\D, \X*3) node[fitting node] (c02) {};
\draw (0+3*\D, \X*3) rectangle (\X*1+3*\D, \X* 4) node[fitting node] (c03) {};

\node at (1.5*\X+3*\D, \X*6) {$\scriptstyle{C_1}$};
\draw (\X*1+3*\D,\X*1) rectangle (\X*2+3*\D,\X*2) node[fitting node] (c11) {};
\draw (\X*1+3*\D,\X*2) rectangle (\X*2+3*\D,\X*3) node[fitting node] (c12) {};
\draw (\X*1+3*\D,\X*3) rectangle (\X*2+3*\D,\X*4) node[fitting node] (c13) {};
\draw (\X*1+3*\D,\X*4) rectangle (\X*2+3*\D,\X*5) node[fitting node] (c14) {};

\node at (0.5*\X+4*\D, \X*6) {$\scriptstyle{T_{0}}$};
\draw[white] (0+4*\D, \X*4) rectangle (\X*1+4*\D, \X*5) node[fitting node] (tt04) {};

\draw (0+4*\D,0) rectangle (\X*1+4*\D,\X*1) node[fitting node] (tt00) {};
\draw (0+4*\D, \X*1) rectangle (\X*1+4*\D, \X*2) node[fitting node] (tt01) {};
\draw (0+4*\D, \X*2) rectangle (\X*1+4*\D, \X*3) node[fitting node] (tt02) {};
\draw (0+4*\D, \X*3) rectangle (\X*1+4*\D, \X* 4) node[fitting node] (tt03) {};

\node at (1.5*\X+4*\D, \X*6) {$\scriptstyle{T_1}$};
\draw (\X*1+4*\D,\X*1) rectangle (\X*2+4*\D,\X*2) node[fitting node] (tt11) {};
\draw (\X*1+4*\D,\X*2) rectangle (\X*2+4*\D,\X*3) node[fitting node] (tt12) {};
\draw (\X*1+4*\D,\X*3) rectangle (\X*2+4*\D,\X*4) node[fitting node] (tt13) {};
\draw (\X*1+4*\D,\X*4) rectangle (\X*2+4*\D,\X*5) node[fitting node] (tt14) {};

\draw[arrows={- latex'},dotted] (c03.center) -- (tt14.center);
\draw[arrows={- latex'},dotted] (c02.center) -- (tt13.center);
\draw[arrows={- latex'},dotted] (c01.center) -- (tt12.center);
\draw[arrows={- latex'},dotted] (c00.center) -- (tt11.center);


\draw[arrows={-triangle 45},thick] (10,1.5) -- (11,1.5);


\node at (0 + 6*\D+.23, \X*6) {$\scriptscriptstyle{B_{0,0}}$};
\draw (0+6*\D,0) rectangle (\X*1+6*\D,\X*1) node[fitting node] (g00) {};
\draw (0+6*\D, \X*1) rectangle (\X*1+6*\D, \X*2) node[fitting node] (g01) {};
\draw (0+6*\D, \X*2) rectangle (\X*1+6*\D, \X*3) node[fitting node] (g02) {};
\draw (0+6*\D, \X*3) rectangle (\X*1+6*\D, \X* 4) node[fitting node] (g03) {};

\node at (g00.center) {$\scriptstyle{N^+}$};
\node at (g01.center) {$\scriptstyle{N^+}$};
\node at (g02.center) {$\scriptstyle{N^+}$};
\node at (g03.center) {$\scriptstyle{L^-}$};

\node at (2*\X+6*\D-.23, \X*6) {$\scriptscriptstyle{B_{0,1}}$};
\draw (\X*1+6*\D,\X*1) rectangle (\X*2+6*\D,\X*2) node[fitting node] (g11) {};
\draw (\X*1+6*\D,\X*2) rectangle (\X*2+6*\D,\X*3) node[fitting node] (g12) {};
\draw (\X*1+6*\D,\X*3) rectangle (\X*2+6*\D,\X*4) node[fitting node] (g13) {};
\draw (\X*1+6*\D,\X*0) rectangle (\X*2+6*\D,\X*1) node[fitting node] (g14) {};

\node at (g11.center) {$\scriptstyle{N^+}$};
\node at (g12.center) {$\scriptstyle{N^+}$};
\node at (g13.center) {$\scriptstyle{N^+}$};
\node at (g14.center) {$\scriptstyle{N^+}$};

\end{tikzpicture}
\end{center}

By composing the red and blue arrows, we see that the condition $(rT)$ is always satisfied. 
\end{remark}

\begin{defn}
Let $(C,T)$ be a pair of KR-tableaux and let $S=(S_1,\ldots, S_k)$, where $S_l$ is a union of one or more $L$-strips in $(C_l,T_l)$. We say $S$ is \emph{exchangeable} if it is possible to exchange the boxes in $S$ and still obtain valid KR-tableaux, i.e. the resulting new pair $(\tilde{C},\tilde{T})$ have strictly increasing columns and weakly increasing diagonals. $S$ is \emph{minimally exchangeable} if removing any nonempty subset of $L$-strips from $S$ results in a non-exchangeable sequence. 
\end{defn}

\begin{example} Consider $C=(C_1,C_2,C_3,C_4)$ and $T=(T_1,T_2,T_3)$. 

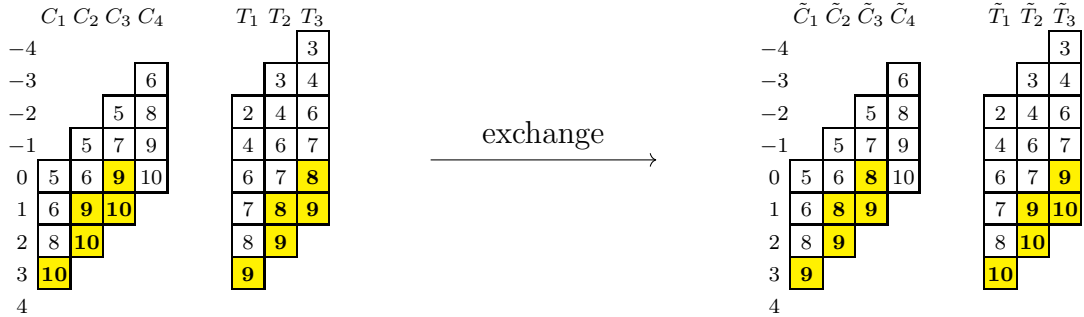
\begin{center}
\begin{tikzpicture}
\node at (0,0) {$
\ytableausetup
{mathmode, boxsize=1em}
\begin{ytableau}
\none[\scriptstyle   ]&	\none[ \scriptstyle C_1]&\none[\scriptstyle C_2]&\none[\scriptstyle C_3]&\none[\scriptstyle C_4]&\none&\none&	\none[\scriptstyle T_1]&\none[\scriptstyle T_2]&\none[\scriptstyle T_3]\\
\none[\scriptstyle -4]&\none&\none&\none&\none&\none&\none&\none &\none&\scriptstyle3	\\
\none[\scriptstyle -3]&\none&\none&\none&\scriptstyle6	   &\none&\none&\none&\scriptstyle3       &\scriptstyle4\\
\none[\scriptstyle -2]&\none&\none&\scriptstyle5	&\scriptstyle8	   &\none&\none&\scriptstyle2       &\scriptstyle4       &\scriptstyle6\\
\none[\scriptstyle -1]&\none&\scriptstyle5	     &\scriptstyle7	&\scriptstyle9 	   &\none&\none	    &\scriptstyle4       &\scriptstyle6       &\scriptstyle7\\
\none[\scriptstyle 0]&\scriptstyle5	  &\scriptstyle6	     &*(yellow)\scriptstyle\mathbf{9}	&\scriptstyle10     &\none&\none	    &\scriptstyle6       &\scriptstyle7       &*(yellow)\scriptstyle\mathbf{8}\\
\none[\scriptstyle 1]&\scriptstyle6	  &*(yellow)\scriptstyle\mathbf{9}	     &*(yellow)\scriptstyle\mathbf{10}	&\none&\none&\none    &\scriptstyle7       &*(yellow)\scriptstyle\mathbf{8}       &*(yellow)\scriptstyle\mathbf{9}\\
\none[\scriptstyle 2]&\scriptstyle8	  &*(yellow)\scriptstyle\mathbf{10}	     &\none&\none&\none&\none	    &\scriptstyle8       &*(yellow)\scriptstyle\mathbf{9}       &\none\\
\none[\scriptstyle 3]&*(yellow)\scriptstyle\mathbf{10}	  &\none&\none&\none&\none&\none&*(yellow)\scriptstyle\mathbf{9}	       &\none&\none\\
\none[\scriptstyle 4]&	\none &\none&\none&\none&\none&\none&\none&\none&\none\\
\end{ytableau}
$};

\draw[->] (3.5,0) -- (6.5,0);
\node at (5,0)[above] {exchange};

\node at (10,0) {$
\begin{ytableau}
\none[\scriptstyle   ]&	\none[ \scriptstyle \tilde{C}_1]&\none[\scriptstyle \tilde{C}_2]&\none[\scriptstyle \tilde{C}_3]&\none[\scriptstyle \tilde{C}_4]&\none&\none&	\none[\scriptstyle \tilde{T}_1]&\none[\scriptstyle \tilde{T}_2]&\none[\scriptstyle \tilde{T}_3]\\
\none[\scriptstyle -4]&\none&\none&\none&\none&\none&\none&\none &\none&\scriptstyle3	\\
\none[\scriptstyle -3]&\none&\none&\none&\scriptstyle6	   &\none&\none&\none&\scriptstyle3       &\scriptstyle4\\
\none[\scriptstyle -2]&\none&\none&\scriptstyle5	&\scriptstyle8	   &\none&\none&\scriptstyle2       &\scriptstyle4       &\scriptstyle6\\
\none[\scriptstyle -1]&\none&\scriptstyle5	     &\scriptstyle7	&\scriptstyle9 	   &\none&\none    &\scriptstyle4       &\scriptstyle6       &\scriptstyle7\\
\none[\scriptstyle 0]&\scriptstyle5	  &\scriptstyle6	     &*(yellow)\scriptstyle\mathbf{8}	&\scriptstyle10     &\none&\none   &\scriptstyle6       &\scriptstyle7       &*(yellow)\scriptstyle\mathbf{9}\\
\none[\scriptstyle 1]&\scriptstyle6	  &*(yellow)\scriptstyle\mathbf{8}	     &*(yellow)\scriptstyle\mathbf{9}&\none&\none&\none	    &\scriptstyle7       &*(yellow)\scriptstyle\mathbf{9}       &*(yellow)\scriptstyle\mathbf{10}\\
\none[\scriptstyle 2]&\scriptstyle8	  &*(yellow)\scriptstyle\mathbf{9}	     &\none&\none&\none&\none    &\scriptstyle8       &*(yellow)\scriptstyle\mathbf{10}       &\none\\
\none[\scriptstyle 3]&*(yellow)\scriptstyle\mathbf{9}	  &\none&\none&\none&\none&\none	    &*(yellow)\scriptstyle\mathbf{10}	       &\none&\none\\
\none[\scriptstyle 4]&	\none &\none&\none&\none&\none&\none    &\none&\none&\none\\
\end{ytableau}
$};

\end{tikzpicture}
\end{center}

A sequence of $L$-strips is shown in bold letters above given by $S_1 = (C_1[3],T_1[3])$, $S_2 = (C_2[1,2],T_2[1,2])$, and $S_3 = (C_3[0,1],T_3[0,1])$. Notice that each $L$-strip is column-compatible, $S_1$ is left-compatible, $S_3$ is right-compatible, and $S_1,S_2$ are not right-compatible. The sequence $S$ is exchangeable since it is possible to exchange the colored boxes and still obtain a pair of valid KR-tableaux. It is also minimal since removing any subset results in a non-exchangeable sequence. For example, exchanging boxes of $S_1,S_2$ without exchanging the boxes of $S_3$ will violate the weakly increasing diagonals condition since $S_2$ is not right-compatible. 
\end{example}

\begin{defn}
Given a left-compatible $L$-strip, we say \emph{it can be completed} to an exchangeable sequence if it is possible to include boxes that, when not exchanged, cause violations of the left or right compatibility conditions and achieve an exchangeable sequence. If, in the process, we end up with an irresolvable contradiction, we say the $L$-strip \emph{cannot be completed to an exchangeable sequence}. 
\end{defn}
Let us demonstrate some situations where an $L$-strip cannot be completed to an exchangeable sequence. 
\begin{example}
In both cases demonstrated below, there is an $L$-strip in $(C_1,T_1)$ (bold), which is not right-compatible. We iteratively add all the boxes that cause right or column compatibility violations. In both cases, we run out of boxes to include before resolving all the violations. 
\begin{center}
\begin{tikzpicture}
\node at (0,0) {$
\begin{ytableau}
\none[\scriptstyle   ]&	\none[ \scriptstyle C_{1}]&\none[\scriptstyle C_2]&\none[\scriptstyle C_3]&\none[\scriptstyle C_4]&\none&\none&	\none[\scriptstyle T_1]&\none[\scriptstyle T_2]&\none[\scriptstyle T_3]&\none[\scriptstyle T_4]\\
\none[\scriptstyle -4]&\none&\none&\none&\none&\none&\none&\none &\none&\none&\scriptstyle4	\\
\none[\scriptstyle -3]&\none&\none&\none&\scriptstyle4	   &\none&\none&\none&\none&2       &\scriptstyle5\\
\none[\scriptstyle -2]&\none&\none&*(yellow)\scriptstyle\mathbf{4}	&\scriptstyle6	   &\none&\none&\none&\scriptstyle2       &*(yellow)\scriptstyle\mathbf{3}       &\scriptstyle7\\
\none[\scriptstyle -1]&\none&*(yellow)\scriptstyle\mathbf{4}	     &\scriptstyle6	&\scriptstyle7 	   &\none&\none&\scriptstyle2	    &*(yellow)\scriptstyle\mathbf{3}       &\scriptstyle7       &\scriptstyle8\\
\none[\scriptstyle 0]&*(yellow)\scriptstyle\mathbf{4}	  &*(yellow)\scriptstyle\mathbf{6}	     &\scriptstyle7	&\scriptstyle8     &\none&\none&*(yellow)\scriptstyle\mathbf{3}	    &*(yellow)\scriptstyle\mathbf{4}       &\scriptstyle8       &\scriptstyle9\\
\none[\scriptstyle 1]&*(yellow)\scriptstyle\mathbf{5}	  &*(yellow)\scriptstyle\mathbf{7}	     &\scriptstyle8	&\none&\none&\none&*(yellow)\scriptstyle\mathbf{4}	    &*(yellow)\scriptstyle\mathbf{6}       &\scriptstyle9       &\scriptstyle10\\
\none[\scriptstyle 2]&\scriptstyle6  &*(yellow)\scriptstyle\mathbf{8}     &\none&\none&\none&\none&\scriptstyle6	    &*(yellow)\scriptstyle\mathbf{7}       &\scriptstyle10       &\none\\
\none[\scriptstyle 3]&\scriptstyle7	  &\none&\none&\none&\none&\none&\scriptstyle7	    &*(red)\scriptstyle\mathbf{8}	       &\none&\none\\
\none[\scriptstyle 4]&	\none &\none&\none&\none&\none&\none&\scriptstyle8	    &\none&\none&\none\\
\end{ytableau}
$};

\node at (0,-3) {$(C_2,T_2)$ is not column-compatible};

\node at (8,0) {$
\begin{ytableau}
\none[\scriptstyle   ]&	\none[ \scriptstyle C_1]&\none[\scriptstyle C_2]&\none[\scriptstyle C_3]&\none[\scriptstyle C_4]&\none&\none&	\none[\scriptstyle T_1]&\none[\scriptstyle T_2]&\none[\scriptstyle T_3]&\none[\scriptstyle T_4]\\
\none[\scriptstyle -4]&\none&\none&\none&\none&\none&\none&\none &\none&\none&\scriptstyle4	\\
\none[\scriptstyle -3]&\none&\none&\none&\scriptstyle4	   &\none&\none&\none&\none&*(red)\scriptstyle\mathbf{4}       &\scriptstyle5\\
\none[\scriptstyle -2]&\none&\none&*(yellow)\scriptstyle\mathbf{4}	&\scriptstyle6	   &\none&\none&\none&\scriptstyle3       &*(yellow)\scriptstyle\mathbf{5}       &\scriptstyle7\\
\none[\scriptstyle -1]&\none&*(yellow)\scriptstyle\mathbf{4}	     &\scriptstyle6	&\scriptstyle7 	   &\none&\none&\scriptstyle3	    &*(yellow)\scriptstyle\mathbf{5}       &\scriptstyle7       &\scriptstyle8\\
\none[\scriptstyle 0]&*(yellow)\scriptstyle\mathbf{4}	  &\scriptstyle6	     &\scriptstyle7	&\scriptstyle8     &\none&\none&*(yellow)\scriptstyle\mathbf{5}	    &\scriptstyle6       &\scriptstyle8       &\scriptstyle9\\
\none[\scriptstyle 1]&*(yellow)\scriptstyle\mathbf{5}	  &\scriptstyle7	     &\scriptstyle8	&\none&\none&\none&*(yellow)\scriptstyle\mathbf{6}	    &\scriptstyle7       &\scriptstyle9       &\scriptstyle10\\
\none[\scriptstyle 2]&	*(yellow)\scriptstyle\mathbf{6}  &\scriptstyle8     &\none&\none&\none&\none&*(yellow)\scriptstyle\mathbf{7}	    &\scriptstyle8       &\scriptstyle10       &\none\\
\none[\scriptstyle 3]&*(yellow)\scriptstyle\mathbf{7}	  &\none&\none&\none&\none&\none&*(yellow)\scriptstyle\mathbf{8}	    &\scriptstyle10	       &\none&\none\\
\none[\scriptstyle 4]&	\none &\none&\none&\none&\none&\none&\scriptstyle10	    &\none&\none&\none\\
\end{ytableau}
$};
\node at (8.5,-3) {$(C_3,T_3)$ is not column-compatible};
\end{tikzpicture}
\end{center}

\end{example}

\begin{defn}
Let $(C,T)$ be a pair of KR-tableaux such that no $L$-strip can be completed to an exchangeable sequence. Then we say $(C,T)$ \emph{has no exchangeable sequences}. 
\end{defn}

We now describe the the map $\sigma$. Consider the set $\mathcal{B}_{k,j}^{(i)} \times \mathcal{B}_{k',j'}^{(i')}$ for any $(i,j,k)$ and $(i',j',k')$. We define 
\begin{eqnarray*}
\mathcal{P}_0 := \left\{ (C,T) \big| C\in \mathcal{B}_{k,j}^{(i)}, T\in \mathcal{B}_{k',j'}^{(i')}, \text{ and } (C,T) \text{ has no exchangeable sequences}\right\}
\end{eqnarray*}
Let $(C,T) \in \mathcal{B}_{k,j}^{(i)} \times \mathcal{B}_{k',j'}^{(i')} \backslash \mathcal{P}_0$. By definition, $(C,T)$ has an exchangeable sequence. Suppose $(C',T') \in \mathcal{B}_{k,j}^{(i)} \times \mathcal{B}_{k',j'}^{(i')} \backslash \mathcal{P}_0$ such that $(C',T')$ is obtained from $(C,T)$ by exchanging a minimally exchangeable sequence. Then we assign $(C,T)$ to $\mathcal{P}_1$, $(\tilde{C},\tilde{T})$ to $\mathcal{P}_{-1}$, and define $\sigma(C,T) = (\tilde{C},\tilde{T})$. The exact order is not important. All that matters is that we can partition $\mathcal{B}_{k,j}^{(i)} \times \mathcal{B}_{k',j'}^{(i')} \backslash \mathcal{P}_0$ into $2$ disjoint subsets. Since the elements in $\mathcal{P}_0$ do not have exchangeable sequences, we define $\sigma(C,T) = (C,T)$ for all $(C,T)\in \mathcal{P}_0$.

\subsection{Proof of Theorem \ref{thm:main1}} \label{proofOfMain1}
Let $C=(C_i)$ and $T=(T_j)$ be a pair of KR-tableaux with no exchangeable sequences. We want to show $\gamma(C,T)=0$. Recall from Remark \ref{rmk:gammaAdditive} that:
\begin{eqnarray*}
\gamma(C,T) = \sum_{i,j}\gamma(C_i,T_j)\,\, ,
\end{eqnarray*} 
and by Lemma \ref{gamma}:
\begin{eqnarray*}
\gamma(C_i,T_j) = \sum_{p=h}^{t} L_p (C_i,T_j) + \left\{ \begin{array}{cc} 
N_{t+1}(C_i,T_j) &\text{ if } (C_i,T_j) \text{ is (anti-)fundamental}\\
L_{t+1}(C_i,T_j) &\text{ otherwise}
\end{array}\right.  \,\, ,
\end{eqnarray*} 
where $h=\head(C,T)$ and $t=\tail(C,T)$. 

Let us prove first the theorem in case when both $C$ and $T$ are column tableaux. Recall that column diagrams correspond to fundamental modules ($k=1$) and the fundamental modules in the fundamental cluster $\mathcal{C}$ form (anti-)fundamental pairs (see Section \ref{fundClusterDiagrams} and Definition \ref{def:pairs}).

\begin{lemma} \label{gammaContr}
Let $(C,T)$ be a pair of column KR-tableaux of type (anti-)fundamental with no exchangeable sequences, i.e. $(C,T)\in\mathcal{P}_0$. Then $\gamma(C,T)=0$. 
\end{lemma}
\begin{proof}
Suppose $(C,T)$ admits an $L$-strip. Notice that an $L$-strip in $(C,T)$ is exchangeable if and only if it is column-compatible since there are no columns to the left or to the right of both $C$ and $T$. Since $(C,T)$ has no exchangeable sequences, there can be exactly one $L$-strip. Otherwise, by Remark \ref{cross-pattern}, the $L$-strips other than the last one are all column-compatible, and, therefore, exchangeable. 

We may assume the $L$-strip in question is an $L^-$-strip. If not, we simply consider $(T,C)$ instead, where $L^+$-strips in $(C,T)$ become $L^-$-strips in $(T,C)$. Let $p^*$ be the index of the head of the $L$-strip. Then,
\begin{eqnarray*}
\gamma(C,T) = \underbrace{L_{p^*}(C,T)}_\text{-1} + N_{t+1}(C,T)\,\, ,
\end{eqnarray*}
where $t=\tail(C,T)$. Since $L_{p^*}$ is not column-compatible, it is not followed by an $LU$-block. That is, it must be followed by $N$-blocks only. By Remark \ref{rmk:obs1}(a,b), an $L^-$-block can only be followed by an $N^+$, and $N^+$-blocks cannot be followed by $N^-$-blocks. Therefore, we must have that $N_{t+1}(C,T) = +1$ and $\gamma(C,T) = 0$. 

Suppose $(C,T)$ has no $L$-strips. We can assume $(C,T)$ is a fundamental pair. If not, we simply consider $(T,C)$ instead. If the boundary term is zero, that is $N_{t+1}(C,T)=0$, then $\gamma(C,T)=0$ and the result holds. Suppose the boundary term is not zero. By Remark \ref{rmk:obs2}(b), since $(C,T)$ is a fundamental pair, the last block is never $N^+$. Therefore, $N_{t+1}(C,T)=-1$. By Remark \ref{rmk:obs1}(a,b), an $N^-$-block can only be preceded by $L^+$ or $N^-$. However, there are no $L$-blocks in $(C,T)$. So, there must be only $N^-$-blocks in $(C,T)$. By Remark \ref{rmk:obs2}(a), the very first block cannot be $N^-$. Contradiction. This concludes the proof. 
\end{proof}

\begin{cor}
Non-column-compatible $L$-strips in (anti-)fundamental pairs do not contribute to $\gamma$ since their contribution is always canceled out by the boundary term. 
\end{cor}
Our strategy in proving Theorem \ref{thm:main1} is to find a way to systematically cancel contributions from $L_p$ to $\gamma$ when the sequence cannot be completed to an exchangeable sequence. In the lemmas that follow, we always have $C=(C_0,C_1)$ and $T=(T_0,T_1)$ a pair of KR-tableaux and $(C_0,T_0)$ has a non-exchangeable $L^-$-strip. The goal is to find a unique non-exchangeable $L^+$-strip in $(C,T)$ to cancel out the contribution of $L^-$. Moreover, when the $L^-$ does not contribute to $\gamma$ (as in Lemma \ref{gammaContr}), we want to show that there is no such corresponding $L^+$. We make an exhaustive list of all the possibilities for $L^-$ to be non-exchangeable and give the lemma that address the situation as a reference in Table \ref{tab:cancellations}.

\begin{table}[h]
\begin{tabular}{ |c|c|c|c|c|c|} 
\hline
$(C_0,T_0)$ type& \begin{tabular}{@{}c@{}}$L^-$ in $(C_0,T_0)$ is\\right-compatible\end{tabular} &  \begin{tabular}{@{}c@{}}$L^-$ in $(C_0,T_0)$ is\\column-compatible\end{tabular}&Contributes to $\gamma$&Lemma\\
 \hline\hline
any type & No & Yes &Yes& \ref{chain1}\\ 
 \hline
 (anti-)regular & Yes & No &Yes& \ref{chain2},\ref{chain3}\\ 
 \hline
 (anti-)fundamental & No & No &No& \ref{chain4}\\ 
 \hline
 (anti-)fundamental & Yes & No &No& \ref{chain5}\\ 
 \hline
 \end{tabular}
 \caption{List of non-exchangeable $L^-$-strips}
 \label{tab:cancellations}
\end{table}

\begin{example} \label{ex:chain1}
We demonstrate the most common type of cancellation in Figure~\ref{fig:cancel}. There is $L^-$-block in $(C_0,T_0)$ that is column-compatible, but not right-compatible given by $1<2<3$ (colored, left), which contributes $-1$ to $\gamma$. It is not right-compatible due to $2$ in $C_1$ since the weakly increasing diagonals condition will be violated post-exchange of the $L^-$-strip. However, there is a non-left-compatible $L^+$-block in $(C_1,T_0)$ given by $2<3<4$ (colored, right). The $L^+$-strip is not left-compatible due to $4$ in $C_0$. That is, we found a non-exchangeable $+1$ contribution to $\gamma$ to cancel out the previous $-1$.

\begin{figure}[h]
\begin{center}
\begin{tikzpicture}

\node at (0.5*\X, \X*6) {$\scriptstyle{C_{0}}$};
\draw (0,0) rectangle (\X*1,\X*1) node[fitting node] (00) {};
\draw (0, \X*1) rectangle (\X*1, \X*2) node[fitting node] (01) {};
\draw[fill=yellow] (0, \X*2) rectangle (\X*1, \X*3) node[fitting node] (02) {};
\draw (0, \X*3) rectangle (\X*1, \X* 4) node[fitting node] (03) {};

\node at (03.center) {$\scriptstyle 1$};
\node at (02.center) {$\scriptstyle 2$};
\node at (01.center) {$\scriptstyle 4$};
\node at (00.center) {$\scriptstyle 5$};

\node at (1.5*\X, \X*6) {$\scriptstyle{C_1}$};
\draw (\X*1,\X*1) rectangle (\X*2,\X*2) node[fitting node] (11) {};
\draw (\X*1,\X*2) rectangle (\X*2,\X*3) node[fitting node] (12) {};
\draw[fill=red] (\X*1,\X*3) rectangle (\X*2,\X*4) node[fitting node] (13) {};
\draw (\X*1,\X*4) rectangle (\X*2,\X*5) node[fitting node] (14) {};

\node at (14.center) {$\scriptstyle 1$};
\node at (13.center) {$\scriptstyle 2$};
\node at (12.center) {$\scriptstyle 4$};
\node at (11.center) {$\scriptstyle 5$};

\node at (0.5*\X+\D, \X*6) {$\scriptstyle{T_{0}}$};
\draw[white] (0+\D, \X*4) rectangle (\X*1+\D, \X*5) node[fitting node] (t04) {};

\draw (0+\D,0) rectangle (\X*1+\D,\X*1) node[fitting node] (t00) {};
\draw (0+\D, \X*1) rectangle (\X*1+\D, \X*2) node[fitting node] (t01) {};
\draw[fill=yellow] (0+\D, \X*2) rectangle (\X*1+\D, \X*3) node[fitting node] (t02) {};
\draw (0+\D, \X*3) rectangle (\X*1+\D, \X* 4) node[fitting node] (t03) {};

\node at (t03.center) {$\scriptstyle 1$};
\node at (t02.center) {$\scriptstyle 3$};
\node at (t01.center) {$\scriptstyle 5$};
\node at (t00.center) {$\scriptstyle 6$};

\draw[arrows={- latex'},red] (t03.west) -- (02.east);
\draw[arrows={- latex'},red] (02.east) -- (t02.west);

\draw[arrows={-triangle 45},thick] (4,1.5) -- (5,1.5);

\node at (0.5*\X+3*\D, \X*6) {$\scriptstyle{C_{0}}$};
\draw (0+3*\D,0) rectangle (\X*1+3*\D,\X*1) node[fitting node] (c00) {};
\draw[fill=red] (0+3*\D, \X*1) rectangle (\X*1+3*\D, \X*2) node[fitting node] (c01) {};
\draw (0+3*\D, \X*2) rectangle (\X*1+3*\D, \X*3) node[fitting node] (c02) {};
\draw (0+3*\D, \X*3) rectangle (\X*1+3*\D, \X* 4) node[fitting node] (c03) {};

\node at (c03.center) {$\scriptstyle 1$};
\node at (c02.center) {$\scriptstyle 2$};
\node at (c01.center) {$\scriptstyle 4$};
\node at (c00.center) {$\scriptstyle 5$};

\node at (1.5*\X+3*\D, \X*6) {$\scriptstyle{C_1}$};
\draw (\X*1+3*\D,\X*1) rectangle (\X*2+3*\D,\X*2) node[fitting node] (c11) {};
\draw[fill=yellow] (\X*1+3*\D,\X*2) rectangle (\X*2+3*\D,\X*3) node[fitting node] (c12) {};
\draw (\X*1+3*\D,\X*3) rectangle (\X*2+3*\D,\X*4) node[fitting node] (c13) {};
\draw (\X*1+3*\D,\X*4) rectangle (\X*2+3*\D,\X*5) node[fitting node] (c14) {};

\node at (c14.center) {$\scriptstyle 1$};
\node at (c13.center) {$\scriptstyle 2$};
\node at (c12.center) {$\scriptstyle 4$};
\node at (c11.center) {$\scriptstyle 5$};

\node at (0.5*\X+4*\D, \X*6) {$\scriptstyle{T_{0}}$};
\draw[white] (0+4*\D, \X*4) rectangle (\X*1+4*\D, \X*5) node[fitting node] (tt04) {};

\draw (0+4*\D,0) rectangle (\X*1+4*\D,\X*1) node[fitting node] (tt00) {};
\draw (0+4*\D, \X*1) rectangle (\X*1+4*\D, \X*2) node[fitting node] (tt01) {};
\draw[fill=yellow] (0+4*\D, \X*2) rectangle (\X*1+4*\D, \X*3) node[fitting node] (tt02) {};
\draw (0+4*\D, \X*3) rectangle (\X*1+4*\D, \X* 4) node[fitting node] (tt03) {};

\node at (tt03.center) {$\scriptstyle 1$};
\node at (tt02.center) {$\scriptstyle 3$};
\node at (tt01.center) {$\scriptstyle 5$};
\node at (tt00.center) {$\scriptstyle 6$};

\draw[arrows={- latex'},red] (c13.east) -- (tt02.west);
\draw[arrows={- latex'},red] (tt02.west) -- (c12.east);
\end{tikzpicture}
\end{center}
\caption{Cancellation of non-right-compatible $L^-$-strips with non-left-compatible $L^+$-strips}
\label{fig:cancel}
\end{figure}
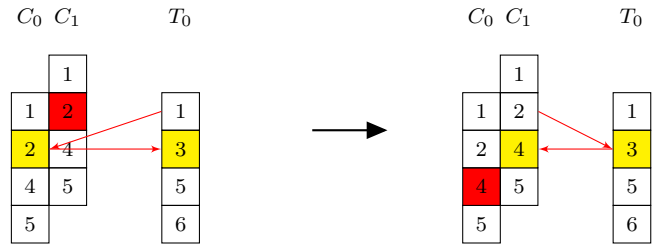

\end{example}

\begin{remark}\label{rmk:obs3}

Suppose $C=(C_0,C_1)$ and $T_0$ are KR-tableaux. Suppose there is an $LU$-block at $p$ in $(C_0,T_0)$. Then there is an $L^+N^-$-block at $p-1$ in $(C_1,T_0)$ (see Figure~\ref{fig:obs3}). Indeed, 
\begin{figure}[h]
\begin{center}
\begin{tikzpicture}

\node at (0,0) {
\begin{tikzpicture}
\node at (0.5*\X, \X*5) {$\scriptstyle{C_0}$};

\draw (0, \X*2) rectangle (\X*1, \X*3) node[fitting node] (02) {};
\draw (0, \X*3) rectangle (\X*1, \X* 4) node[fitting node] (03) {};
\node at ($(02) +(-.6,0)$) {$\scriptstyle p$};
\node at ($(03) +(-.6,0)$) {$\scriptstyle p-1$};

\node at (\X*1.5, \X*5) {$\scriptstyle{C_1}$};
\draw (\X*1,\X*2) rectangle (\X*2,\X*3) node[fitting node] (12) {};
\draw (\X*1,\X*3) rectangle (\X*2,\X*4) node[fitting node] (13) {};

\node at (0.5*\X+\D, \X*5) {$\scriptstyle{T_0}$};
\draw (0+\D, \X*2) rectangle (\X*1+\D, \X*3) node[fitting node] (t02) {};
\draw (0+\D, \X*3) rectangle (\X*1+\D, \X* 4) node[fitting node] (t03) {};

\node at (0.5*\X+\D+1,\X*5) {$\scriptstyle{B_{0,0}}$};
\draw (0+\D+1, \X*2) rectangle (\X*1+\D+1, \X*3) node[fitting node] (b02) {};
\draw (0+\D+1, \X*3) rectangle (\X*1+\D+1, \X* 4) node[fitting node] (b03) {};

\node at (0.5*\X+\D+1.8,\X*5) {$\scriptstyle{B_{1,0}}$};
\draw (\X*1+\D+1, \X*3) rectangle (\X*1+\D+2, \X* 4) node[fitting node] (b13) {};

\node at (b02.center) {$\scriptscriptstyle L^-$};
\node at (b13.center) {$\scriptscriptstyle L^+N^-$};

\draw[arrows={ - latex'},dotted,blue] (02.center) -- (13.center);
\draw[arrows={ - latex'},blue] (t03.center) -- (02.center);
\draw[arrows={ - latex'},red,thick] (t03.center) -- (13.center);

\end{tikzpicture}
};

\node at (6,0) {
\begin{tikzpicture}
\node at (0.5*\X, \X*5) {$\scriptstyle{C_0}$};

\draw (0, \X*2) rectangle (\X*1, \X*3) node[fitting node] (02) {};
\draw (0, \X*3) rectangle (\X*1, \X* 4) node[fitting node] (03) {};
\node at ($(02) +(-.6,0)$) {$\scriptstyle p$};
\node at ($(03) +(-.6,0)$) {$\scriptstyle p-1$};

\node at (\X*1.5, \X*5) {$\scriptstyle{C_1}$};
\draw (\X*1,\X*2) rectangle (\X*2,\X*3) node[fitting node] (12) {};
\draw (\X*1,\X*3) rectangle (\X*2,\X*4) node[fitting node] (13) {};

\node at (0.5*\X+\D, \X*5) {$\scriptstyle{T_0}$};
\draw (0+\D, \X*2) rectangle (\X*1+\D, \X*3) node[fitting node] (t02) {};
\draw (0+\D, \X*3) rectangle (\X*1+\D, \X* 4) node[fitting node] (t03) {};

\node at (0.5*\X+\D+1,\X*5) {$\scriptstyle{B_{0,0}}$};
\draw (0+\D+1, \X*2) rectangle (\X*1+\D+1.1, \X*3) node[fitting node] (b02) {};
\draw (0+\D+1, \X*3) rectangle (\X*1+\D+1.1, \X* 4) node[fitting node] (b03) {};

\node at (0.5*\X+\D+1.8,\X*5) {$\scriptstyle{B_{1,0}}$};
\draw (\X*1+\D+1.1, \X*3) rectangle (\X*1+\D+2, \X* 4) node[fitting node] (b13) {};

\node at (b02.center) {$\scriptscriptstyle L^+U$};
\node at (b13.center) {$\scriptscriptstyle L^+N^-$};

\draw[arrows={ - latex'},dotted,blue] (02.center) -- (13.center);
\draw[arrows={ - latex'},blue,dotted] (t02.center) -- (02.center);
\draw[arrows={ - latex'},blue] (t03.center) -- (t02.center);
\draw[arrows={ - latex'},red,thick] (t03.center) -- (13.center);

\end{tikzpicture}
};
\end{tikzpicture}
\end{center}
\caption{Composition of arrows.}
\label{fig:obs3}
\end{figure}
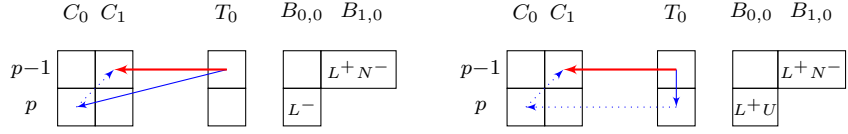
by composing arrows, we find $T_0[p-1]<C_1[p-1]$, which means there cannot be an $L^-U$-block in $(C_1,T_0)$ at $p-1$. Also, there cannot be an $N^+$-block since $T_0[p-2]<T_0[p-1] < C_1[p-1]$. Thus, there can be either $L^+$ or an $N^-$ block at $p-1$ in $(C_1,T_0)$. 
\end{remark}
This next lemma is the main cancellation action demonstrated in Example \ref{ex:chain1}. 
\begin{lemma} \label{chain1}
Let $C=(C_0,C_1)$ and $T=(T_0,T_1)$ be a pair of KR-tableaux. To every right-incompatible $L^-$-strip in $(C_0,T_0)$ that contributes to $\gamma$, there exists a unique left-incompatible $L^+$-strip in $(C_1,T_0)$ that contributes to $\gamma$. All other $L^+$-strips in $(C_1,T_0)$ are left-compatible. 
\end{lemma}

\begin{proof}
Without loss of generality, we may assume $(C_0,T_0)$ is either fundamental or regular. If not, we simply consider $(T,C)$ which is of the desired type. As before, denote the block tableaux of $(C_i,T_j)$ by $B_{i,j}$. 

We will consider the case of the very last non-right-compatible $L^-$-strip separately at the end of the proof. Let $\tilde{L}^-$ be not the very last one, and let $p'$ be its index. By Lemma \ref{automaticComp}, $\tilde{L}^-$ satisfies the condition $(rT)$, i.e. $T_1$ does not pose violations. Since $\tilde{L}^-$ is not right-compatible, the condition $(rC)$ must fail. In other words, $C_1$ must pose a right-compatibility violation.  Since $\tilde{L}^-$ is not the last $L^-$, it is followed by an $L$-block or a $U$-block (whichever one comes first). Let $p''$ be the index of the $LU$-block that follows $\tilde{L}^-$. 
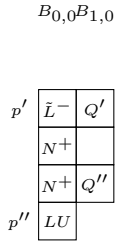
\begin{center}
\begin{tikzpicture}

\node at (0.5*\X, \X*6) {$\scriptscriptstyle{B_{0,0}}$};
\draw (0,0) rectangle (\X*1,\X*1) node[fitting node] (00) {};
\node at (00.center) {$\scriptscriptstyle{LU}$};
\node at ($(00.center) + (-.5,0)$) {$\scriptscriptstyle p''$};
\draw (0, \X*1) rectangle (\X*1, \X*2) node[fitting node] (01) {};
\draw (0, \X*2) rectangle (\X*1, \X*3) node[fitting node] (02) {};
\draw (0, \X*3) rectangle (\X*1, \X* 4) node[fitting node] (03) {};
\node at (03.center) {$\scriptscriptstyle{\tilde{L}^-}$};
\node at ($(03.center) + (-.5,0)$) {$\scriptscriptstyle p'$}; 
\node at (02.center) {$\scriptscriptstyle{N^+}$};
\node at (01.center) {$\scriptscriptstyle{N^+}$};

\node at (1.5*\X, \X*6) {$\scriptscriptstyle{B_{1,0}}$};
\draw (\X*1,\X*1) rectangle (\X*2,\X*2) node[fitting node] (11) {};
\node at (11.center) {$\scriptscriptstyle{Q''}$};
\draw (\X*1,\X*2) rectangle (\X*2,\X*3) node[fitting node] (12) {};
\draw (\X*1,\X*3) rectangle (\X*2,\X*4) node[fitting node] (13) {};
\node at (13.center) {$\scriptscriptstyle{Q'}$};

\node at (\X*1+4*\D, \X*3) {\begin{minipage}{0.6\textwidth}
Denote the blocks in $(C_1,T_0)$ at indices $p'$ and $p''-1$ by $Q'$ and $Q''$ respectively. By Remark \ref{rmk:obs3}, $Q''$ is either $L^+$ or $N^-$. If $Q''=L^+$, we have a candidate. Suppose not, i.e. $Q''=N^-$. Then, by Remark \ref{rmk:obs1} it can only be preceded by $L^+$ or $N^-$.  If we don't allow any $L^+$ between $Q'$ and $Q''$, then Condition $(rC)$ is satisfied. Contradiction. Therefore, there must be at least one $L^+$-block between $Q'$ and $Q''$. \end{minipage}};
\end{tikzpicture}
\end{center}
If there is more than one $L^+$-block, let $\tilde{L}^+$ be the one with the largest index, i.e. closest to $Q''$. Then $\tilde{L}^+$ is not left-compatible due to the $LU$-block in $(C_0,T_0)$. Moreover, all other $L^+$-strips between $Q'$ and $Q''$ are left-compatible since $(lT)$ is satisfied (seen from picture) and $(lC)$ is always satisfied for $L^+$-strips.

It is not clear that the $\tilde{L}^+$ we found contributes to $\gamma$. Suppose it does not. There is exactly one situation where an $L$-block does not contribute to $\gamma$. This can happen only if $(C_1,T_0)$ is fundamental and $\tilde{L}^+$ is the last $L$-block followed by $N^-$'s. The boundary term of $\gamma$ is then $-1$, which cancels out the contribution of $\tilde{L}^+$. Then we have the following picture: 
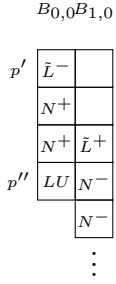
\begin{center}
\begin{tikzpicture}

\node at (0.5*\X, \X*5) {$\scriptscriptstyle{B_{0,0}}$};
\draw (0,0) rectangle (\X*1,\X*1) node[fitting node] (00) {};
\node at (00.center) {$\scriptscriptstyle{LU}$};
\node at ($(00.center) + (-.5,0)$) {$\scriptscriptstyle p''$};

\draw (0, \X*1) rectangle (\X*1, \X*2) node[fitting node] (01) {};
\draw (0, \X*2) rectangle (\X*1, \X*3) node[fitting node] (02) {};
\draw (0, \X*3) rectangle (\X*1, \X* 4) node[fitting node] (03) {};
\node at (03.center) {$\scriptscriptstyle{\tilde{L}^-}$};
\node at ($(03.center) + (-.5,0)$) {$\scriptscriptstyle p'$};

\node at (02.center) {$\scriptscriptstyle{N^+}$};
\node at (01.center) {$\scriptscriptstyle{N^+}$};

\draw[white] (\X*1,-2*\X) rectangle (\X*2,-1*\X) node[fitting node] (1-2) {};
\node at (1.5*\X, \X*5) {$\scriptscriptstyle{B_{1,0}}$};
\draw (\X*1,\X*1) rectangle (\X*2,\X*2) node[fitting node] (11) {};
\node at (11.center) {$\scriptscriptstyle{\tilde{L}^+}$};
\draw (\X*1,\X*0) rectangle (\X*2,\X*1) node[fitting node] (10) {};
\node at (10.center) {$\scriptscriptstyle{N^-}$};
\draw (\X*1,-1*\X) rectangle (\X*2,\X*0) node[fitting node] (1-1) {};
\node at (1-1.center) {$\scriptscriptstyle{N^-}$};
\node at (1-2.center) {$\scriptscriptstyle{\vdots}$};

\draw (\X*1,\X*2) rectangle (\X*2,\X*3) node[fitting node] (12) {};
\draw (\X*1,\X*3) rectangle (\X*2,\X*4) node[fitting node] (13) {};

\node at (\X*1+4*\D, \X*3) {\begin{minipage}{0.6\textwidth}
All $L^-$-strips in $(C_0,T_0)$ appearing at any indices $p\geq p''$ satisfy Condition $(rC)$ since $B_{1,0}$ consists of $N^-$-blocks only for all $p \geq p''$, and, therefore, are right-compatible. This means $\tilde{L}^-$ is the very last non-right-compatible $L^-$-strip, which is a contradiction. This concludes the proof for this case.  
 \end{minipage}};
\end{tikzpicture}
\end{center}
Now let $\tilde{L}^-$  be the very last $L^-$-strip that contributes to $\gamma$ and is not right-compatible. There are exactly two situations: 
\begin{enumerate}
\item The situation described above, i.e. $\tilde{L}^-$ is followed by $LU$, but the corresponding $\tilde{L}^+$ in $(C_1,T_0)$ does not contribute to $\gamma$. This happens when $(C_1,T_0)$ is fundamental. 
\item $\tilde{L}^-$ is not followed by $LU$ and $(C_0,T_0)$ is regular. 
\end{enumerate}
In situation (1), the $\tilde{L}^+$ we found did not contribute to $\gamma$ and it is the last $L^+$ in $(C_1,T_0)$. Therefore, we must look for the appropriate $L^+$ elsewhere in $(C_1,T_0)$. In situation (2), since $\tilde{L}^-$ is the absolute last $L^-$-strip that is not followed by $LU$, it is followed by $N^+$-blocks only. This means that any $L^+$ in $(C_1,T_0)$ adjacent to the strip associated to $\tilde{L}^-$ is left-compatible. Since we are looking for a non-left-compatible $L^+$ to pair with $\tilde{L}^-$, we must also look elsewhere in $(C_1,T_0)$. 

We now consider both situations. Let $\hat{L}^-$ be the very first $L^-$-strip in $(C_0,T_0)$, and let $p'$ be its index. Denote the first block in $(C_1,T_0)$ by $Q_0$ and the block at index $p'-1$ in $(C_1,T_0)$ by $Q_1$. 
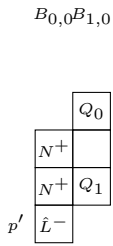
\begin{center}
\begin{tikzpicture}

\node at (0.5*\X, \X*6) {$\scriptscriptstyle{B_{0,0}}$};
\draw (0,0) rectangle (\X*1,\X*1) node[fitting node] (00) {};
\node at (00.center) {$\scriptscriptstyle{\hat{L}^-}$};
\node at ($(00.center) + (-.5,0)$) {$\scriptscriptstyle p'$};

\draw (0, \X*1) rectangle (\X*1, \X*2) node[fitting node] (01) {};
\draw (0, \X*2) rectangle (\X*1, \X*3) node[fitting node] (02) {};
\node at (01.center) {$\scriptscriptstyle{N^+}$};
\node at (02.center) {$\scriptscriptstyle{N^+}$};

\node at (1.5*\X, \X*6) {$\scriptscriptstyle{B_{1,0}}$};
\draw (\X*1,\X*1) rectangle (\X*2,\X*2) node[fitting node] (11) {};
\node at (11.center) {$\scriptscriptstyle{Q_1}$};
\draw (\X*1,\X*2) rectangle (\X*2,\X*3) node[fitting node] (12) {};
\draw (\X*1,\X*3) rectangle (\X*2,\X*4) node[fitting node] (13) {};
\node at (13.center) {$\scriptscriptstyle{Q_{0}}$};

\node at (\X*1+4*\D, \X*3) {\begin{minipage}{0.75\textwidth}
By Remark \ref{rmk:obs2}, the first block in $B_{1,0}$, i.e. $Q_{0}$, is not $N^-$. By Remark \ref{rmk:obs3}, $Q_1$ is either $L^+$ or $N^-$. The only way to transition from non $N^-$-block to $N^-$-block is through $L^+$. Therefore, there must be at least one $L^+$-block between $Q_{0}$ and $Q_1$. Let $\hat{L}^+$ be the $L^+$ closest to $Q_1$. Then the $\hat{L}^+$-strip is not left-compatible due to $\hat{L}^-$ in $(C_0,T_0)$ and contributes to $\gamma$ since it is not the last $L$-block in $(C_1,T_0)$. All other $L^+$ above $\hat{L}^+$, if they exist, are left-compatible as seen from the picture. Notice that it is possible to have $U$-blocks above $\hat{L}^-$. Then we replace $\hat{L}^-$ with the very first $U$ and argue as before.\end{minipage}};
\end{tikzpicture}
\end{center}
\end{proof}

Let us put everything together in Figure~\ref{fig:chain1} to emphasize the fact that we have unique pairing of non-right-compatible $L^-$-strips with non-left-compatible $L^+$-strips. 
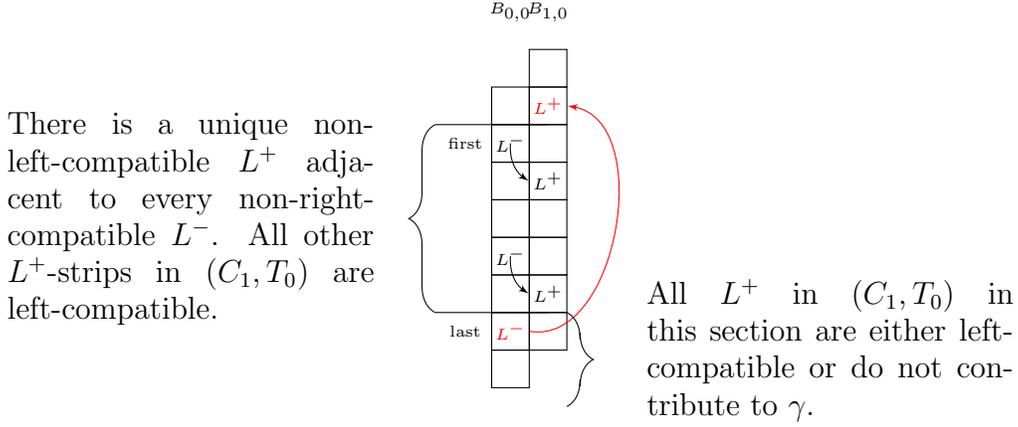
\begin{figure}[h]
\begin{tikzpicture}

\node at (0.5*\X, \X*6) {$\scriptscriptstyle{B_{0,0}}$};
\draw (0,0) rectangle (\X*1,\X*1) node[fitting node] (00) {};

\draw (0,-1*\X) rectangle (\X*1,0*\X) node[fitting node] (0-1) {};
\node at (0-1.center) {$\scriptscriptstyle{L^-}$};

\draw (0,-2*\X) rectangle (\X*1,-1*\X) node[fitting node] (0-2) {};
\draw (0,-3*\X) rectangle (\X*1,-2*\X) node[fitting node] (0-3) {};
\node at (0-3.center) {${\color{red}\scriptscriptstyle{L^-}}$};
\node at ($(0-3)-(0.6,0)$) {$\scriptscriptstyle{\text{last}}$};
\draw (0,-4*\X) rectangle (\X*1,-3*\X) node[fitting node] (0-4) {};
\draw (0, \X*1) rectangle (\X*1, \X*2) node[fitting node] (01) {};
\draw (0, \X*2) rectangle (\X*1, \X*3) node[fitting node] (02) {};
\node at (02.center) {$\scriptscriptstyle{L^-}$};
\node at ($(02)-(0.6,0)$) {$\scriptscriptstyle{\text{first}}$};

\draw (0, \X*3) rectangle (\X*1, \X* 4) node[fitting node] (03) {};

\node at (1.5*\X, \X*6) {$\scriptscriptstyle{B_{1,0}}$};
\draw (\X*1,\X*1) rectangle (\X*2,\X*2) node[fitting node] (11) {};
\draw (\X*1,\X*0) rectangle (\X*2,\X*1) node[fitting node] (10) {};
\draw (\X*1,-1*\X) rectangle (\X*2,\X*0) node[fitting node] (1-1) {};
\draw (\X*1,-2*\X) rectangle (\X*2,-1*\X) node[fitting node] (1-2) {};
\draw (\X*1,-3*\X) rectangle (\X*2,-2*\X) node[fitting node] (1-3) {};

\draw (\X*1,\X*2) rectangle (\X*2,\X*3) node[fitting node] (12) {};
\draw (\X*1,\X*3) rectangle (\X*2,\X*4) node[fitting node] (13) {};
\draw (\X*1,\X*4) rectangle (\X*2,\X*5) node[fitting node] (14) {};
\draw ($(0-3)+(-1,0.25)$) -- ($(1-3)+(0,0.25)$)
($(02)+(-1,0.25)$) -- ($(12)+(0,0.25)$)
;

\path[arrows={- latex'}]
    (02.center) edge[bend right] node [right] {} (11.west)
    (0-1.center) edge[bend right] node [right] {} (1-2.west)
    (0-3.east) edge[red,bend right=90] node [right] {} (13.east);

\node at (11.center) {$\scriptscriptstyle L^+$};
\node at (1-2.center) {$\scriptscriptstyle L^+$};
\node at (13.center) {$\color{red} \scriptscriptstyle L^+$};

\draw [decorate,decoration={brace,amplitude=10pt},xshift=-4pt,yshift=0pt]
($(0-3)+(-1,0.25)$) -- ($(02)+(-1,0.25)$) node [black,midway,xshift=-0.6cm] 
{};

\draw [decorate,decoration={brace,amplitude=10pt},xshift=-4pt,yshift=0pt]
($(1-3)+(0.25,0.25)$) -- ($(1-3)+(0.25,-1)$) node [black,midway,xshift=-0.6cm] 
{};

\node at (-4,0.25) {\begin{minipage}{0.3\textwidth} There is a unique non-left-compatible $L^+$ adjacent to every non-right-compatible $L^-$. All other $L^+$-strips in $(C_1,T_0)$ are left-compatible.  \end{minipage}};

\node at (4.5,-1.5) {\begin{minipage}{0.3\textwidth} All $L^+$ in $(C_1,T_0)$ in this section are either left-compatible or do not contribute to $\gamma$. \end{minipage}};

\end{tikzpicture}
\caption{Unique pairings of non-exchangeable $L^-$ and $L^+$. }
\label{fig:chain1}
\end{figure}

\begin{cor}\label{chain1cor}
Let $C=(C_{-1},C_0)$ and $T=(T_{-1},T_0)$ be a pair of KR-tableaux. To every non-left-compatible $L^+$-strip in $(C_0,T_0)$ that contributes to $\gamma$, there exists a unique non-right-compatible $L^-$-strip in $(C_{-1},T_{0})$ that contributes to $\gamma$. All other $L^-$-strips in $(C_{-1},T_0)$ are right-compatible. 
\end{cor}
\begin{proof}
Left-compatibility and right-compatibility are, in fact, identical requirements with different points of references. This is evident in the underlying inequalities being the same (see Definition \ref{defn:compat}). We simply start with non-left-compatible $L^+$-strips in $(C_0,T_0)$ and assign the $L^-$-blocks in $(C_{-1},T_{0})$ that pair with the $L^+$-blocks by Lemma \ref{chain1}. 
\end{proof}

Notice that in the proof of Lemma \ref{chain1}, when considering the last $L^-$-strip with $(C_0,T_0)$ regular (case two), we did not use the fact that $L^-$ is not right-compatible. Moreover, this situation includes the case when $L^-$ is not column-compatible.  
\begin{cor}\label{chain2}
Suppose $(C_0,T_0)$ is regular and there is a non-column-compatible $L^-$-strip. If there is a column in $C$ to the right of $C_0$, call it $C_1$, then there is a unique non-left-compatible $L^+$ in $(C_1,T_0)$. \qed
\end{cor}

\begin{lemma}\label{chain3}
Suppose $(C_0,T_0)$ is regular and there is a non-column-compatible $L^-$. If there is a column in $T$ to the left of $T_0$, call it $T_{-1}$, then $L^-$ is never \emph{left}-compatible.
\end{lemma}

\begin{proof}
We want to show that Condition $(lT)$ is not satisfied. We have the following picture:

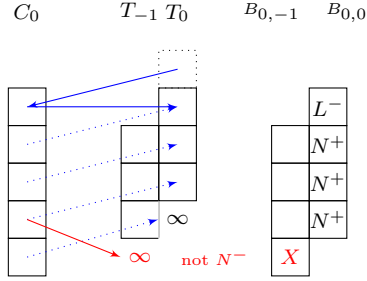
\begin{center}
\begin{tikzpicture}

\node at (1.5*\X, \X*6) {$\scriptstyle{C_0}$};
\draw (\X*1,-1*\X) rectangle (\X*2,\X*0) node[fitting node] (1-1) {};
\draw (\X*1,\X*0) rectangle (\X*2,\X*1) node[fitting node] (10) {};
\draw (\X*1,\X*1) rectangle (\X*2,\X*2) node[fitting node] (11) {};
\draw (\X*1,\X*2) rectangle (\X*2,\X*3) node[fitting node] (12) {};
\draw (\X*1,\X*3) rectangle (\X*2,\X*4) node[fitting node] (13) {};

\node at (0.5*\X+\D, \X*6) {$\scriptstyle{T_{-1}}$};
\draw[white] (0+\D,-1*\X) rectangle (\X*1+\D,\X*0) node[fitting node] (t0-1) {};
\draw (0+\D,0) rectangle (\X*1+\D,\X*1) node[fitting node] (t00) {};
\draw (0+\D, \X*1) rectangle (\X*1+\D, \X*2) node[fitting node] (t01) {};
\draw (0+\D, \X*2) rectangle (\X*1+\D, \X*3) node[fitting node] (t02) {};
\node at (t0-1.center) {$\scriptstyle \color{red} \infty$};

\node at (1.5*\X+\D, \X*6) {$\scriptstyle{T_0}$};
\draw[white] (\X*1+\D,\X*0) rectangle (\X*2+\D,\X*1) node[fitting node] (t10) {};
\draw (\X*1+\D,\X*1) rectangle (\X*2+\D,\X*2) node[fitting node] (t11) {};
\draw (\X*1+\D,\X*2) rectangle (\X*2+\D,\X*3) node[fitting node] (t12) {};
\draw (\X*1+\D,\X*3) rectangle (\X*2+\D,\X*4) node[fitting node] (t13) {};
\draw[dotted] (\X*1+\D,\X*4) rectangle (\X*2+\D,\X*5) node[fitting node] (t14) {};
\node at (t10.center) {$\scriptstyle \infty$};

\draw[arrows={- latex'},blue] (t14.center) -- (13.center);
\draw[arrows={- latex'},blue] (13.center) -- (t13.center);
\draw[arrows={- latex'},dotted,blue] (12.center) -- (t13.center);
\draw[arrows={- latex'},dotted,blue] (11.center) -- (t12.center);
\draw[arrows={- latex'},dotted,blue] (10.center) -- (t11.center);
\draw[arrows={- latex'},dotted,blue] (1-1.center) -- (t10.west);
\draw[arrows={- latex'},red] (10.center) -- (t0-1.west);
\node at ($(t0-1.center)+(1,0)$) {$\scriptscriptstyle\color{red} \text{not } N^-$};

\node at (0 + 2*\D, \X*6) {$\scriptscriptstyle{B_{0,-1}}$};
\draw (0+2*\D,-1*\X) rectangle (\X*1+2*\D,\X*0) node[fitting node] (g0-1) {};
\draw (0+2*\D,0) rectangle (\X*1+2*\D,\X*1) node[fitting node] (g00) {};
\draw (0+2*\D, \X*1) rectangle (\X*1+2*\D, \X*2) node[fitting node] (g01) {};
\draw (0+2*\D, \X*2) rectangle (\X*1+2*\D, \X*3) node[fitting node] (g02) {};

\node at (g0-1.center) {$\color{red} \scriptstyle{X}$};

\node at (2*\X+2*\D, \X*6) {$\scriptscriptstyle{B_{0,0}}$};
\draw (\X*1+2*\D,\X*0) rectangle (\X*2+2*\D,\X*1) node[fitting node] (g10) {};
\draw (\X*1+2*\D,\X*1) rectangle (\X*2+2*\D,\X*2) node[fitting node] (g11) {};
\draw (\X*1+2*\D,\X*2) rectangle (\X*2+2*\D,\X*3) node[fitting node] (g12) {};
\draw (\X*1+2*\D,\X*3) rectangle (\X*2+2*\D,\X*4) node[fitting node] (g13) {};

\node at (g10.center) {$\scriptstyle{N^+}$};
\node at (g11.center) {$\scriptstyle{N^+}$};
\node at (g12.center) {$\scriptstyle{N^+}$};
\node at (g13.center) {$\scriptstyle{L^-}$};


\end{tikzpicture}
\end{center}
\end{proof}

\begin{lemma} \label{fundClustShape}
Let $C=(C_i)$ and $T=(T_j)$ be KR-tableaux from $\mathcal{C}$. Suppose $(C_i,T_j)$ is regular. Then at least one of $C_{i+1}$ or $T_{j-1}$ must exist. In other words, either there is a column to the right of $C_i$ or to the left of $T_j$ or both. 
\end{lemma}

\begin{proof}
Since the central columns of $C$ and $T$ form a fundamental pair (see Section \ref{fundClusterDiagrams}), the central column of $T$ must be to the left of $T_j$ or the central column of $C$ must be to the right of $C_i$. 
\end{proof}

When looking for $L^+$ to cancel out the contributions of non-exchangeable $L^-$-blocks, we need only address left-compatible $L^-$-blocks. The reason is that non-left-compatible $L^-$-blocks would have already been paired with a non-right-compatible $L^+$-blocks previously. Lemma \ref{chain1} addresses most of the situations when we have a non-exchangeable $L^-$ that contributes to $\gamma$. Lemma \ref{fundClustShape} shows that there are exactly two other situations when $L^-$ contributes to $\gamma$ and needs a pair. Lemmas \ref{chain2} and \ref{chain3} address each of those situations respectively. This concludes the unique pairing when we have a non-exchangeable $L^-$ that contributes to $\gamma$. 

Next, we address the case when we have a non-exchangeable $L^-$ that does not contribute to $\gamma$. We must ensure that there is no corresponding $L^+$, which would create an imbalance. There are two possibilities as listed in Table \ref{tab:cancellations}.

\begin{lemma}\label{chain4}
Suppose $(C_0,T_0)$ is fundamental. Then every $L^-$-strip in $(C_0,T_0)$ is column-compatible, i.e. contributes to $\gamma$.
\end{lemma}
\begin{proof}
By Remark \ref{rmk:obs2}(b), the last block in $B_{0,0}$ is not $N^+$, and, therefore, the boundary term of $\gamma(C_0,T_0)$ is zero. 
\end{proof}

\begin{remark} \label{rmk:obs4}
Let $C=(C_0,C_1)$ and $T=(T_0)$ be KR-tableaux. If $(C_0,T_0)$ has an $L^+UN^-$-block at $p$, then $(C_1,T_0)$ has an $N^-$-block at $p$ (see Figure~\ref{fig:obs4}).  
\begin{figure}[h]
\begin{center}
\begin{tikzpicture}

\node at (0,0) {
\begin{tikzpicture}
\node at (0.5*\X, \X*5) {$\scriptstyle{C_0}$};

\draw (0, \X*2) rectangle (\X*1, \X*3) node[fitting node] (02) {};
\draw (0, \X*3) rectangle (\X*1, \X* 4) node[fitting node] (03) {};
\node at ($(02) +(-.6,0)$) {$\scriptstyle p$};
\node at ($(03) +(-.6,0)$) {$\scriptstyle p-1$};

\node at (\X*1.5, \X*5) {$\scriptstyle{C_1}$};
\draw (\X*1,\X*2) rectangle (\X*2,\X*3) node[fitting node] (12) {};
\draw (\X*1,\X*3) rectangle (\X*2,\X*4) node[fitting node] (13) {};

\node at (0.5*\X+\D, \X*5) {$\scriptstyle{T_0}$};
\draw (0+\D, \X*2) rectangle (\X*1+\D, \X*3) node[fitting node] (t02) {};
\draw (0+\D, \X*3) rectangle (\X*1+\D, \X* 4) node[fitting node] (t03) {};
\node at ($(t02) +(1.5,0)$) {$\scriptscriptstyle N^-\text{ in }(C_0,T_0)$};

\draw[arrows={ - latex'},dotted,blue] (02.center) -- (13.center);
\draw[arrows={ - latex'},blue] (03.center) -- (02.center);
\draw[arrows={ - latex'},blue,dotted] (t02.center) -- (03.center);
\draw[arrows={ - latex'},red,thick,dotted] (t02.center) -- (13.center);

\end{tikzpicture}
};

\node at (6,0) {
\begin{tikzpicture}
\node at (0.5*\X, \X*5) {$\scriptstyle{C_0}$};

\draw (0, \X*2) rectangle (\X*1, \X*3) node[fitting node] (02) {};
\draw (0, \X*3) rectangle (\X*1, \X* 4) node[fitting node] (03) {};
\node at ($(02) +(-.6,0)$) {$\scriptstyle p$};
\node at ($(03) +(-.6,0)$) {$\scriptstyle p-1$};

\node at (\X*1.5, \X*5) {$\scriptstyle{C_1}$};
\draw (\X*1,\X*2) rectangle (\X*2,\X*3) node[fitting node] (12) {};
\draw (\X*1,\X*3) rectangle (\X*2,\X*4) node[fitting node] (13) {};

\node at (0.5*\X+\D, \X*5) {$\scriptstyle{T_0}$};
\draw (0+\D, \X*2) rectangle (\X*1+\D, \X*3) node[fitting node] (t02) {};
\draw (0+\D, \X*3) rectangle (\X*1+\D, \X* 4) node[fitting node] (t03) {};
\node at ($(t02) +(1.5,0)$) {$\scriptscriptstyle L^+U\text{ in }(C_0,T_0)$};

\draw[arrows={ - latex'},dotted,blue] (02.center) -- (13.center);
\draw[arrows={ - latex'},blue,dotted] (t02.center) -- (02.center);
\draw[arrows={ - latex'},red,thick,dotted] (t02.center) -- (13.center);

\end{tikzpicture}
};
\end{tikzpicture}
\end{center}
\caption{Composition of arrows.}
\label{fig:obs4}
\end{figure}
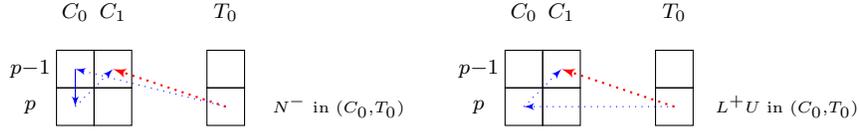
Indeed, by composing arrows, we find $T_0[p]\leq C_1[p-1]$, which means there is an $N^-$-block at $p$ in $(C_1,T_0)$.  
\end{remark}

\begin{lemma}\label{chain5}
Suppose $(C_0,T_0)$ is anti-fundamental and there is a non-right-compatible and non-column-compatible $L^-$-strip. Then there does not exist a corresponding non-exchangeable $L^+$-strip in $(C_1,T_0)$.  
\end{lemma}

\begin{proof}
Let $\tilde{L}^-$ be the $L^-$-strip in question and let $p'$ be its index. Since $L^-$ is not column-compatible, it must be the last $L$-strip in $(C_0,T_0)$ and the last block in $(C_0,T_0)$ must be $N^+$. In other words, the boundary term of $\gamma(C_0,T_0)$ is $+1$, which cancels out the contribution of $\tilde{L}^-$. Since the contribution of $\tilde{L}^-$ is already canceled out, we want to show that there is no non-exchangeable $L^+$ in $(C_1,T_0)$ that is paired with $\tilde{L}^-$. It suffices to show that all possible $L^+$'s that can be paired with $\tilde{L}^-$ are, in fact, left-compatible. Notice that there must be $C_1$ in order for $\tilde{L}^-$ to fail the condition $(rC)$.  
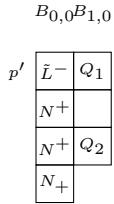
\begin{center}
\begin{tikzpicture}

\node at (0.5*\X, \X*5) {$\scriptscriptstyle{B_{0,0}}$};
\draw (0,0) rectangle (\X*1,\X*1) node[fitting node] (00) {};
\node at (00.center) {$\scriptscriptstyle{N_+}$};

\draw (0, \X*1) rectangle (\X*1, \X*2) node[fitting node] (01) {};
\node at (01.center) {$\scriptscriptstyle{N^+}$};
\draw (0, \X*2) rectangle (\X*1, \X*3) node[fitting node] (02) {};
\node at (02.center) {$\scriptscriptstyle{N^+}$};
\draw (0, \X*3) rectangle (\X*1, \X* 4) node[fitting node] (03) {};
\node at (03.center) {$\scriptscriptstyle{\tilde{L}^-}$};
\node at ($(03.center)+(-.5,0)$) {$\scriptscriptstyle p'$};

\node at (1.5*\X, \X*5) {$\scriptscriptstyle{B_{1,0}}$};
\draw (\X*1,\X*1) rectangle (\X*2,\X*2) node[fitting node] (11) {};
\node at (11.center) {$\scriptscriptstyle{Q_2}$};
\draw (\X*1,\X*2) rectangle (\X*2,\X*3) node[fitting node] (12) {};
\draw (\X*1,\X*3) rectangle (\X*2,\X*4) node[fitting node] (13) {};
\node at (13.center) {$\scriptscriptstyle{Q_1}$};

\node at (\X*1+4*\D, \X*3) {\begin{minipage}{0.6\textwidth}
Let $Q_1$ and $Q_2$ be the index $p'$ and the last block in $(C_1,T_0)$. If there are any $L^+$-blocks between $Q_1$ and $Q_2$, then they are all left-compatible. Indeed, $(lC)$ is satisfied (seen from the picture) and $(lT)$ is always satisfied for $L^+$-strips. \end{minipage}};
\end{tikzpicture}
\end{center}

Let $\hat{L}^-$ be the very first $L^-$-strip in $(C_0,T_0)$. 

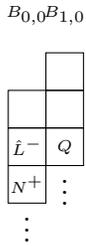
\begin{center}
\begin{tikzpicture}

\node at (0.5*\X, \X*4) {$\scriptscriptstyle{B_{0,0}}$};
\draw[white] (0,-2*\X) rectangle (\X*1,-1*\X) node[fitting node] (0-2) {};
\node at (0-2.center) {$\scriptscriptstyle{\vdots}$};
\draw[white] (\X*1,-1*\X) rectangle (\X*2,\X*0) node[fitting node] (1-1) {};
\node at (1-1.center) {$\scriptscriptstyle{\vdots}$};

\draw (0,0) rectangle (\X*1,\X*1) node[fitting node] (00) {};
\node at (00.center) {$\scriptscriptstyle{\hat{L}^-}$};

\draw (0,-1*\X) rectangle (\X*1,0*\X) node[fitting node] (0-1) {};
\node at (0-1.center) {$\scriptscriptstyle{N^+}$};

\draw (0, \X*1) rectangle (\X*1, \X*2) node[fitting node] (01) {};

\node at (1.5*\X, \X*4) {$\scriptscriptstyle{B_{1,0}}$};

\draw (\X*1,\X*1) rectangle (\X*2,\X*2) node[fitting node] (11) {};
\node at (10.center) {$\scriptscriptstyle{Q}$};
\draw (\X*1,\X*0) rectangle (\X*2,\X*1) node[fitting node] (10) {};

\draw (\X*1,\X*2) rectangle (\X*2,\X*3) node[fitting node] (12) {};

\node at (\X*1+4*\D, \X*1) {\begin{minipage}{0.6\textwidth}
Since $(C_0,T_0)$ is anti-fundamental, by Remark \ref{rmk:obs2}, the first block in $B_{0,0}$ is not $N^+$. Since $\hat{L}^-$ is the first $L^-$-strip in $(C_0,T_0)$, the first block in $B_{0,0}$ is not an $L^-$-block either. So, it can be either $N^-, U$ or $L^+$. By Remark \ref{rmk:obs1}, since $N^+$ cannot follow an $L^+N^-U$-block, there are no $N^+$-blocks above $\hat{L}^-$.  By Remark \ref{rmk:obs4}, the adjacent blocks in $(C_1,T_0)$, i.e. the blocks above $Q$, are all $N^-$. In other words, there are no $L^+$-blocks above $Q$. This concludes the proof. 
\end{minipage}};
\end{tikzpicture}
\end{center}
\end{proof}

We are now ready to put everything together. 

\begin{proof}[Proof of Theorem \ref{thm:main1}]
If there are no $L$-blocks anywhere in $(C,T)$, the statement is trivially true. Let's assume there is at least one $L$-block in $(C,T)$, and without loss of generality, we may assume it is $L^-$. Otherwise we consider $(T,C)$ instead. Moreover, we can assume the $L^-$ is left-compatible. Indeed, if there is a non-left-compatible $L^-$ in $(C_i,T_j)$, there is a non-left-compatible $L^+$ in $(T_j,C_i)$. By Corollary \ref{chain1cor}, there is a non-right-compatible $L^-$ in $(T_{j-1},C_{i})$. If this $L^-$ is again non-left-compatible, we go through the same chain of arguments. We continue inductively and eventually, since the process must end when $C$ or $T$ runs out of columns, we are guaranteed to find a left-compatible $L^-$.  

Consider a left-compatible $L^-$-strip in $(C_i,T_0)$. If there are no such $L^-$-strips in $(C,T_0)$, which happens if there are no $L^-$-blocks in $(C,T_0)$, we remove $T_0$ from $T$ and consider $(C,T_1)$. The list of possibilities for $L^-$ are listed in Table \ref{tab:cancellations}.
\begin{enumerate}
\item Suppose the $L^-$-block in question contributes to $\gamma$. If it is non-right-compatible, there must be $C_{i+1}$. This is because all $L^-$-strips satisfy the condition $(rT)$. In order for the $L^-$ to be non-right-compatible, there must be $C_{i+1}$ that pose violations. By Lemma \ref{chain1}, there exists a unique non-left-compatible $L^+$ in $(C_{i+1},T_j)$. If $(C_i,T_0)$ is regular and there exists $C_{i+1}$, then Corollary \ref{chain1cor} is used. If there is no $C_{i+1}$, by Lemma \ref{fundClustShape} there must be $T_{-1}$ and by Lemma \ref{chain3}, the $L^-$ is not left-compatible, which is a contradiction. 
\item Suppose the $L^-$-block does not contribute to $\gamma$. Then $(C_i,T_0)$ is (anti-)fundamental. If $(C_i,T_0)$ is fundamental, then every $L^-$ contributes to $\gamma$ and we apply the previous analysis. If $(C_i,T_0)$ is anti-fundamental and $L^-$ is right-compatible, there is no need to pair it with anything since $L^-$ does not contribute to $\gamma$ and pose no violations with $C_{i+1}$. If $(C_i,T_0)$ is anti-fundamental and $L^-$ is not right-compatible, Lemma \ref{chain4} shows there is no corresponding non-left-compatible $L^+$'s. 
\end{enumerate}

We now remove $T_0$ from $T$ and consider $\gamma(C,T_1)$. Since $T_0$ is removed, all non-left-compatible $L^-$-strips become left-compatible. However, restrictions on $L^+$-strips are not changed since left-compatibility comes from $C$ itself. By the exact same argument, all negative non-exchangeable contributions in $\gamma(C,T_1)$ are canceled by positive non-exchangeable contributions. We continue this argument for all $T_i$. This proves all the negative non-exchangeable terms in $\gamma(C,T)$ are uniquely canceled out by non-right-compatible terms in $\gamma(C,T)$. 

Next, we consider $(T,C)$, where all left-compatible $L^+$-strips in $(C,T)$ become left-compatible $L^-$-strips in $(T,C)$ and apply the same argument. 
\end{proof}

\section{Proof of Conditions II and III}
The previous section showed that the condition I holds. That is, the $(q,t)$-characters in the fundamental cluster $t$-commute with one another. Then, by Corollary \ref{gammaComm}, the commutation matrix $\Lambda$ of the fundamental cluster $\mathcal{C}$ is given by:
\begin{eqnarray*}
\Lambda_{i,k}^{i',k'} = 2\epsilon( \mathbf{Y}_{k,-k+(i+k+1)_2}^{(i)}, \mathbf{Y}_{k',-k'+(i'+k'+1)_2}^{(i')} )\,\, ,
\end{eqnarray*}
where $\mathbf{Y}_{k,j(i,k)}^{(i)}$ and $\mathbf{Y}_{k',j(i',k')}^{(i')}$ are dominant monomials of KR-modules in the fundamental cluster $\mathcal{C}$.  

We now show condition II holds. 

\begin{thm} \label{main2}
Let $B$ be the infinite matrix associated to the quiver $\Gamma_B$ in Figure~\ref{fig:Tsystem} (on page~\pageref{fig:Tsystem} ). Then, 
\begin{equation*}
\Lambda B = D\,,
\end{equation*}
where $D$ is a diagonal matrix with positive entries. In other words, $(\Lambda, B)$ is a compatible pair (as in \cite{berenstein2005quantum}, Section $3$). 
\end{thm}

We will provide some definitions and lemmas first. 

\begin{defn} \label{notationCond2}
We fix the following notations for convenience:
\begin{enumerate}
\item Given $k\in \mathbb{Z}$, the $s$-number is defined as:
\begin{eqnarray*}
\left[0\right]_s&:=&0\,,\\
\left[k\right]_s &:=& \frac{ s^k - s^{-k}}{s-s^{-1}} = s^{k-1} + s^{k-3} + \cdots + s^{-k+3} + s^{-k+1}\,.
\end{eqnarray*}
\item Denote $\mathcal{F} := \mathbb{Z}[[s]][s^{-1}]$. Then, multiplication operator is well-defined in $\mathcal{F}$.
\item Given $f(s)\in \mathcal{F}$, we define $\left[f\right]_0$ to be the constant term in $f(s)$, e.g. $[s^{-1}+3+s]_0=3$. 
\end{enumerate}
\end{defn}

\begin{defn} \label{uGenFun}
Let $m$ be a monomial in $\mathbb{Z}[Y_{i,j}^{\pm 1}]$. We define the following generating series:
\begin{eqnarray*}
u_{i}(m)(s) := \sum_{j\in \mathbb{Z}} u_{i,j}(m) s^j &;& u(m)(s) := \sum_{i=1}^r e_i \otimes u_{i}(m)(s)\,\, ,
\end{eqnarray*}
where $u_{i,j}(m)$ is the exponent of $Y_{i,j}$ in $m$ as defined in Definition \ref{uv} and $e_i \in \mathbb{Z}^r$ is the vector with $1$ in the $i$th position and $0$'s everywhere else. 
\end{defn}

\begin{lemma} \label{sNum}
Let $\mathbf{Y}_{k,-k+(i+k+1)_2}^{(i)}$ be the dominant monomial of $\chi_{k,-k+(i+k+1)_2}^{(i)}\in \mathcal{C}$. Then,
\begin{eqnarray*}
u(\mathbf{Y}_{k,-k+(i+k+1)_2}^{(i)})(s)&=& e_i \otimes s^{-1+(i+k+1)_2}\,\,[k]_s\,.
\end{eqnarray*}
\end{lemma}
\begin{proof}
Denote $j:= (i+k+1)_2$. Modules in $\mathcal{C}$ have dominant monomials of the form:
\begin{eqnarray} \label{Yexp}
\mathbf{Y}_{k,j}^{(i)} = Y_{i, -k+j} Y_{i, -k+j+2} \cdots Y_{i, -k+ j + 2(k-1)}\,.
\end{eqnarray}
Then, using Definition \ref{uGenFun}, we directly compute as follows: 
\begin{eqnarray*}
u_i (\mathbf{Y}_{k,j}^{(i)})(s) &=& \sum_{p\in \mathbb{Z}} u_{i,p}(\mathbf{Y}_{k,j}^{(i)})s^p \\
&=&s^{-k+j} + s^{-k+j+2} + \cdots +s^{j+k-2} \\
&=& s^{j-1} \left(s^{-k+1} + s^{-k+3} + \cdots  + s^{k-1}\right)\\
&=&s^{j-1} [k]_s\,,
\end{eqnarray*}
where $u_{i,p}\left( \mathbf{Y}_{k,j}^{(i)}\right)$ is the power of $Y_{i,p}$ in $\mathbf{Y}_{k,j}^{(i)}$, which is either $1$ or $0$ as seen from \eqref{Yexp}. Notice that $u_{i',p}(\mathbf{Y}_{k,j}^{(i)})=0$ if $i' \neq i$. Therefore, $u_{i'}(\mathbf{Y}_{k,j}^{(i)})(s)=0$ and we have:
\begin{eqnarray*}
u(\mathbf{Y}_{k,j}^{(i)}) = e_i \otimes s^{j-1}[k]_s\,.
\end{eqnarray*}
\end{proof}

\begin{defn} \label{utildeGenFun}
Let $m$ be a monomial in $\mathbb{Z}[Y_{i,j}^{\pm 1}]$. We define the following generating series:
\begin{eqnarray*}
\tilde{u}_i(m)(s) = \sum_{j\in \mathbb{Z}} \tilde{u}_{i,j}(m) s^j &\text{and}& \tilde{u}(m)(s) = \sum_{i=1}^r e_i \otimes \tilde{u}_i(m)(s)\,,
\end{eqnarray*}
where $\tilde{u}_{i,j}(m)$ are the solutions of the system given in Definition \ref{uinverse}. 
\end{defn}

\begin{defn}\label{action}
Let $M\in {\rm{Mat}}_{r\times r}(\mathbb{Z})$ and $g \in \mathcal{F}$. We define an action of $M\otimes g$ on the space $\mathbb{Z}^r \times \mathcal{F}$ as follows:
\begin{eqnarray*}
\left(M\otimes g\right)\left(v\otimes f \right) = \left(Mv\right)\otimes \left(gf\right) &\text{ for any }& v\in \mathbb{Z}^r, f \in  \mathcal{F},
\end{eqnarray*}
where $Mv$ is the matrix multiplication and $gf$ is the usual multiplication. 
\end{defn}

Recall the system of equations in Definition \ref{uinverse}:
\begin{equation*} 
u_{i,j}(m) = \tilde{u}_{i,j-1}(m) + \tilde{u}_{i,j+1}(m) - \tilde{u}_{i-1,j}(m) - \tilde{u}_{i+1,j}(m)\,,
\end{equation*}
defined for any monomial $m$ and any $i\in I, j\in \mathbb{Z}$. Let $A = C-2I$, where $C$ is the Cartan matrix of $\mathfrak{sl}_{r+1}$. We rewrite this system for $m=\mathbf{Y}_{k,-k+(i+k+1)_2}^{(i)}$ as follows:
\begin{equation}
u^{i,k}(s) = (1\otimes s + 1\otimes s^{-1} + A\otimes 1) \tilde{u}^{i,k}(s)\,,
\end{equation}
where the action is as in Definition \ref{action}.
\begin{defn}
Denote the operator
\begin{equation} \label{K}
K = (1\otimes s^{-1}) ( 1\otimes 1 + A\otimes s + 1\otimes s^{2})\,.
\end{equation}
Define operator $D$ as a formal power series in $s$, expanded at $0$, given by:
\begin{eqnarray} \label{D}
D=(1\otimes 1 +  A\otimes s + 1\otimes s^2 )^{-1}(1\otimes s)\,.
\end{eqnarray}
Then we have $DK=KD=1\otimes 1$ and $Du^{i,k}(s) = \tilde{u}^{i,k}(s)$.
\end{defn}
\begin{remark}
In Definition \ref{uinverse}, we require $\tilde{u}_{i,j}(m)=0$ for $j$ sufficiently small. This condition is equivalent to expanding the formal inverse of the power series \eqref{K} at $0$, which is the choice we made in \eqref{D}.  
\end{remark}

\begin{remark}
It is easy to see that $K$ commutes with $1\otimes s^n$ for any $n$. Since $D$ is the inverse of $K$, we have:
\begin{eqnarray*}
D(1\otimes s^n) = D(1\otimes s^n)KD = DK(1\otimes s^n) D = (1\otimes s^n) D\,.
\end{eqnarray*}
That is, the operator $D$ commutes with $1\otimes s^n$ for $n\in \mathbb{Z}$. 
\end{remark}

\begin{defn} \label{myInner}
Let $v,w\in \mathbb{Z}^r$ and $f,g\in \mathcal{F}$. We define the following inner product on $\mathbb{R}^r\otimes \mathcal{F}$ as follows:
\begin{equation}
\left(v\otimes f \right) \cdot \left(w\otimes g \right) = \left< v,w\right>  \left[ f\left(s^{-1}\right)g\left(s\right) \right]_{0}\,,
\end{equation}
where $\left<v,w\right>$ is the usual inner product on $\mathbb{R}^r$ and $\left[f(s)\right]_{0}$ is the constant term in $f(s)$ as defined in Definition \ref{notationCond2}. 
\end{defn}
\begin{remark}
Notice that the inner product in Definition \ref{myInner} is symmetric. That is, 
\begin{equation*}
(v\otimes f) \cdot (w\otimes g) = (w\otimes g) \cdot (v\otimes f)\,. 
\end{equation*}
\end{remark}

\begin{defn}
Given $M\otimes h \in {\rm{Mat}}_{r\times r}(\mathbb{Z}) \times \mathcal{F}$, we define the \emph{transpose} of $M\otimes h$, denoted $(M\otimes h)^t$, by the following condition:
\begin{equation*}
\left(v\otimes f \right) \cdot M\otimes h \left(w\otimes g\right) = (M\otimes h)^t\left(v\otimes f\right) \cdot \left(w\otimes g\right)\,, 
\end{equation*}  
for any $v,w\in \mathbb{Z}^r$ and $f,g\in \mathcal{F}$. An operator $M\otimes h$ is \emph{symmetric} if $M\otimes h=(M\otimes h)^t$. 
\end{defn}
\begin{lemma} \label{transpose}
\begin{eqnarray*}
(1\otimes s)^t = 1\otimes s^{-1} ;&K^t = K; & D^t = D\,.
\end{eqnarray*}
\end{lemma}
\begin{proof}
It is easy to see that $K$ is symmetric. Since $D^{-1}=K$, $D$ is also symmetric. The remaining result is shown by direct computation:
\begin{eqnarray*}
\left(v\otimes f(s) \right) \cdot \left(1\otimes s\right) \left(w\otimes g(s) \right) &=& \left< v,w\right> \left[f(s^{-1}) (sg(s)) \right]_0 \\
&=&\left< v,w\right> \left[(s^{-1})^{-1}f(s^{-1}) g(s) \right]_0  \\&=& (1\otimes s^{-1}) (v\otimes f(s)) \cdot (w\otimes g(s))\,,
\end{eqnarray*}
for any $v,w\in \mathbb{Z}^r$ and $f,g\in \mathcal{F}$. 
\end{proof}

\begin{lemma} \label{EpsilonForm}
Let $p,p'$ be dominant monomials in $\mathcal{M}$. Then,
\begin{equation*}
\epsilon\left(p,p'\right) = (1\otimes s- 1\otimes s^{-1})Du\left(p\right)(s) \cdot u\left(p'\right)(s)\,,
\end{equation*}
where $\epsilon$ is given in Definition \ref{def:Nakajima}.
\end{lemma}


\begin{proof}
By definition, we have
\begin{eqnarray*}
u(p)(s)=\sum_i e_i \otimes \sum_{j} u_{i,j}(p)s^j &\text{ and }&\tilde{u} (p')(s) = \sum_{i'} e_{i'}\otimes \sum_{j'=1}^r \tilde{u}_{i',j'}(p') s^{j'}\,.
\end{eqnarray*}
Then,
\begin{eqnarray*}
u(p)(s) \cdot (1\otimes s) \tilde{u}(p')(s) &=& \sum_{i,i'=1}^r \left<e_i, e_{i'} \right>  \left[ \sum_{j,j'} u_{i,j}(p) s^{-j} \tilde{u}_{i',j'}(p')s^{j'+1}\right]_0 \\ &=& \sum_j u_{i,j}(p) \tilde{u}_{i, j-1}(p')\,.
\end{eqnarray*}
Therefore, using Deifnition \ref{def:Nakajima}, we have:
\begin{eqnarray*}
\epsilon(p,p')&=& - \sum_{i,j} u_{i,j}(p) \tilde{u}_{i,j-1}(p') + \sum_{j} u_{i,j}(p') \tilde{u}_{i,j-1}(p) \\
&=&- u(p)(s) \cdot (1\otimes s) \tilde{u}(p')(s)+ u(p')(s) \cdot (1\otimes s) \tilde{u}(p)(s)\\
&=&- u(p)(s) \cdot (1\otimes s) Du(p')(s)+ u(p')(s) \cdot (1\otimes s)D u^(p)(s)\\
&=&- D^t (1\otimes s)^t u(p)(s) \cdot u(p')(s) + (1\otimes s) D u(p)(s) \cdot u(p')(s) \\
&=&(1\otimes s-1\otimes s^{-1})D u(p)(s) \cdot u(p')(s)\,.
\end{eqnarray*}
\end{proof}
\begin{remark}
As a sanity check, let us verify that the expression we found for $\epsilon$ is also anti-symmetric. 
\begin{eqnarray*}
\epsilon(p',p) &=& (1\otimes s - 1\otimes s^{-1} ) D u(p')(s) \cdot u(p)(s)\\
 &=& u(p)(s) \cdot (1\otimes s - 1\otimes s^{-1} ) D u(p')(s)\\
&=& D^t (1\otimes s^{-1} - 1\otimes s) u(p)(s) \cdot u(p')(s) = - \epsilon(p,p')\,.
\end{eqnarray*}
\end{remark}

\begin{defn}\label{convention}
When $u$ and $\tilde{u}$ generating functions of Definitions \ref{uGenFun} and \ref{utildeGenFun} are applied to the dominant monomial of a module in the fundamental cluster $\mathcal{C}$, i.e. monomial of the form $\mathbf{Y}_{k,-k+(i+k+1)_2}^{(i)}$, we make the following simplifying notation:
\begin{eqnarray*}
u(\mathbf{Y}_{k,-k+(i+k+1)_2}^{(i)})(s) := u^{i,k}(s) &\text{ and }& \tilde{u}(\mathbf{Y}_{k,-k+(i+k+1)_2}^{(i)})(s) :=  \tilde{u}^{i,k}(s)\,.
\end{eqnarray*}
\end{defn}
\begin{remark}
We emphasize that the superscript $u^{i,k}(s)$ in Definition \ref{convention} indicates the monomial $Y_{k,-k+(i+k+1)_2}^{(i)}$. In contrast, the subscript $u_{i,j}(m)$ indicates the exponent of $Y_{i,j}$ in $m$. 
\end{remark}

\begin{lemma}
The following equation holds:
\begin{eqnarray*}
 u^{i,k-1}(s) + u^{i,k+1}(s) - u^{i-1,k}(s) - u^{i+1,k}(s)   =  1\otimes s^{-1+(i+k)_2} Ke_i\otimes [k]_s\,,
\end{eqnarray*}
for any $i\in I$ and $k\geq 1$. 
\end{lemma}
\begin{proof}
We have:
\begin{eqnarray*}
[k-1]_s+[k+1]_s &=& \left(s^{k-2} + \cdots + s^{-k+2} \right) + \left(s^{k} + s^{k-2} +\cdots +s^{-k+2} + s^{-k}\right) \\
&=&  \left(s^{k-2}  + \cdots + s^{-k+2} + s^{-k}\right) + \left(s^{k} + s^{k-2} +\cdots + s^{-k+2}\right) \\
&=&s^{-1} \left(s^{k-1} + \cdots + s^{-k+1} \right) + s\left(s^{k-1} +\cdots + s^{-k+1}\right) \\
&=& (s^{-1}+s) [k]_s\,.
\end{eqnarray*}
Notice that the above equation holds when $k=1$ as well:
\begin{eqnarray*}
[0]_s + [2]_s = [2]_s = s + s^{-1} = \left( s+s^{-1}\right) [1]_s\,. 
\end{eqnarray*}
By Lemma \ref{sNum}, we have $u^{i,k}(s)= e_i \otimes s^{-1+(i+k+1)_2}\,\,[k]_s$. Then,
\begin{eqnarray*}
u^{i,k-1}(s) &+& u^{i,k+1}(s) - u^{i-1,k}(s) - u^{i+1,k}(s) = \\
&=&e_i \otimes s^{-1 + (i+k)_2}[k-1]_s + e_i \otimes s^{-1 + (i+k)_2} [k+1]_s + (-e_{i-1}-e_{i+1}) \otimes s^{-1+(i+k)_2} [k]_s\\
&=& 1\otimes s^{-1+(i+k)_2} \left( e_i \otimes \left( s+s^{-1}  \right) [k]_s + (-e_{i-1}-e_{i+1})\otimes [k]_s  \right)\\
&=& 1\otimes s^{-1+(i+k)_2} K e_i\otimes [k]_s\,.
\end{eqnarray*}
Notice that when $i=1$ or $i=r$, the above equation still holds since the matrix $A$ incorporates the boundary values. 
\end{proof}

We are now ready to prove the theorem stated at the beginning of the section. 

\begin{proof}[Proof of Theorem \ref{main2}]
We want to compute:
\begin{eqnarray}
\frac{1}{2} (\Lambda B) _{i,k}^{i',k'}&=& \frac{1}{2} \sum_{p,n} \Lambda_{i,k}^{p,n} B_{p,n}^{i',k'} \label{LambdaBValue}\\
&=&\sum_{p,n} (1\otimes s-1\otimes s^{-1}) D u^{i,k}(s) \cdot u^{p,n}(s) B_{p,n}^{i', k'} \,.\nonumber 
\end{eqnarray}
The goal is to show that the value of \eqref{LambdaBValue} is $1$ if $(i,k)= (i',k')$ and $0$ otherwise. 

Given $(i',k')$, the only non-zero terms in $B$ are given as follows:
\begin{eqnarray*}
B_{i'-1,k'}^{i',k'} = B_{i'+1,k'}^{i',k'} = (-1)^{i'+k'} &\text{ and }& B_{i',k'-1}^{i',k'} = B_{i',k'+1}^{i',k'} = (-1)^{i'+k'+1}\,,
\end{eqnarray*}
and $B_{n,p}^{i',k'} =0$ if $n\notin I$ or $p<0$ (see Figure~\ref{fig:Tsystem} on page~\pageref{fig:Tsystem}). Then,
\begin{eqnarray}
\frac{1}{2}(\Lambda B) _{i,k}^{i',k'}&=&(1\otimes s - 1\otimes s^{-1}) Du^{i,k}(s) \cdot \nonumber\\
&&\hspace{1in} (-1)^{i'+k'+1} \left( u^{i',k'-1}(s) + u^{i',k'+1}(s) -u^{i'-1,k'}(s) - u^{i'+1,k'}(s) \right) \nonumber\\
&=&(-1)^{i'+k'+1} (1\otimes s - 1\otimes s^{-1}) D e_i\otimes s^{-1+(i+k+1)_2} [k]_s \cdot  (1\otimes s^{-1 + (i'+k')_2} )K e_{i'}\otimes [k']_s\nonumber \\
&=&(-1)^{i'+k'+1} ( 1\otimes s^{(i+k+1)_2-(i'+k')_2} )(1\otimes s - 1\otimes s^{-1}) e_i\otimes [k]_s  \cdot     e_{i'} \otimes [k']_s \nonumber\\
&=&(-1)^{i'+k'+1} \left< e_i,e_{i'}\right> \left[ ( s^{-1} - s)(s^{(i'+k')_2 - (i+k+1)_2} [k]_{s^{-1}} [k']_s\right]_0 \nonumber\\
&=& (-1)^{i+k'+1} \delta_{ii'} \left[ s^\delta \frac{1}{s-s^{-1}} \left( s^{k'-k} - s^{-k-k'} - s^{k+k'} + s^{k'-k}\right)     \right]_0 \label{CT}\,,
\end{eqnarray}
where $\delta:=(i+k')_2 - (i+k+1)_2$. Notice that $\delta$ can only be $+1,0$ or $-1$. 

It is clear that if $i\neq i'$, the value of \eqref{CT} is zero. Suppose $k\neq k'$, and without loss of generality, let's assume $k'>k$. We want to show that the constant term part of \eqref{CT} vanishes in this case. Let $b:=k'-k > 0$ and $a:=-k'-k < 0$. Notice that $a$ and $b$ have the same parity. 

If $\delta=1$, the constant term expression of \eqref{CT} is as follows:
\begin{eqnarray}\label{kneqk'}
\begin{array}{lll}s &\times& \frac{1}{s(1-s^{-2})} \left( (s^{b} + s^{-b}) - (s^{a} + s^{-a}) \right) =\\
&=&  (s^b + s^{-b}) ( 1+s^{-2}+s^{-4} +\cdots ) - (s^a+s^{-a}) (1+s^{-2} + s^{-4} +\cdots ) \end{array}
\end{eqnarray}
Notice that $-b-2n, a-2m < 0$ for any $n,m\geq 0$, and therefore, $s^{-b}s^{-2n}$ and $s^{a}s^{-2m}$ are not constants for any $n,m\geq 0$. If $a$ and $b$ are odd, it is clear that there is no constant term in \eqref{kneqk'}. If $a$ and $b$ are even, there exists some $n>0$ such that $s^{b-2n} = 1$. Since $a<0$, there must exist some $m>0$ such that $s^{-a-2m}=1$. Therefore, the constant term in \eqref{kneqk'} vanishes.  

If $\delta=-1$, we write
\begin{eqnarray*}
s^{-1} \times \frac{1}{s^{-1}(s^2-1)} \left( (s^{b} + s^{-b}) - (s^{a} + s^{-a}) \right) \,,
\end{eqnarray*}
and use the same argument. 

If $\delta = 0$, the constant term part of \eqref{CT} can be written as:
\begin{eqnarray*}
\frac{1}{(1-s^{-2})} \left( (s^{b-1} + s^{-b-1}) - (s^{a-1} + s^{-a-1}) \right)\,,
\end{eqnarray*}
and we use the same argument. Therefore, the value of \eqref{CT} is always zero if $k\neq k'$. 

Suppose $k=k'$. If $(i+k)_2=0$, we have $\delta = -1$ and \eqref{CT} can be written as:
\begin{eqnarray*}
\frac{1}{2}(\Lambda B)_{i,k}^{i,k} &=& - \left[ s^{-1} \frac{1}{-s^{-1}(1-s^2)} \left( 2 - s^{2k} - s^{-2k} \right)  \right]_0 \\
&=& \left[ \left( 2 - s^{2k} - s^{-2k} \right) ( 1 + s^{2} + s^{4} + \cdots + s^{2k} + \cdots ) \right]_0= 1\,.
\end{eqnarray*}

If $(i+k)_2=1$, we have $\delta=1$ and \eqref{CT} can be written as:
\begin{eqnarray*}
\frac{1}{2}(\Lambda B)_{i,k}^{i,k} &=& \left[ s \frac{1}{s(1-s^{-2})} \left( 2 - s^{2k}  - s^{-2k} \right)  \right]_0  \\
&=& \left[ \left( 2 - s^{2k} - s^{-2k} \right) ( 1 + s^{-2} + s^{-4} + \cdots + s^{-2k} + \cdots ) \right]_0= 1\,.
\end{eqnarray*}
This concludes the proof.
\end{proof}

We are now ready to prove the final condition.

\begin{thm} \label{main3}
The quantum mutation is given by: 
\begin{equation*}
T_{k,l-1}^{(i)}*T_{k,l+1}^{(i)} = t^{\frac{1}{2} \Lambda_{i,k,l-1}^{i,k-1,l} + \frac{1}{2}\Lambda_{i,k,l-1}^{i,k+1,l} - \frac{1}{2} \Lambda_{i,k-1,l}^{i,k+1,l} } T_{k-1,l}^{(i)}*T_{k+1,l}^{(i)}+t^{\frac{1}{2} \Lambda_{i,k,l-1}^{i-1,k,l} + \frac{1}{2}\Lambda_{i,k,l-1}^{i+1,k,l} - \frac{1}{2} \Lambda_{i-1,k,l}^{i+1,k,l} } T_{k,l}^{(i-1)}*T_{k,l}^{(i+1)}.
\end{equation*}
\end{thm}

Theorem \ref{main3} is stated in terms of $T_{k,l}^{(i)}$ variables, while Nakajima's $t$-deformed $T$-system is written in terms of $\chi_{k,j}^{(i)}$ variables, which is achieved by a change of variables as described in Remark \ref{rmk:changeOfVar}. Also recall that $\Lambda$ is expressed in terms of the $\epsilon$ function (see Corollary \ref{gammaComm}). We now restate Theorem \ref{main3} in terms of $\chi_{k,j}^{(i)}$ variables and $\epsilon$ expressions for $\Lambda$. 
\begin{thm}\label{main3chi}
The following equation holds:
\begin{eqnarray}\label{cond3}
\begin{array}{c}\chi_{k,j}^{(i)}*\chi_{k,j+2}^{(i)} = t^{ \epsilon(\mathbf{Y}_{k,j}^{(i)}, \mathbf{Y}_{k-1,j+2}^{(i)}) + \epsilon(\mathbf{Y}_{k,j}^{(i)}, \mathbf{Y}_{k+1,j}^{(i)}) - \epsilon(\mathbf{Y}_{k+1,j+2}^{(i)}, \mathbf{Y}_{k-1,j}^{(i)})  } \chi_{k+1,j}^{(i)}*\chi_{k-1,j+2}^{(i)} \\
+t^{\epsilon( \mathbf{Y}_{k,j}^{(i)}, \mathbf{Y}_{k,j+1}^{(i-1)} ) + \epsilon( \mathbf{Y}_{k,j}^{(i)}, \mathbf{Y}_{k,j+1}^{(i+1)} ) -  \epsilon(\mathbf{Y}_{k,j+1}^{(i-1)}, \mathbf{Y}_{k,j+1}^{(i+1)}) } \chi_{k,j+1}^{(i-1)}*\chi_{k,j+1}^{(i+1)}\end{array}.
\end{eqnarray}
\end{thm}

\begin{proof}
Recall Nakajima's $t$-deformed $T$-system (see Theorem \ref{thm:Nakajima}):
\begin{eqnarray*}
 t^{-\epsilon(\mathbf{Y}_{k,j}^{(i)},\mathbf{Y}_{k,j+2}^{(i)})} \chi_{k,j}^{(i)}*\chi_{k,j+2}^{(i)} = t^{- \epsilon(\mathbf{Y}_{k+1,j+2}^{(i)}, \mathbf{Y}_{k-1,j}^{(i)})  } \chi_{k+1,j}^{(i)}*\chi_{k-1,j+2}^{(i)} +t^{-1- \epsilon(\mathbf{Y}_{k,j+1}^{(i-1)}, \mathbf{Y}_{k,j+1}^{(i+1)}) } \chi_{k,j+1}^{(i-1)}*\chi_{k,j+1}^{(i+1)}.
\end{eqnarray*}
We will show that the $t$-deformed $T$-system is equivalent to \eqref{cond3}. It suffices to show:
\begin{eqnarray}
\epsilon(\mathbf{Y}_{k,j}^{(i)}, \mathbf{Y}_{k-1, j+2}^{(i)}) + \epsilon(\mathbf{Y}_{k,j}^{(i)}, \mathbf{Y}_{k+1, j}^{(i)}) &=& \epsilon(\mathbf{Y}_{k,j}^{(i)}, \mathbf{Y}_{k,j+2}^{(i)}), \label{cond3eq1}\\
 \epsilon(\mathbf{Y}_{k,j}^{(i)}, \mathbf{Y}_{k,j+1}^{(i-1)}) + \epsilon(\mathbf{Y}_{k,j}^{(i)}, \mathbf{Y}_{k,j+1}^{(i+1)})&=&-1+\epsilon(\mathbf{Y}_{k,j}^{(i)}, \mathbf{Y}_{k,j+2}^{(i)}).\label{cond3eq2}
\end{eqnarray}
By an abuse of notation, let us denote $\mathbf{Y}_{k,j}^{(i)} := u(\mathbf{Y}_{k,j}^{(i)})(s) =  e_i \otimes ( s^j + s^{j+2} + \cdots + s^{j+2(k-1)})$ (see \eqref{KRdom} and Definition \ref{uGenFun}). 
\begin{eqnarray*}
\mathbf{Y}_{k-1,j+2}^{(i)} + \mathbf{Y}_{k+1,j}^{(i)}&=& e_i \otimes \left( (s^{j+2} + \cdots + s^{j+2+2(k-2)}) + (s^j + s^{j+2} +\cdots + s^{j+2k}) \right)\\
&=&e_i \otimes \left( (s^{j+2} + \cdots + s^{j+2k-2} + s^{j+2k}) + (s^j + s^{j+2} +\cdots + s^{j+2k-2}) \right)\\&=&\mathbf{Y}_{k, j+2}^{(i)} + \mathbf{Y}_{k,j}^{(i)}\,.
\end{eqnarray*}
Then, using Definition \ref{EpsilonForm}, we compute:
\begin{eqnarray*}
\epsilon(\mathbf{Y}_{k,j}^{(i)}, \mathbf{Y}_{k-1, j+2}^{(i)}) + \epsilon(\mathbf{Y}_{k,j}^{(i)}, \mathbf{Y}_{k+1, j}^{(i)}) &=& (1\otimes s - 1\otimes s^{-1})D \mathbf{Y}_{k,j}^{(i)} \cdot \left( \mathbf{Y}_{k-1,j+2}^{(i)} + \mathbf{Y}_{k+1,j}^{(i)} \right)\\
&=&(1\otimes s - 1\otimes s^{-1})D \mathbf{Y}_{k,j}^{(i)} \cdot \left( \mathbf{Y}_{k, j+2}^{(i)} + \mathbf{Y}_{k,j}^{(i)} \right)\\
&=&\epsilon(\mathbf{Y}_{k,j}^{(i)}, \mathbf{Y}_{k,j+2}^{(i)}) + \underbrace{\epsilon(\mathbf{Y}_{k,j}^{(i)}, \mathbf{Y}_{k,j}^{(i)})}_\text{0}\,,
\end{eqnarray*}
where $\epsilon(\mathbf{Y}_{k,j}^{(i)}, \mathbf{Y}_{k,j}^{(i)})=0$ due to anti-commutativity of $\epsilon$. This proves \eqref{cond3eq1}. 

Notice that
\begin{eqnarray*}
(1\otimes s + 1\otimes s^{-1}) \mathbf{Y}_{k,j+1}^{(i)} = \mathbf{Y}_{k,j+2}^{(i)} + \mathbf{Y}_{k,j}^{(i)}\,.
\end{eqnarray*}
Then,
\begin{eqnarray*}
\mathbf{Y}_{k,j+1}^{(i-1)} + \mathbf{Y}_{k,j+1}^{(i+1)} &=& \mathbf{Y}_{k,j+1}^{(i-1)} + \mathbf{Y}_{k,j+1}^{(i+1)} - (1\otimes s+1\otimes s^{-1})\mathbf{Y}_{k,j+1}^{(i)}  + (\mathbf{Y}_{k,j+2}^{(i)} + \mathbf{Y}_{k,j}^{(i)}) \\
&=& -(1\otimes s + 1\otimes s^{-1} +A\otimes 1) \mathbf{Y}_{k,j+1}^{(i)} +\mathbf{Y}_{k,j+2}^{(i)} + \mathbf{Y}_{k,j}^{(i)} \\
&=& -K \mathbf{Y}_{k,j+1}^{(\alpha)} + \mathbf{Y}_{k,j}^{(\alpha)} + \mathbf{Y}_{k, j+2}^{(\alpha)}\,.
\end{eqnarray*}
Next, 
\begin{eqnarray*}
\epsilon(\mathbf{Y}_{k,j}^{(i)}, \mathbf{Y}_{k,j+1}^{(i-1)}) + \epsilon(\mathbf{Y}_{k,j}^{(i)},\mathbf{Y}_{k,j+1}^{(i+1)}) &=& (1\otimes s - 1\otimes s^{-1})D(\mathbf{Y}_{k,j}^{(i)}) \cdot \left( \mathbf{Y}_{k,j+1}^{(i-1)} + \mathbf{Y}_{k,j+1}^{(i+1)}\right)\\
&=&(1\otimes s - 1\otimes s^{-1}) D(\mathbf{Y}_{k,j}^{(i)}) \cdot \left( -K(\mathbf{Y}_{k,j+1}^{(i)}) + \mathbf{Y}_{k,j}^{(i)} + \mathbf{Y}_{k, j+2}^{(i)}\right)\\
&=&-(1\otimes s - 1\otimes s^{-1}) \mathbf{Y}_{k,j}^{(i)} \cdot \mathbf{Y}_{k,j+1}^{(i)} + \underbrace{\epsilon(\mathbf{Y}_{k,j}^{(i)},\mathbf{Y}_{k,j}^{(i)})}_\text{0}
+ \epsilon(\mathbf{Y}_{k,j}^{(i)}, \mathbf{Y}_{k,j+2}^{(i)}) \,,
\end{eqnarray*}
where
\begin{eqnarray*}
(1\otimes s - 1\otimes s^{-1}) \mathbf{Y}_{k,j}^{(i)} \cdot \mathbf{Y}_{k,j+1}^{(i)} &=& \left[ (s^{-1}-s)( s^{-j} + \cdots +s^{-j-2k+2}) ( s^{j+1} + \cdots + s^{j+2k-1}) \right]_0\\
&=& \left[  (s^{-1}-s) s^{-j-k+1} [k]_s s^{j+k} [k]_s  \right]_0\\
&=& \left[ -s\frac{1}{s-s^{-1}} (s^{2k} - 2 + s^{-2k})  \right]_0=1
\end{eqnarray*}
This concludes the proof of \eqref{cond3eq2}. 
\end{proof}

\subsection{Explicit proof of Condition II for type $A_1$}
The commutation matrix $\Lambda$ can be computed explicitly and Theorem \ref{main2} can be verified through direct computation. We will work out the explicit description of $\Lambda$ and the direct verification of Theorem \ref{main2} for the case of type $A_1$ now. 

\begin{defn}
Given a matrix $M = (m_{i,j})_{i\in I, j\in J}$ for some index sets $I$ and $J$ (possibly infinite), we can write $M$ as a generating series $M(z_1,z_2) = \sum_{i\in I, j\in J} m_{i,j}z_1^iz_2^j$, where $z_1,z_2$ are indeterminates keeping track of the indices of the matrix.   
\end{defn}

The quiver associated to the $T$-system of type $A_1$ is as follows:

\begin{figure}[h]
\begin{center}
\begin{tikzpicture}

	\tikzstyle{black} = [circle, minimum width=5pt, fill, inner sep=0pt]
	\tikzstyle{white} = [circle, minimum width=5pt,draw,inner sep=0pt]

\node at (0,-1) [above] {$1$};
\node at (1,-1) [above] {$2$};
\node at (2,-1) [above] {$3$};
\node at (3,-1) [above] {$4$};
\node at (4,-1) [above] {$5$};

\node[black] (10) at (0,-1) {};
\node[white] (11) at (1,-1) {};
\node[black] (12) at (2,-1) {};
\node[white] (13) at (3,-1) {};
\node[black] (14) at (4,-1) {};

\node at (5,-1) {$\cdots$};

\path[-{latex'},thick]

(10) edge (11)
(12) edge (11)
(12) edge (13)
(14) edge (13)
(6,-1) edge node[above]{$k$} (7,-1)
;

\end{tikzpicture}
\end{center}
\caption{The quiver $\Gamma_T$, ($k\in \mathbb{Z}_+$)}
\label{fig:TsystemA1}
\end{figure}
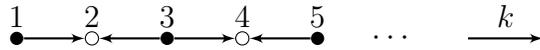

The signed adjacency matrix of the quiver $\Gamma_T$ in Figure~\ref{fig:TsystemA1} is as follows:

\begin{equation*}
B=\left(
\begin{array}{c|cccccccccc}
{\red k}&{\red 1}&{\red 2}&{\red 3}&{\red 4}&{\red 5}&{\red 6}&{\red 7}&{\red 8}&{\red 9}&{\red \cdots}\\
\hline
{\red 1}		&0&1&0&0&0&0&0&0&0&\cdots\\
{\red 2}		&-1&0&-1&0&0&0&0&0&0&\cdots\\
{\red 3}		&0&1&0&1&0&0&0&0&0&\cdots\\
{\red 4}		&0&0&-1&0&-1&0&0&0&0&\cdots\\
{\red 5}		&0&0&0&1&0&1&0&0&0&\cdots\\
{\red 6}		&0&0&0&0&-1&0&-1&0&0&\cdots\\
{\red 7}		&0&0&0&0&0&1&0&1&0&\cdots\\
{\red 8}		&0&0&0&0&0&0&-1&0&-1&\cdots\\
{\red 9}		&0&0&0&0&0&0&0&1&0&\cdots\\
{\red \vdots}&\vdots&\vdots&\vdots&\vdots&\vdots&\vdots&\vdots&\vdots&\vdots&\ddots\\
\end{array}\right)
\end{equation*}

which can be written as:
\begin{eqnarray*}
B(z_1,z_2)&=&z_1z_2\frac{z_2-z_1}{1+z_1z_2}.
\end{eqnarray*}

Recall that to each vertex $k$, we associate the KR-module $W_{k, l = (k)_2-k}$ (see Equation \eqref{eq:FundClustT}), and the commutation matrix $\Lambda$ is given by:
\begin{eqnarray*}
\Lambda(k,k') = 2 \epsilon(\mathbf{Y}_{k, (k)_2-k} , \mathbf{Y}_{k', (k')_2-k'}), 
\end{eqnarray*}
where $\mathbf{Y}_{k,j} = Y_{j} Y_{j+2}\cdots Y_{j+2k-2}$ (see Equation \eqref{KRdom}), where we dropped the index $i$ in $Y_{i,j}$ since $i$ can take exactly one value, and $\epsilon$ is from Definition \ref{def:Nakajima}. More precisely, we have: 
\begin{eqnarray*}
\epsilon(\mathbf{Y}_{k,(k)_2-k}, \mathbf{Y}_{k', (k')_2-k'}) &=& - \sum_{j=0}^{k-1} \tilde{u}_{(k)_2 - k + 2j -1}( \mathbf{Y}_{k', (k')_2-k'} ) + \sum_{j = 0}^{k'-1} \tilde{u}_{(k')_2 - k' + 2j - 1}(\mathbf{Y}_{k,(k)_2-k}),
\end{eqnarray*}
where $\tilde{u}_{j}(m) \in \mathbb{R}$ ($j\in \mathbb{Z}$) is the unique solution of the system:
\begin{equation*} 
u_{j}(m) = \tilde{u}_{j-1}(m) + \tilde{u}_{j+1}(m)\,\, ,
\end{equation*}
such that $\tilde{u}_{j}(m)=0$ for $j$ sufficiently small (see Definition \ref{uinverse}). Again, we dropped the dependance on $i$ and set $u_{i,j}(m) = u_j(m)$.

\begin{example}
Consider $m=\mathbf{Y}_{1,0} = Y_{0}$. Then $\tilde{u}_{1}(m) = 1$ and $\tilde{u}_{j}(m) = 0$ for all $j\geq 1$. 
\begin{eqnarray*}
&\vdots&\\
u_4(m) = 0 &=& \tilde{u}_{5}(m) + \tilde{u}_{7}(m) = 0 + 0\\
u_2(m) = 0 &=& \tilde{u}_{1}(m) + \tilde{u}_{3}(m) = 1 + 0\\
u_0(m) = 1 &=& \tilde{u}_{-1}(m) + \tilde{u}_{1}(m) = 0 + 1\\
u_{-2}(m) = 0 &=& \tilde{u}_{-3}(m) + \tilde{u}_{-1}(m) =  0 + 0\\
&\vdots&
\end{eqnarray*}
Here, we must have $\tilde{u}_{-1}(m)=0$. Otherwise, $\tilde{u}_{-2j-1}(u)=0$ for all $j\geq 0$, which contradicts the condition $\tilde{u}_j(m)=0$ for $j$ sufficiently small. 
\end{example}

\begin{example}
When $m=\mathbf{Y}_{2,-2} = Y_{-2}Y_0$, we can compute $\tilde{u}_{-1}(m) = 1$ and $\tilde{u}_j(m) = 0$ for all $j\neq -1$. Then,
\begin{eqnarray*}
\epsilon(\mathbf{Y}_{1,0}, \mathbf{Y}_{2,-2}) &=&  - \tilde{u}_{-1}( Y_{-2}Y_0 ) + (\tilde{u}_{-3}(Y_0) + \tilde{u}_{-1}(Y_0) ) = - 1 + 0 = -1.\\
\end{eqnarray*}
\end{example}

\begin{example}
When $m=\mathbf{Y}_{4,-4} = Y_{-4}Y_{-2}Y_0Y_2$, we can compute $\tilde{u}_{-3}(m) = \tilde{u}_{1}(m) = 1$ and $\tilde{u}_j(m) = 0$ for all $j\neq -3,1$. Then,
\begin{eqnarray*}
\epsilon(\mathbf{Y}_{1,0}, \mathbf{Y}_{4,-4}) &=&  - \tilde{u}_{-1}( Y_{-4}Y_{-2}Y_0Y_2 ) + (\tilde{u}_{-5}(Y_0) + \tilde{u}_{-3}(Y_0) + \tilde{u}_{-1}(Y_0) + \tilde{u}_{1}(Y_0) ) = - 0 + 1 = 1.\\
\end{eqnarray*}
\end{example}
Similarly, the commutation matrix $\Lambda$ can be computed and is given by:
\begin{equation*}
\Lambda=\left(
\begin{array}{c|cccccccccc}
{\red k}&{\red 1}&{\red 2}&{\red 3}&{\red 4}&{\red 5}&{\red 6}&{\red 7}&{\red 8}&{\red 9}&{\red \cdots}\\
\hline
{\red 1}		&0&-1&0&1&0&-1&0&1&0&\cdots\\
{\red 2}		&1&0&0&0&0&0&0&0&0&\cdots\\
{\red 3}		&0&0&0&-1&0&1&0&-1&0&\cdots\\
{\red 4}		&-1&0&1&0&0&0&0&0&0&\cdots\\
{\red 5}		&0&0&0&0&0&-1&0&1&0&\cdots\\
{\red 6}		&1&0&-1&0&1&0&0&0&0&\cdots\\
{\red 7}		&0&0&0&0&0&0&0&-1&0&\cdots\\
{\red 8}		&-1&0&1&0&-1&0&1&0&0&\cdots\\
{\red 9}		&0&0&0&0&0&0&0&0&0&\cdots\\
{\red \vdots}&\vdots&\vdots&\vdots&\vdots&\vdots&\vdots&\vdots&\vdots&\vdots&\ddots\\
\end{array}\right)
\end{equation*}
which can be written as:
\begin{eqnarray*}
\Lambda(z_1,z_2)&=&\frac{z_1z_2(z_1-z_2)}{(1+z_1z_2)(1+z_1^2)(1+z_2^2)}.
\end{eqnarray*}
Then,
\begin{eqnarray*}
\frac{1}{2}\Lambda B (z_1,z_2) &=& {\rm{Res}}_w \left( \frac{1}{w} \Lambda(z_1,w) B(\frac{1}{w}, z_2)  \right) \\
&=&{\rm{Res}}_w \left( \frac{1}{w} \frac{z_1w(z_1-w)}{(1+wz_1)(1+z_1^2)(1+w^2)}\cdot \frac{1}{w}z_2 \frac{z_2-\frac{1}{w}}{1+\frac{z_2}{w}} \right)\\
&=& {\rm{Res}}_w \left(  \frac{1}{w} \frac{z_1z_2(z_1-w)(wz_2-1)}{(1+wz_1)(1+z_1^2)(1+w^2)(w+z_2)} \right)\\
&=&-\frac{z_1^2z_2}{(1+z_1^2)z_2}+\frac{z_1z_2(z_1+z_2)(-z_2^2-1)}{(-z_2)(1-z_1z_2)(1+z_1^2)(1+z_2^2)}\\
&=&-\frac{z_1^2}{1+z_1^2}+\frac{z_1(z_1+z_2)}{(1-z_1z_2)(1+z_1^2)}\\
&=&\frac{-z_1^2+z_1^3z_2+z_1^2+z_1z_2}{(1-z_1z_2)(1+z_1^2)}=\frac{z_1z_2}{1-z_1z_2} = I(z_1,z_2)
\end{eqnarray*}

\section{Evolution in $k$-direction}

We showed that Nakajima's deformed $T$-system forms a quantum cluster algebra with evolution in $l$-direction in $T_{k,l}^{(i)}$ variables (equivalently in $j$-direction in $\chi_{k,j}^{(i)}$ variables). We now show that the same deformed $T$-system is not a quantum cluster algebra with evolution in $k$-direction. In particular, this shows that the quantum $T$-system is not compatible with the quantum $Q$-system considered in \cite{di2011non}. 

By re-writing Nakajima's $t$-deformed $T$-system of Theorem \ref{thm:Nakajima} so that the evolution is in $k$-direction, we obtain:
\begin{eqnarray*}
\chi_{k+1,j}^{(i)} *_\gamma \chi_{k-1,j+2}^{(i)} = \chi_{k,j}^{(i)} *_\gamma \chi_{k,j+2}^{(i)} - t^{-1} \chi_{k,j+1}^{(i-1)} *_\gamma \chi_{k,j+1}^{(i+1)} \,.
\end{eqnarray*}
Notice that the negative sign on the right-hand side is not compatible with cluster algebra interpretation, where the right-hand side expression must have exactly $2$ positive contributions. However, it is possible to renormalize $\chi_{k,j}^{(i)}$'s such that the negative sign becomes positive \cite{di2009positivity}. We call the resulting variables $\widehat{\chi}_{k,j}^{(i)}$. The ensuing $T$-system of type $A_1$ is as follows:
\begin{eqnarray} \label{A1T}
\widehat{\chi}_{k+1,j}*_\gamma \widehat{\chi}_{k-1,j+2} =  \widehat{\chi}_{k,j}*_\gamma \widehat{\chi}_{k,j+2} + t^{-1} ,
\end{eqnarray}
where we dropped the parameter $i$ since it can have only one value. 

The variables on the right-hand side of \eqref{A1T} must belong to the same cluster, and therefore, must $t$-commute if \eqref{A1T} does form a quantum mutation. We will give a simple counter example to this condition, which shows that \eqref{A1T} is not a quantum cluster algebra.  

\begin{example}
Let $\g = \mathfrak{sl}_2$. We consider $\widehat{\chi}_{1,0}$ and $\widehat{\chi}_{1,2}$, which are both on the right-hand side of \eqref{A1T}. The values of these variables are given as follows:

\begin{center}
\begin{tikzpicture}

\node at (0,0) {
\begin{tikzpicture}
\node at (-3,2) {$\widehat{\chi}_{1,0} = \chi_{q,t}(W_{1,0}^{(1)})$};
\node at (-1,2) {$=$};
\node (1') at (0,2) {$Y_{1,0}$};
\node (1) at (0,.5) {
\begin{ytableau}
\none[{\color{blue}\scriptstyle{0}}]&1
\end{ytableau}
};

\node (2') at (2,2) { $Y_{1,2}^{-1}$};

\node (2) at (2,0.5) {
\begin{ytableau}
\none[{\color{blue}\scriptstyle{0}}]&2
\end{ytableau}
};

\node at (1,2) {$+$};
\path[->,thick]
(1') edge[dotted] (1)
(2') edge[dotted] (2);
\end{tikzpicture}
};

\node at (4,.5) {;};

\node at (8,0) {
\begin{tikzpicture}
\node at (-3,2) {$\widehat{\chi}_{1,2} = \chi_{q,t}(W_{1,2}^{(1)})$};
\node at (-1,2) {$=$};
\node (1') at (0,2) {$Y_{1,2}$};
\node (1) at (0,.5) {
\begin{ytableau}
\none[{\color{blue}\scriptstyle{-1}}]&1
\end{ytableau}
};

\node (2') at (2,2) { $Y_{1,4}^{-1}$};

\node (2) at (2,0.5) {
\begin{ytableau}
\none[{\color{blue}\scriptstyle{-1}}]&2
\end{ytableau}
};

\node at (1,2) {$+$};
\path[->,thick]
(1') edge[dotted] (1)
(2') edge[dotted] (2);
\end{tikzpicture}
};

\end{tikzpicture}
\end{center}
Their twisted product can be computed using Theorem \ref{gamma}, and is given by:

\begin{center}
\begin{tikzpicture}

\node at (-3,2) {$\widehat{\chi}_{1,0} *_\gamma \widehat{\chi}_{1,2}$};
\node at (-1.5,2) {$=$};

\node (1') at (0,2) {$Y_{1,0}Y_{1,2}$};

\node (1) at (0,0) {
\begin{ytableau}
\none[{\color{blue}\scriptstyle{-1}}]&\none&1\\
\none[{\color{blue}\scriptstyle{0}}]&1\\
\end{ytableau}
};

\node (2') at (3,2) { $Y_{1,0}Y_{1,4}^{-1}$};

\node (2) at (3,0) {
\begin{ytableau}
\none[{\color{blue}\scriptstyle{-1}}]&\none&2\\
\none[{\color{blue}\scriptstyle{0}}]&1\\
\end{ytableau}
};

\node (3') at (6,2) { $t^{-1} \,Y_{1,2}^{-1}Y_{1,2}$};

\node (3) at (6,0) {
\begin{ytableau}
\none[{\color{blue}\scriptstyle{-1}}]&\none&1\\
\none[{\color{blue}\scriptstyle{0}}]&2\\
\end{ytableau}
};

\node (4') at (9,2) { $Y_{1,2}^{-1}Y_{1,4}^{-1}$};

\node (4) at (9,0) {
\begin{ytableau}
\none[{\color{blue}\scriptstyle{-1}}]&\none&2\\
\none[{\color{blue}\scriptstyle{0}}]&2\\
\end{ytableau}
};

\node at (1.5,2) {$+$};
\node at (4.5,2) {$+$};
\node at (7.5,2) {$+$};

\path[->,thick]
(1') edge[dotted] (1)
(2') edge[dotted] (2)
(3') edge[dotted] (3)
(4') edge[dotted] (4)
;

\end{tikzpicture}
\end{center}
Notice that the third term on the right-hand side forms a pair of column tableaux of regular type with non-zero boundary term in $\gamma$. On the other hand, 
\begin{center}
\begin{tikzpicture}

\node at (-3,2) {$\widehat{\chi}_{1,2} *_\gamma \widehat{\chi}_{1,0}$};
\node at (-1.5,2) {$=$};

\node (1') at (0,2) {$Y_{1,2}Y_{1,0}$};

\node (1) at (0,0) {
\begin{ytableau}
\none[{\color{blue}\scriptstyle{-1}}]&1\\
\none[{\color{blue}\scriptstyle{0}}]&\none&1\\
\end{ytableau}
};

\node (2') at (3,2) { $Y_{1,4}^{-1}Y_{1,0}$};

\node (2) at (3,0) {
\begin{ytableau}
\none[{\color{blue}\scriptstyle{-1}}]&2\\
\none[{\color{blue}\scriptstyle{0}}]&\none&1\\
\end{ytableau}
};

\node (3') at (6,2) {$t \,Y_{1,2}Y_{1,2}^{-1}$};

\node (3) at (6,0) {
\begin{ytableau}
\none[{\color{blue}\scriptstyle{-1}}]&1\\
\none[{\color{blue}\scriptstyle{0}}]&\none&2\\
\end{ytableau}
};

\node (4') at (9,2) {$Y_{1,4}^{-1}Y_{1,2}^{-1}$};

\node (4) at (9,0) {
\begin{ytableau}
\none[{\color{blue}\scriptstyle{-1}}]&2\\
\none[{\color{blue}\scriptstyle{0}}]&\none&2\\
\end{ytableau}
};

\node at (1.5,2) {$+$};
\node at (4.5,2) {$+$};
\node at (7.5,2) {$+$};

\path[->,thick]
(1') edge[dotted] (1)
(2') edge[dotted] (2)
(3') edge[dotted] (3)
(4') edge[dotted] (4)
;

\end{tikzpicture}
\end{center}

We see that $\widehat{\chi}_{1,0}$ and $\widehat{\chi}_{1,2}$ do not $t$-commute. 
\end{example}

\section{Conclusion/Discussion}
The Nakajima $(q,t)$-characters of KR-modules satisfy a deformed $T$-system, introduced in \cite{nakajima2003t}, which is a $t$-deformed discrete dynamical system with $3$ independent parameters: $i,k,j$. In this thesis, we showed that this $t$-deformed $T$-system forms a quantum mutation in a quantization of the $T$-system cluster algebra only when the direction of mutation is in the $l$-parameter. 

This result pertains to type $A$ only. It is an open question whether the same holds for other types, and in particular to type $D$. There are noticeable differences between types $A$ and $D$. For example, unlike the case in type $A$, the KR-modules of type $D$ are reducible as $U_q(\g)$-modules. Also, the $(q,t)$-characters of the KR-modules are not identical to their $q$-characters. 

However, despite these differences, the proofs of Conditions II and III are applicable to type $D$ with minimal adjustments. The hard part is the proof of Condition I, the $t$-commutativity of the fundamental cluster variables. The proof of Condition I is entirely combinatorial and requires the knowledge of the $(q,t)$-characters of all the KR-modules in the fundamental cluster. Nakajima's tableaux-sum notation exists for type $D$ as well (see \cite{nakajima2002t}). However, it is combinatorially different from the type $A$ case, and, therefore, all the combinatorial structures introduced as part of the proof of Condition I will need to be translated to the combinatorics of the type $D$.

\bibliographystyle{plain}
\bibliography{reference}

\end{document}